\documentclass{article}
\usepackage[utf8]{inputenc}

\usepackage[a4paper,top=3cm,bottom=2cm,left=3cm,right=3cm,marginparwidth=1.75cm]{geometry}

\usepackage[T1]{fontenc}
\usepackage[english]{babel}
\usepackage{tikz-cd}

\usepackage{enumitem} 
\usepackage{scalerel,stackengine}
\usepackage{tikz-cd}
\usepackage{bm}
\usepackage{url}
\usepackage{amsmath}
\usepackage{amssymb}
\usepackage{amsthm}
\usepackage{mathtools}
\usepackage{amsmath,calligra,mathrsfs}
\usepackage{yhmath}

\usepackage{graphicx}
\usepackage[colorinlistoftodos]{todonotes}
\usepackage[colorlinks=true, allcolors=blue]{hyperref}

\DeclareMathOperator{\Deriv}{\mathscr{D}\text{\kern -3pt {\calligra\large er}}\,}
\DeclareMathOperator{\Endo}{\mathscr{E}\text{\kern -3pt {\calligra\large nd}}\,}
\DeclareMathOperator{\Hom}{Hom \text{\kern -3pt }\,}
\DeclareMathOperator{\Der}{Der \text{\kern -3pt }\,}
\DeclareMathOperator{\End}{End \text{\kern -3pt }\,}
\DeclareMathOperator{\Coh}{Coh \text{\kern -2pt }\,}
\DeclareMathOperator{\Gal}{Gal \text{\kern -2pt }\,}
\DeclareMathOperator{\dev}{dev \text{\kern -2pt }\,}
\DeclareMathOperator{\Kdim}{Kdim \text{\kern -2pt }\,}
\DeclareMathOperator{\GKdim}{GKdim \text{\kern -2pt }\,}
\DeclareMathOperator{\Supp}{Supp \text{\kern -2pt }\,}
\DeclareMathOperator{\Ext}{Ext \text{\kern -2pt }\,}
\DeclareMathOperator{\QCoh}{QCoh \text{\kern -2pt }\,}
\DeclareMathOperator{\WCoh}{WCoh \text{\kern -2pt }\,}
\DeclareMathOperator{\Spec}{Spec \text{\kern -2pt }\,}
\DeclareMathOperator{\Frac}{Frac \text{\kern -2pt }\,}
\DeclareMathOperator{\res}{res \text{\kern -1pt }\,}
\DeclareMathOperator{\act}{act \text{\kern -1pt }\,}
\DeclareMathOperator{\Loc}{Loc \text{\kern -2pt }\,}
\DeclareMathOperator{\Lie}{Lie \text{\kern -2pt }\,}
\DeclareMathOperator{\loc}{loc \text{\kern -2pt }\,}
\DeclareMathOperator{\ob}{ob \text{\kern -2pt }\,}
\DeclareMathOperator{\op}{op \text{\kern -2pt }\,}
\DeclareMathOperator{\ad}{ad \text{\kern -2pt }\,}
\DeclareMathOperator{\fg}{fg \text{\kern -2pt }\,}
\DeclareMathOperator{\Mod}{Mod \text{\kern -2pt }\,}
\DeclareMathOperator{\Ann}{Ann \text{\kern -2pt }\,}
\DeclareMathOperator{\Sym}{Sym \text{\kern -2pt }\,}
\DeclareMathOperator{\Dif}{Dif \text{\kern -2pt }\,}
\DeclareMathOperator{\co}{co \text{\kern -2pt }\,}
\DeclareMathOperator{\Ad}{Ad \text{\kern -2pt }\,}
\DeclareMathOperator{\gr}{gr \text{\kern -2pt }\,}
\DeclareMathOperator{\Gr}{Gr \text{\kern -2pt }\,}
\DeclareMathOperator{\im}{im \text{\kern -2pt }\,}
\DeclareMathOperator{\Image}{Image \text{\kern -2pt }\,}
\DeclareMathOperator{\coim}{coim \text{\kern -2pt }\,}
\DeclareMathOperator{\id}{id \text{\kern -2pt }\,}
\DeclareMathOperator{\bimod}{bimod \text{\kern -2pt }\,}
\DeclareMathOperator{\height}{ht \text{\kern -2pt }\,}
\DeclareMathOperator{\tensor}{tensor \text{\kern -2pt }\,}
\DeclareMathOperator{\coker}{coker \text{\kern -2pt }\,}
\stackMath
\newcommand\reallywidehat[1]{%
\savestack{\tmpbox}{\stretchto{%
  \scaleto{%
    \scalerel*[\widthof{\ensuremath{#1}}]{\kern-.6pt\bigwedge\kern-.6pt}%
    {\rule[-\textheight/2]{1ex}{\textheight}}
  }{\textheight}%
}{0.5ex}}%
\stackon[1pt]{#1}{\tmpbox}%
}
\newcommand{\invlim}{\lim\limits_{\longleftarrow}}

\newcommand{\U}[1]{U(\fr{#1})}

\newcommand{\hugnK}{\widehat{\U{g}_{n,K}}}
\newcommand{\hugnKl}{\reallywidehat{\U{g}_{n,K}^{\lambda}}}

\newcommand{\huggnK}{\widehat{\U{g \times g}_{n,K}}}
\newcommand{\hugn}{\widehat{\U{g}_n}}
\newcommand{\huggn}{\widehat{\U{g \times g}_n}}

\newcommand{\fr}[1]{\mathfrak{{#1}}}

\newcommand{\uset}[1]{\underset{#1}{\otimes}{}}

\newcommand{\hugnL}{\widehat{\U{g'}_{en,L}}}

\newcommand{\hugnKm}{\reallywidehat{\U{g}_{n,K}^{\mu}}}
\newcommand{\hugnKlst}{\reallywidehat{\U{g}_{n,K}^{\lambda^*}}}

\newcommand{\huggnKlstl}{\reallywidehat{\U{g \times g}_{n,K}^{\lambda^*,\lambda}}}
\newcommand{\huggnKlm}{\reallywidehat{\U{g \times g}_{n,K}^{\lambda,\mu}}}

\newtheorem{theorem}{Theorem}[section]
\newtheorem{lemma}[theorem]{Lemma}
\newtheorem{proposition}[theorem]{Proposition}
\newtheorem{corollary}[theorem]{Corollary}
\newtheorem{question}[theorem]{Question}
\newtheorem{definition}[theorem]{Definition}

\newtheorem{remark}[theorem]{Remark}

\newtheorem{assumption}[theorem]{Assumption}

\title{Primitive ideals in the affinoid enveloping algebras of semisimple Lie algebras}
\author{Ioan Stanciu}

\begin{document}
\maketitle

\begin{abstract}
For a semisimple Lie algebra defined over a discrete valuation ring with field of fractions $K$, we prove that any primitive ideal with rational central character in the affinoid enveloping algebra, $\widehat{U(\mathfrak{g})_{K}},$ is the annihilator of an affinoid highest weight module.
\end{abstract}

\section{Introduction}

This is the final paper of the series \cite{Sta1}, \cite{Sta2} that aims to answer Question A from \cite{Munster} regarding the classification of primitive ideals of the affinoid enveloping algebra of a semisimple Lie algebra defined over a discrete valuation ring. We begin by reviewing the classical results.

\subsection{Classical Duflo's Theorem}
Let $K$ be a field of characteristic $0$ and $G$ be a connected, smooth, split-semisimple, affine algebraic group over $K$ with Lie algebra $\mathfrak{g}_K= \Lie(G)$. Fix $\mathfrak{g}_K:= \mathfrak{n}_K^{-} \oplus \mathfrak{h}_K \oplus \mathfrak{n}_K^+$ a Cartan decomposition and let $\mathfrak{b}_K=\mathfrak{h}_K \oplus \mathfrak{n}_K^+$. In a seminal paper \cite{BGG}, the authors define the BGG category $\mathcal{O}$ of representations for the algebra $U(\mathfrak{g}_K)$. The building blocks are given by Verma modules $M(\lambda)=U(\mathfrak{g}_K) \uset{U(\mathfrak{b}_K)}  K_{\lambda}$ for $\lambda \in \mathfrak{h}_K^*.$ These are highest weight modules with unique maximal submodule $N(\lambda)$ and unique simple quotient $L(\lambda)$. Moreover, this category $\mathcal{O}$ is Artinian and the set of simple objects is characterised exactly by $L(\lambda)$ for $\lambda \in \mathfrak{h}_K^*$. An excellent exposition of category $\mathcal{O}$ can be found in \cite{Hu1}. The importance of category $\mathcal{O}$ in the representation theory of the ring $U(\mathfrak{g}_K)$ can be seen in the following theorem:

\begin{theorem} \cite[Theorem 4.3]{Du}
\label{ClassicalDuflo}

Let $I$ be a primitive/prime ideal with $K$-rational central character. Then 
$$I=\Ann(L(\lambda)) \text{ for some } \lambda \in \mathfrak{h}_K^*. $$

\end{theorem}

Duflo's original statement requires the ground field to be $\mathbb{C}$. In their paper \cite{BG}, Bernstein and Gelfand extend this result to all algebraically closed fields of characteristic $0$ which in turn can be extended to the generality stated by a base change argument. Another purely algebraic proof of Duflo's theorem can be found \cite{Jos} and for a categorical proof see \cite{Gin}. One should note that if $K$ is algebraically closed, all primitive ideals have $K$-rational central character, so the theorem gives a full classification of the primitive spectrum.

\subsection{Classical Beilinson-Bernstein Localisation Theorem}

Fix $B$ a Borel subgroup of $G$ and let $X=G/B$ denote the flag variety of $G$. For a $K$-linear map $\lambda: \mathfrak{h}_K \to K$, we let $\chi_\lambda:Z(U(\mathfrak{g}_K)) \to K$ denote the corresponding central character. Furthermore, let $U(\mathfrak{g}_K)^{\lambda}$ be the quotient of $U(\mathfrak{g}_K)$ by the two-sided ideal generated by $\ker \chi_\lambda$, and let $\mathcal{D}^{\lambda}_{X}$ denote the sheaf of $\lambda$-twisted differential operators on $X$ as defined in \cite{BB}.

Let $L$ be a closed subgroup of $G$. In $\cite{BB}$, the authors define the notion of $L$-equivariant $\mathfrak{g}$ and $\mathcal{D}$-modules. A more detailed definition can be found in \cite[Section 11.5]{HTT}. 

Let $\mathbf{\Phi}$ be the root system of $\mathfrak{g}$ and $\mathbf{\Phi}^{+}$ the subset of positive roots. For a root $\alpha$, we use $\alpha^{\vee}$ to denote the corresponding coroot. We say that a weight $\lambda \in \mathfrak{h}_K^*$ is dominant if $(\lambda+\rho)(\alpha^{\vee}) \in \mathbb{Z}^{\leq -1}$ for all $\alpha \in \mathbf{\Phi}^+$. We say that $\lambda$ is regular if $(\lambda+\rho)(\alpha^{\vee}) \neq 0$ for all $\lambda \in \mathbf{\Phi}^{+}$.

\begin{theorem}[Equivariant Beilinson-Bernstein localisation, \cite{BB}] 
\label{classicalequivariantBeilinson-Bernstein}

Let $\lambda:\mathfrak{h}_K \to K$ be a $K$-linear dominant weight. Consider the functors:

\begin{equation}
\begin{split}
&\Loc: \Mod_{\fg}(U(\mathfrak{g}_K)^{\lambda},L) \to \Coh(\mathcal{D}^\lambda_{X},L), \quad \Loc(M):= \mathcal{D}^\lambda_{X} \uset{U(\mathfrak{g}_K)^{\lambda}} M,  \\
&\Gamma:  \Coh(\mathcal{D}^\lambda_{X},L) \to \Mod_{\fg}(U(\mathfrak{g}_K)^{\lambda},L), \qquad \Gamma(\mathcal{M}):=\mathcal{M}(X).
\end{split}
\end{equation}

Then $\Loc$ and $\Gamma$ induce quasi-inverse equivalences of categories between\\
$\Mod_{\fg}(U(\mathfrak{g}_K)^{\lambda},L)$ and the quotient category $\Coh(\mathcal{D}^\lambda_{X},L)/\ker \Gamma$. In case $\lambda$ is also regular, we have $\ker \Gamma=0$.

\end{theorem}
\subsection{Classical approach to Duflo's theorem}
In his paper \cite{Jos}, Joseph proves Duflo's theorem by proving:

\begin{proposition}
\label{classicalJosfun}
Let $\lambda:\mathfrak{h}_K  \to K$ be a $K$-linear dominant weight. Consider the function $J$:\{two-sided ideals in $U(\mathfrak{g}_K)^{\lambda} \} \to$ \{submodules of $M(\lambda)$\} defined by
             $$J(I):= I M(\lambda).$$ 
             
Then $J$ is injective.             
\end{proposition}

This approach to prove Duflo's theorem was suggested by Dixmier, see \cite[Problem 30]{Dix}. In \cite{BoBr}, Borho and Brylinski prove Proposition \ref{classicalJosfun} using the Beilinson-Bernstein localisation in the case $\lambda=0$. The key step is proving that the category of coherent $G$-equivariant $\mathcal{D}_{X \times X}$-modules is equivalent to the category of coherent $B$-equivariant $\mathcal{D}_{X}$-modules. In \cite{Sta2}, we obtained a geometric proof of this result for an arbitrary dominant $\lambda$.

\subsection{Affinoid enveloping algebras}

Let $R$ be a mixed characteristic $(0,p)$ complete discrete valuation ring with uniformiser $\pi$, field of fractions $K$ and residue field $k$; further let $G$ denote a connected, simply connected, split semisimple, smooth affine algebraic group scheme over $\Spec R$ with Lie algebra $\mathfrak{g}:=\Lie(G)$. We define $\widehat{U(\mathfrak{g})}= \invlim U(\mathfrak{g})/\pi^i U(\mathfrak{g})$ to be the $\pi$-adic completion of $U(\mathfrak{g})$. We call $\widehat{U(\mathfrak{g})_K}:= \widehat{U(\mathfrak{g})} \uset{R} K$ the affinoid enveloping algebra of $\mathfrak{g}$. In \cite{Munster}, the authors ask the following question:

\begin{question}\cite[Question A]{Munster}
Is it the case that every primitive ideal of $\widehat{U(\mathfrak{g})_K}$ with $K$-rational infinitesimal central character is the annihilator of a simple affinoid highest weight module?
\end{question}

We should outline the strategy of answering this question. Let $\mathfrak{h}$ be a Cartan subalgebra of $\mathfrak{g}$; for $\lambda \in \mathfrak{h}^*$, we may define the affinoid Verma module of weight $\lambda$, $\widehat{M(\lambda)}$, see Definition \ref{d3affinoidvermadef}.

\begin{theorem}[Theorem \ref{d3affinoidclassicvermacorrespondenceprop}]
Let $\lambda \in \mathfrak{h}^*$. There is a one to one correspondence between submodules of $\widehat{M(\lambda)}$ and submodules of $M(\lambda)$. In particular, $\widehat{M(\lambda)}$ has a unique simple quotient $\widehat{L(\lambda)}$.

\end{theorem}

We can now state the most important result of this article.

\begin{theorem}[Theorem \ref{d3affinoidduflotheorem} and Corollary \ref{d3corollarymunsterduflo}]

Let $R$ be a mixed characteristic $(0,p)$ complete discrete valuation ring and let $G$ be a connected, simply-connected, split semisimple, smooth affine algebraic group scheme over $\Spec R$. Denote $\mathfrak{g}:=\Lie(G)$ the Lie algebra of $G$.

Any primitive ideal in the affinoid enveloping algebra $\widehat{U(\mathfrak{g})_K}$ with $K$-rational infinitesimal central character is the annihilator of some $\widehat{L(\lambda)}$.

\end{theorem}

In fact, we prove a more general result classifying all the primitive ideals with $K$-rational infinitesimal central character in $\hugnK:=(\invlim U(\mathfrak{g})/ \pi^{i} U(\mathfrak{\pi^n g})) \uset{R} K$. The $K$-rationality condition for central characters occurs because $K$ is not algebraically closed, so there are maximal ideals of the centre of $\hugnK$ that do not correspond to points in $K$. Fortunately, using the affinoid version of Quillen's Lemma \cite[Corollary 8.6]{Annals} by Ardakov and Wadsley which is enhanced in \cite[Theorem 6.4.6]{StaPhd}, we classify  in Theorem \ref{d3alltheprimitiveideals} \emph{all} the primitive ideals in $\hugnK$. Further, we are able to prove in Theorem \ref{d3finaltheorem} that a large class of two-sided ideals in $\hugnK$ is controlled by ideals in the classical enveloping algebra.

Let $\lambda$ be a regular dominant weight. Using Theorem $\ref{d3affinoidduflotheorem}$, we prove in Theorem \ref{d3controllertheorem} that any two-sided ideal in $\hugnK$ with $\chi_{\lambda}$ central character is controlled by a two-sided ideal in the classical enveloping algebra $U(\fr{g}_K)$.

In order to prove Theorem \ref{d3affinoidduflotheorem}, we enhance the affinoid Beilinson-Bernstein localisation \cite[Theorem C]{Annals} developed by Ardakov and Wadsley to the equivariant setting, prove an affinoid version of the Borho-Brylinski equivalence, and prove an affinoid version of Proposition \ref{classicalJosfun}.

We should also remark that our initial approach was to try to adapt one of the classical proofs in \cite{Du}, \cite{BGG}, \cite{Gin}, \cite{Jos} to the affinoid setting. Unfortunately, these approaches failed to produce results for $\mathfrak{g} \neq \mathfrak{sl}_2$. It boils down to the fact that the weight spaces of the $\ad$-action of the Cartan subalgebra on the affinoid enveloping algebra are not finite dimensional. This is in contrast to what happens in Theorem \ref{d3affinoidclassicvermacorrespondenceprop}, where we can adapt classical machinery to obtain a correspondence between the lattices of submodules of $\widehat{M(\lambda)}$ and $M(\lambda)$, respectively.

\subsection{Connection to the Iwasawa algebras}

Assume that $K$ is a finite extension of $\mathbb{Q}_p$ and let $\mathfrak{G}$ be a compact open subgroup of $G(R)$. Let

  $$K \mathfrak{G}:= (\varprojlim_{\mathfrak{N} \unlhd_{o} \mathfrak{G}} R[\mathfrak{G}/\mathfrak{N}]) \uset{R} K$$
  
denote the Iwasawa algebra of $\mathfrak{G}$. It is known that there is an equivalence between continuous $K$-representations of $\mathfrak{G}$ and finitely generated $K\mathfrak{G}$-modules. Following Jacobson, we aim to characterise simple $K\mathfrak{G}$-modules by classifying the primitive ideals in $K\mathfrak{G}$. It is conjectured that all the non-zero primitive ideals in $K\mathfrak{G}$ arise as annihilators of finite dimensional simple modules. Ardakov and Wadsley claim in \cite{Munster} that our theorem \ref{d3affinoidduflotheorem} implies that the conjecture is true provided one can prove that every affinoid highest weight module that is not finite-dimensional over $K$ is faithful as a $K\mathfrak{G}$-module.

\textbf{Structure of the paper}

The paper is organised as follows: in Section \ref{sectionBackground}, we review the main constructions and results in the previous papers of the series: \cite{Sta1} and \cite{Sta2}. Next, in Section \ref{d3sectionaffinoidenveloping}, we introduce affinoid enveloping algebras and affinoid Verma modules. We prove that for any weight $\lambda$ of the Cartan subalgebra, there is an explicit one-to-one correspondence between submodules of affinoid Verma module of weight $\lambda$ and the classical Verma module of weight $\lambda$.

In \cite{Sta2}, we have proven that there is an equivalence of categories between $G$-equivariant $(\lambda,\mu)$-twisted $\mathcal{D}$-modules on the double flag variety and $B$-equivariant $\lambda$-twisted $\mathcal{D}$-modules on the flag variety for any all weights $\lambda,\mu$. In Section \ref{d3sectionBoBr}, we prove an affinoid version of this equivalence.

Next, we enhance the affinoid Beilinson-Bernstein equivalence proven by Ardakov and Wadsley in \cite{Annals} to the equivariant setting. We further prove that any two-sided ideal in the affinoid enveloping algebra is $G$-equivariant when viewed as a bimodule over the affinoid enveloping algebra.

In Section \ref{d3sectionglobalsections}, we compute global sections under the affinoid pullback functor defined in Section \ref{d3sectionBoBr}. Finally, in Section \ref{d3sectionduflothm}, we prove an affinoid version of Duflo's theorem and \ref{d3sectioncontrollertheorem} we prove that certain two-sided ideals in the affinoid enveloping algebra are controlled by two-sided ideals in the classical enveloping algebra.

\textbf{Conventions}

Throughout this document, except otherwise stated, $R$ will denote a mixed characteristic $(0,p)$ complete discrete valuation ring with uniformiser $\pi$ and field of fractions $K$. We use $|| \cdot ||$ to denote the norm of an element in $R$ or $K$. 

Given an $R$-module $M$, we define $M_K:=M \uset{R} K$. For any $R$-algebra $A$ and for $\mathfrak{g}$ a $R$-Lie algebra, we define $\mathfrak{g}_A:=\mathfrak{g} \uset{R} A$; if $M$ is an $R$-module, we denote $M_A:=M \uset{R} A$. If $\mathcal{M}$ is a sheaf of $R$-modules on a topological space $Y$, we define a sheaf of $K$ vector spaces on $Y$, $\mathcal{M}_K$, by $\mathcal{M}_K(U):=\mathcal{M}(U) \uset{R} K$ for any $U \subset Y$ open.

Following \cite[definition 2.7]{Annals}, an $R$-module/sheaf of $R$-modules  $M$/$\mathcal{M}$ of a $K$-vector space/sheaf of $K$-vector spaces $V/\mathcal{V}$ will be called a \emph{lattice} if $$M \uset{R} K \cong V/ \mathcal{M} \uset{R} K \cong \mathcal{V} \text { and } \cap_{n \in \mathbb{N}^*} \pi^n M=0/ \cap_{n \in \mathbb{N}^*} \pi^n \mathcal{M}=0.$$

We will use $\hat{\otimes}$ to denote the completed tensor product, see \cite[definition 2.7]{StackProject} for definition and basic properties. We will assume that all the filtrations appearing are exhaustive. Given a filtered $A$ with filtration $F_i A, i \in \mathbb{Z}$, we will use $ \gr A$ to denote the associated graded ring with respect to the filtration. Further, for any ring $A$, $Z(A)$ will denote its centre. We will use the notation $(V_i)$ to denote a set of objects indexed by the non-negative natural numbers.

Lastly, given $f:X \to Y$ a map of schemes, we will use $f^*$ to denote the pullback in the category of $\mathcal{O}$ and $\mathcal{D}$-modules and $f^{-1},f_*$ to denote the inverse/direct image sheaf.

\textbf{Acknowledgements}. This paper is going to be part of the author's D.Phil thesis, which is being produced in Oxford under the supervision of Prof. Konstantin Ardakov. We are grateful for the mathematical ideas, talks, directions of research, encouragements given by him. We also thank Andreas Bode, Joshua Ciappara, Adam Jones, and Richard Mathers for many fruitful mathematical conversations. Lastly, we thank Diana Kessler for correcting lots of typos. The author's D.Phil is supported by an EPSRC scholarship.

\section{Background}
\label{sectionBackground}

We recall the main results and construction from the first two papers in the series \cite{Sta1}, \cite{Sta2} that we will use throughout this document. For now, let $R$ be a commutative Noetherian ring.

\subsection{Deformations}
\label{subsectiondeformations}

\begin{definition}
Let $A$ be a positively $\mathbb{Z}$-filtered $R$-algebra such that $F_0 A$ is an $R$-subalgebra of $A$. We call $A$ a \emph{deformable} $R$-algebra if $\gr A$ is a flat $R$-module. A morphism of deformable $R$-algebras is an $R$-linear filtered ring homomorphism.
\end{definition}

\begin{definition}
Let $A$ be a deformable $R$-algebra and let $r \in R$ a regular element. The $r$-th deformation of $A$ is the following $R$-submodule of $A$:

 $$A_r:= \sum_{i=0}^{\infty} r^i F_i A.$$

\end{definition}

By construction, one can see that $A_r$ is a $R$-subalgebra of $R$. Further, the definition is functorial, and the following lemma states that we have a family of endofunctors $A \mapsto A_r$.

\begin{lemma}
\label{associatedgradedofdeformations}

Let $A$ be a deformable $R$-algebra and $r \in R$ a regular element. Then $A_r$ is also a deformable $R$-algebra and there is a natural isomorphism $\gr A \cong \gr A_r$.

\end{lemma}

\begin{proof}
We give $A_r$ the subspace filtration $F_i A_r= F_i A \cap A_r$. As $\gr A$ is flat over $R$, we have $F_i A_r= \sum_{j=0}^i r^j F_j A$. For $i \geq 1$ define a $R$-linear map 
$$f:F_i A/ F_{i-1} A \to F_i A_r/F_{i-1} A_r, \quad f(x+F_{i-1} A)=r^i x +F_{i-1} A_r.$$

To finish the proof, it is enough to check that $f$ is bijective. First, we prove that $f$ is injective. Assume that $r^i x \in F_{i-1} A_r $, so $r^i x \in F_{i-1} A$ which implies that $x \in F_{i-1} A$ since $\gr A$ is flat, so in particular $R$-torsion free. It is straightforward to see that $f$ is also surjective. 
\end{proof}

\subsection{Deformed twisted differential operators}
For the rest of the section, we let $r \in R$ be a regular element.

\begin{definition}
We call an $R$-scheme $X$ that is smooth, separated and locally of finite type an \emph{$R$-variety}.
\end{definition}

Throughout this subsection, fix $X$ an $R$-variety. We write $\mathcal{T}_X$ for the sheaf of sections of the tangent bundle $TX$.

\begin{definition}\cite[Definition 4.2]{Annals}
\label{crystallinedifferentialoperatorsdefinition}

Let $X$ be an $R$-variety. The sheaf of crystalline differential operators is defined to be the enveloping algebra $\mathcal{D}_X$ of the Lie algebroid $\mathcal{T}_X$.
\end{definition}

We can view $\mathcal{D}_X$  as a sheaf of ring generated by $\mathcal{O}_X$ and $\mathcal{T}_X$ modulo the relations:

\begin{itemize}
\item{$f \partial =f \cdot \partial$;}
\item{$ \partial f - f \partial= \partial (f)$;}
\item{$ \partial \partial' - \partial' \partial=[\partial,\partial'],$}
\end{itemize}

for all $f \in \mathcal{O}_X$ and $\partial,\partial' \in \mathcal{T}_X$. The sheaf $\mathcal{D}_X$ comes equipped with a natural PBW filtration:

 $$0 \subset F_0 (\mathcal{D}_X) \subset  F_1 (\mathcal{D}_X) \subset \ldots $$

consisting of coherent $\mathcal{O}_X$-modules such that

$$F_0(\mathcal{D}_X)= \mathcal{O}_X, \enskip F_1( \mathcal{D}_X)= \mathcal{O}_X \oplus \mathcal{T}_X, \enskip F_m(\mathcal{D}_X) = F_1 (\mathcal{D}_X) \cdot F_{m-1}( \mathcal{D}_X)  \text{ for } m >1.                   $$

Since $X$ is smooth, the tangent sheaf $\mathcal{T}_X$ is locally free and the associated graded algebra of $\mathcal{D}_X$ is isomorphic to the symmetric algebra of $\mathcal{T}_X$:

\begin{equation}
\label{gradingcrysdifop}
 \gr(\mathcal{D}_X)= \bigoplus_{m=1}^{\infty} \frac{F_m(\mathcal{D}_X)}{F_{m-1}(\mathcal{D}_X)} \cong \Sym_{\mathcal{O}_X} \mathcal{T}_X.
\end{equation}

If $q:T^*X \to X$ is the cotangent bundle of $X$ defined by the locally free sheaf $\mathcal{T}_X$, then we can also identify $\gr(\mathcal{D}_X)$ with $q_{*}\mathcal{O}_{T^*X}$.

Let $X$ be an $R$-variety and let $U=\Spec(A) \subset X$ be open affine. Further, we consider $\mathcal{M}$ a sheaf of $\mathcal{O}_X$-bimodules quasi-coherent with respect to the left action. We define a filtration $F_{\bullet}M$ via

\begin{itemize}
\item{$F_{-1}(M)=0,$}
\item{$F_n(M)= \{ m \in M | \ad(a_0)\ad(a_1)...\ad(a_n)(m)=0 \text{ for any } a_0,a_1, \ldots a_n \in A\}$} for $n \geq 0$.

\end{itemize} 

We say that $M$ is differential if $M= \cup F_n(M)$ and we call $\mathcal{M}$ a differential $\mathcal{O}_X$-bimodule if there is an affine open cover $(U_i)_{i \in I}$ such that $\mathcal{M}(U_i)$ is a differential bimodule for all $i \in I$.

Let $\mathcal{M},\mathcal{N}$ be two quasi-coherent $\mathcal{O}_X$-modules. Then for any affine open $U \subset X$ the set $\Hom_{R}(\mathcal{M}(U),\mathcal{N}(U))$ has the structure of a $\mathcal{O}_X(U)$-bimodule. Let $\mathcal{F} \in \Hom_{R}(\mathcal{M}, \mathcal{N}$); we say that $\mathcal{F}$ is a differential operator of degree $\leq n$ if for any affine open $U$, $\mathcal{F}(U) \in F_n(\Hom_{R}(\mathcal{M}(U),\mathcal{N}(U)).$

\begin{definition}
Let $\mathcal{A}$ be a $\mathcal{O}_X$-algebra. We say that $\mathcal{A}$ is a differential algebra if $\mathcal{A}$ is a flat $R$-module and multiplication makes $\mathcal{A}$ a differential $\mathcal{O}_X$-bimodule. The filtration $F_{\bullet}(A)$ becomes a ring filtration and with respect to this filtration, $\gr A$ is commutative.

\end{definition}


\begin{definition}
\label{tdodefinition}
An algebra of $r$-deformed twisted differential operators(tdo) is an $\mathcal{O}_X$-differential algebra $\mathcal{D}$ such that:
\begin{enumerate}[label=\roman*)]
\item{ The natural map $\mathcal{O}_X \to F_0(\mathcal{D})$ is an isomorphism.}
\item{ The morphism $\gr_{1} \mathcal{D} \to \mathcal{T}_X=\Der_R(\mathcal{O}_X,\mathcal{O}_X)$ defined by $\psi \mapsto \ad_{\psi}$ for $\psi \in F_1(\mathcal{D})$ induces an isomorphism $\gr_{1} \mathcal{D} \to r \mathcal{T}_X.$}
\item{The morphism of $\mathcal{O}_X$-algebras $\Sym_{\mathcal{O}_X}(\gr_1 \mathcal{D}) \to \gr \mathcal{D}$ is an isomorphism.}
\end{enumerate}

\end{definition}

In it easy to see by construction that if the base ring $R$ is Noetherian, any $r$-deformed tdo is a sheaf of Noetherian rings.

%
%

\subsection{Equivariant \texorpdfstring{$\mathcal{O}$}{O}-modules}
Let $G$ be an affine algebraic group scheme over $\Spec R$ acting on a  scheme $X$; denote the action by $\sigma_X: G \times X \to X$. Furthermore, we denote $p_X:G \times X \to X$ and $p_{2X}:G \times G \times X \to X$ the projections on the $X$ factor, $p_{23X}:G \times G \times X \to G \times X$ the projection onto the second and third factor and $m:G \times G \to G$ the multiplication of the group $G$.

\begin{definition}

Let $G$ an algebraic group scheme acting on a scheme $X$. A $G$-equivariant $\mathcal{O}_X$-module is a pair $(\mathcal{M},\alpha)$ where $\mathcal{M}$ is a quasi-coherent $\mathcal{O}_X$-module and $\alpha:\sigma_X^*\mathcal{M} \to p_X^*\mathcal{M}$ is an isomorphism of $\mathcal{O}_{G \times X}$-modules such that the diagram

\begin{center}
\begin{tikzcd}

&(1_G \times \sigma_X)^*p_X^*\mathcal{M} \arrow[r," p_{23X}^* \alpha"]   &p_{2X}^*\mathcal{M} \\
&(1_G \times \sigma_X)^* \sigma_X^* \mathcal{M} \arrow [u,"(1_G \times \sigma_X)^* \alpha "] \arrow[r,leftrightarrow,"id "] &(m \times 1_X)^*\sigma_X^* \mathcal{M} \arrow[u, " (m \times 1_X)^* \alpha"]
\end{tikzcd}
\end{center}

of $\mathcal{O}_{G \times G \times X}$-modules commutes (the cocycle condition) and the pullback 
                        $$(e \times 1_X)^* \alpha: \mathcal{M} \to \mathcal{M}$$
is the identity map.
\end{definition}

\begin{lemma}\cite[Lemma 2.2]{Sta1}
\label{Oequivpreservefunctor}

Let $G$ be an affine algebraic group acting on schemes $X$ and $Y$ and let $f:Y \to X$ be a $G$-equivariant morphism. Then the pullback functor $f^*$ given by $$(\mathcal{M},\alpha) \mapsto (f^*\mathcal{M},(1_G \times f)^* \alpha)$$ defines a functor from $G$-equivariant $\mathcal{O}_X$-modules to $G$-equivariant $\mathcal{O}_Y$-modules.

\end{lemma}

\begin{definition}
Let $G$ an affine algebraic group acting on a scheme $X$ via $\sigma_X$. We define the category of $G$-equivariant quasi-coherent $\mathcal{O}_X$-modules. Objects are given by $G$-equivariant $\mathcal{O}_X$-modules.

A morphism of $G$-equivariant $\mathcal{O}_X$ modules $(\mathcal{M},\alpha_M)$ and $(\mathcal{N},\alpha_N)$ is a map $\phi \in \Hom_{\mathcal{O}_X}(\mathcal{M},\mathcal{N})$ such that the following diagram commutes:
\begin{center}
\begin{tikzcd}
&\sigma_X^* \mathcal{M} \arrow[d,"\sigma_X^*\phi"] \arrow[r,"\alpha_M"] &p_X^* \mathcal{M} \arrow[d,"p_X^* \phi"] \\
&\sigma_X^* \mathcal{N}  \arrow[r,"\alpha_N"] &p_X^* \mathcal{N}. 
\end{tikzcd} 
\end{center}

We call such a morphism $G$-equivariant. We denote $\QCoh(\mathcal{O}_X,G)$ the category of $G$-equivariant $\mathcal{O}_X$-modules together with $G$-equivariant morphisms.


\end{definition}

\begin{proposition}\cite[Proposition 2.4]{Sta1}
\label{GequivariantAbeliancat}

Let $G$ an affine algebraic group acting on a scheme $X$. Then the category $\QCoh(\mathcal{O}_X,G)$ is Abelian.

\end{proposition}

From now on, when we use the notion of morphism of $G$-equivariant $\mathcal{O}_X$-modules, we always view it as a morphism in the category $\QCoh(\mathcal{O}_X,G)$.

\textbf{A reformulation of equivariance}

We wish to reformulate the notion of an equivariant $\mathcal{O}$-module. Until the end of the section, we fix $X$ a scheme defined over $R$ acted on by an affine algebraic group $G$. We start with a very simple observation: viewing $\mathcal{O}_X$ as a left $\mathcal{O}_X$-module, $(\mathcal{O}_X,\id)$ is a $G$-equivariant $\mathcal{O}_X$-module. We can reformulate this following ideas in \cite{MvdB}: for each $R$-algebra $A$ inducing a map $s:\Spec A \to \Spec R$ and for each geometric point $i_{g}:\Spec A \to G$ which induces an automorphism $g:X_{A} \to X_{A}$ there exists an isomorphism

                                                                    $$q_g: s^* \mathcal{O} \to (g^{-1})^*s^* \mathcal{O}, \text{ satisfying }$$

\begin{equation}
\label{structureequiveq}
              q_e=\id \text{ and } q_{gh}=(g^{-1})^*(q_h)q_g
\end{equation}              
in such a way that $(q_g)$'s are compatible with base change. Let $r_g=g^* \circ q_g$. For each $U \subset X_A$ affine open, $r_g$ is a map $\mathcal{O}_A(U) \to \mathcal{O}_A(g^{-1}U)$. The equation \ref{structureequiveq} translates as $r_e=\id$ and $r_{gh}=r_hr_g$. Furthermore, the $\mathcal{O}$-module compatibility requires that for any $f_1,f_2 \in \mathcal{O}_A(U)$, we have $r_g(f_1f_2)=r_g(f_1)r_g(f_2)$.  

We define $r_g$ via $r_g(f)(x)=f(g^{-1}x)$ for all $R$-algebras $A$, $U \subset X_A$ affine open, $x \in U$, $f \in \mathcal{O}_A(U)$, $g:X_{A} \to X_{A}$ and it is easy to see that $r_g$'s make $\mathcal{O}_X$ a $G$-equivariant $\mathcal{O}_X$-module. We may now make an abuse of notation: for each $i_g: \Spec A \to G$ and each $ f \in \mathcal{O}_A(U)$, we denote $g.f=r_{g^{-1}}(f)$ and we translate  the equivariance structure as

\begin{equation*}
 e.f_1=f, \enskip g.(h.f_1)=(gh).f_1, \enskip g.(f_1f_2)=(g.f_1)(g.f_2) \text{ for } g,h \in G,  f_1,f_2 \in \mathcal{O}_X.
\end{equation*}

\begin{lemma}\cite[Lemma 2.5]{Sta1}
A $\mathcal{O}_X$-module  $\mathcal{M}$ is $G$-equivariant if and only if for each $R$-algebra $A$, for each $s:\Spec A \to \Spec R$ and for each geometric point $i_g: \Spec A \to G$ which induces an automorphism $g:X_A \to X_A$ there exists an isomorphism of $\mathcal{O}_{A}$-modules

      $$ q_g:s^* \mathcal{M} \to (g^{-1})^* s^* \mathcal{M}        $$

satisfying

\begin{equation}
\label{weaklyequiveqOmod}
q_{e}=\id \text{ and }  q_{gh}=(g^{-1})^*(q_h)q_g
\end{equation}

in such a way that $(q_g)$'s are compatible with base change.

\end{lemma}

Again by setting $s_g=g^* \circ q_g$, we may reformulate equation \ref{weaklyequiveqOmod} as:  for each $R$-algebra $A$ and for each $i_g:\Spec A \to G$, we have an isomorphism of $\mathcal{O}$-modules $s_g: \mathcal{M}_{X_A} \to \mathcal{M}_{X_A}$ such that for each $U \subset X_A$ affine open:

\begin{equation}
\label{alternativeeqOmod}
\begin{split}
&s_{e}=\id,\\
&s_{gh}=s_{h}s_{g},\\
&s_{g}'s \text{ are compatible with base change},\\
& r_g(f.m)=r_g(f).s_g(m) \text{ for all }  f \in \mathcal{O}_{Y_A}(U), m \in \mathcal{M}(U).
\end{split}
\end{equation}

Again, we make an abuse of notation:  for each $i_g: \Spec A \to G$ and each $ m \in \mathcal{M}_A(U)$, we denote $g.m=s_{g^{-1}}(m)$ and we translate  the equivariance structure as:

\begin{equation}
\label{easyweakOmodequivariant}
\begin{split}
&e.m=m, \\
&gh.m= g.(h.m),\\
&g.(f.m)=(g.f).(g.m).
\end{split}
\end{equation}

for all $g,h \in G$, $m \in \mathcal{M}$,  $ f \in \mathcal{O}_X$.

\subsection{Deformed Homogeneous twisted differential operators}

Throughout this subsection, we will assume that $G$ is a connected, smooth affine algebraic group scheme over $\Spec R$ with Lie algebra $\mathfrak{g}:=\Lie(G)$.

\begin{definition}
\label{d3htdodef}
Let $\mathcal{D}$ be a differential $\mathcal{O}_X$-algebra. We call $\mathcal{D}$ a sheaf $r$-deformed $G$-homogeneous twisted differential operators ($G$-htdo) if $\mathcal{D}$ is $G$-equivariant as a left $\mathcal{O}_X$-module, $\mathcal{D}$ is an $r$-deformed tdo  and $\mathcal{D}$ is equipped with a Lie algebra map $i_{\mathfrak{g}}:r\mathfrak{g} \to \mathcal{D}$ such that:

\begin{enumerate}[label=\roman*)]
\item{$g.1=1$ and $g.(d_1d_2)=(g.d_1)(g.d_2)$ for $g \in G$ and $d_1,d_2 \in \mathcal{D}.$ }
\item{$g.(fd)=(g.f)(g.d)$ for $f \in \mathcal{O}_X$ and $d \in \mathcal{D}$.}
\item{$i_{\mathfrak{g}}(g.\psi)=g.i_{\mathfrak{g}}(\psi)$ for $g \in G, \psi \in r\mathfrak{g}.$}
\item{The derivative of the $G$-action induces a $\mathfrak{g}$-action and so a $r \mathfrak{g} \subset \mathfrak{g}$ action. This must coincide with the action $d \to [i_{\mathfrak{g}}(\psi),d]$ for $\psi \in r\mathfrak{g}$ and $ d \in \mathcal{D}$.}

\item{$i_{\mathfrak{g}}(r\mathfrak{g}) \subset F_1 \mathcal{D}.$}
\item{$\eta=\rho \circ i_{\mathfrak{g}}$ as maps from $r \mathfrak{g}$ to $r \mathcal{T}_X$ where $\eta:\mathfrak{g} \to \mathcal{T}_X$ is the infinitesimal map and $\rho: F_1 \mathcal{D} \to \mathcal{T}_X$ is the natural anchor map.}
\end{enumerate}
\end{definition}

Let $Y$ be another $R$-variety such that $G$ acts on $Y$ and $f:Y \to X$ is a $G$-equivariant morphism. Then for $\mathcal{D}$ an $r$-deformed  $G$-htdo on $X$, we defined in \cite[Definition 7.5]{Sta1} its pullback, $f^{\#} \mathcal{D}$ and proved in \cite[Corollary 7.6]{Sta1} that it is an $r$-deformed $G$-htdo on $Y$.

Assume further that $f:Y \to X$ is a locally trivial $G$-torsor (see \cite[Section 4.3]{Annals} for the definition). Then for $\mathcal{A}$ an $r$-deformed $G$-htdo on $Y$, we defined its descent, $f_{\#} \mathcal{A}^G$, see \cite[Definition 10.9]{Sta1} and proved in \cite[Lemma 10.10]{Sta1} that it is an $r$-deformed tdo on $X$. 

The main proposition we need is:

\begin{proposition}\cite[Corollary 10.13]{Sta1}
\label{eqdescentliealg}
Let $f:Y \to X$ be a locally trivial $G$-torsor. Let $B$ be another smooth affine algebraic group acting on $X$ and $Y$, such that $G$ and $B$ on $Y$ commute. The maps $f_{\#}(-)^G$ and $f^{\#}(-)$ induce inverse bijections from the set of $r$-deformed  $G \times B$-htdo's on $Y$ to the set of $r$-deformed $B$-htdo's on $X$.
\end{proposition} 

In particular, by setting $B$ to be a trivial group we obtain a bijection between the set of $r$-deformed $G$-htdo's on $Y$ and the set of $r$-deformed tdo's on $X$.

\begin{definition}
\label{equivarianthtdomoddef}
Let $(\mathcal{D},i_{\mathfrak{g}})$ be a $r$-deformed $G$-htdo and $L$ be a closed subgroup of $G$, with Lie algebra $\mathfrak{l}$. We call $\mathcal{D}$-module $\mathcal{M}$ weakly $L$-equivariant if:

\begin{enumerate}[label=\roman*)]
\item{$\mathcal{M}$ is an $L$-equivariant  $\mathcal{O}_X$-module.}
\item{$g.(D.m)=(g.D).(g.m)$ for any $g \in L, d \in \mathcal{D}, m \in \mathcal{M}$.} 

We call $\mathcal{M}$ $L$-equivariant if in addition:

\item{ The $r\mathfrak{l}$-action induced by the derivative of the $L$-action on $\mathcal{M}$ coincides with the $r\mathfrak{l}$-action induced by the restriction of $i_{\mathfrak{g}}$ to $r\mathfrak{l}$.}

A morphism of (weakly) equivariant $\mathcal{D}$-modules is a $\mathcal{D}$-linear morphism of $L$-equivariant $\mathcal{O}_X$-modules.
\end{enumerate}

\end{definition}

In case $L=G$, we recover \cite[Definition 9.3]{Sta1}, but we will need this more general definition for explaining the localisation mechanism. We denote $\Coh(\mathcal{D},G)$ the category of coherent $G$-equivariant coherent $\mathcal{D}$-modules. 

Let $Y$ be another $R$-variety such that $G$ acts on $Y$, $f:Y \to X$ is a $G$-equivariant morphism and let $\mathcal{D}$ be $G$-htdo on $X$. Given a $G$-equivariant $\mathcal{D}$-module $\mathcal{M}$, we may endow the $\mathcal{O}_Y$-module $f^* \mathcal{M}$ with an action of the ring $f^{\#} \mathcal{D}$ and we call the resulting module $f^{\#} \mathcal{M}$. We prove in \cite[Lemma 9.7]{Sta1} that this is $G$-equivariant.

We may redefine the notion of $G$-equivariance of an $r$-deformed $G$-htdo module. Denote the $G$-action by $\sigma_X: G \times X \to X$. Furthermore, we denote $p_X:G \times X \to X$ and $p_{2X}:G \times G \times X \to X$ the projections on the $X$ factor, $p_{23X}:G \times G \times X \to G \times X$ the projection onto the second and third factor and $m:G \times G \to G$ the multiplication of the group $G$. Then we define a $G$-equivariant $\mathcal{D}$-module as a pair $(\mathcal{M},\alpha)$, where $\mathcal{M}$ is a $\mathcal{D}$-module and $\alpha:\sigma_X^{\#} \mathcal{M} \to p_X^{\#} \mathcal{M}$ is an isomorphism of $p_X^{\#} \mathcal{D}$-modules such that the diagram:

\begin{equation}
\label{eqmoduleGhtdocommdia}
\begin{tikzcd}
&(1_G \times \sigma_X)^{\#}p_X^{\#}\mathcal{M} \arrow[r," p_{23X}^{\#} \alpha"]   &p_{2X}^{\#}\mathcal{M} \\
&(1_G \times \sigma_X)^{\#} \sigma_X^{\#} \mathcal{M} \arrow [u,"(1_G \times \sigma_X)^{\#} \alpha "] \arrow[r,leftrightarrow,"id "] &(m \times 1_X)^{\#}\sigma_X^{\#} \mathcal{M} \arrow[u, " (m \times 1_X)^{\#} \alpha"]
\end{tikzcd}
\end{equation}

commutes and the pullback 
                        $$(e \times 1_X)^{\#} \alpha: \mathcal{M} \to \mathcal{M}$$
is the identity map. We will ignore the equivariance structure when it is understood from the context.

\subsection{An equivalence a la Borho-Brylinski}

Let $G$ be a connected,  simply-connected, smooth affine algebraic group scheme over $\Spec R$,  $B$ a closed subgroup of $G$; we make the following assumption:

\begin{assumption}
\label{d3loctrivialassumption}
The quotient scheme $X=G/B$ is an $R$-variety and the quotient map $d_B:G \to X$ given by $d_B(g)=gB$ is a locally trivial $B$-torsor with respect to the action $\diamond$ given by $b \diamond g=gb^{-1}.$
\end{assumption}

This assumption is, in particular, satisfied when $B$ is a Borel subgroup of a split semisimple group $G$ and $X=G/B$ is the flag scheme.

We consider the diagonal action of $G$ on $X \times X$ and the natural action by left translation of $B$ on $X$. Let $i_r/i_l:X \to X \times X$, $i_r(x)=(eB,x), i_l(x)=(x,eB)$ denote the inclusion of $X$ into the right/left copy of $X \times X$. We also fix $(\mathcal{D},i_{\mathfrak{g}})$ an $r$-deformed $G$-htdo on $X \times X$ with respect to the diagonal $G$-action.

\begin{theorem}\cite[Theorem 3.5,Corollary 3.11]{Sta2}
\label{d3alaBorho-Brylinski}
The pullback $i_r^{\#} \mathcal{D}/i_l^{\#} \mathcal{D}$ is an $r$-deformed $B$-htdo and the functors

          $$i_r^{\#}: \Coh(\mathcal{D},G) \to \Coh(i_r^{\#}\mathcal{D},B), \qquad i_l^{\#}: \Coh(\mathcal{D},G) \to \Coh(i_l^{\#}\mathcal{D},B),$$
          
are equivalences of categories. Let $\mathscr{H}_r$ and $\mathscr{H}_l$ the respective quasi-inverses.

\end{theorem}

We will also need the following corollary:

\begin{corollary} \cite[Corollary 3.16]{Sta2}
\label{d3classicalzeroglosection}
Let $\mathcal{N} \in \Coh(\mathcal{D},G)$ with $\Gamma(X,i_l^{\#} \mathcal{N})=0$. Then $\Gamma(X \times X, \mathcal{N})=0$.

\end{corollary}

\subsection{The localisation mechanism}
\label{subsectionclassciallocalisationmechanism}

Throughout this subsection, $G$ will denote a connected, simply-connected, smooth affine algebraic group over $R$, $\mathfrak{g}=\Lie(G)$ its Lie algebra, $X$ will denote an $R$-variety with a $G$-action and $r \in R$ a regular element. We fix $(\mathcal{D},i_{\mathfrak{g}})$ an $r$-deformed $G$-htdo on $X$.

Fix $r \in R$ a regular element and consider the $r$-th deformation of $U(\mathfrak{g})$ denoted $U(\mathfrak{g})_r$. Using the PBW theorem we obtain that $U(\mathfrak{g})_r \cong U(r \mathfrak{g})$.
The enveloping algebra $U(\mathfrak{g})$ is a $G$-representation via the Adjoint action, so by the module-comodule duality we obtain a  map $\rho: U(\mathfrak{g}) \to \mathcal{O}(G) \uset{R} U(\mathfrak{g})$ making $U(\mathfrak{g})$ a comodule for the Hopf algebra $\mathcal{O}(G)$. Furthermore, since the $G$ action commutes with the $R$ action, the map $\rho$ restricts to a map $\rho:U(r\mathfrak{g}) \to \mathcal{O}(G) \uset{R} U(r \mathfrak{g})$.

Let $L$ be a closed subgroup of $G$. Then $L$ also acts on $U(r\mathfrak{g})$ via the restriction to $L$ of the Adjoint action of $G$. Again, by duality we obtain a comodule map $\rho_{r\mathfrak{g},L}:U(r\mathfrak{g}) \to \mathcal{O}(L) \uset{R} U(r\mathfrak{g})$.

Let $M$ be a $U(r\mathfrak{g})$-module that is also an $\mathcal{O}(L)$-comodule. The comodule structure induces an action of $L$;  the derivative of the $L$-action induces an action of the Lie algebra $\mathfrak{l}=\Lie(L)$, and so of $r \mathfrak{l}$, on $M$. Furthermore, since $U(r\mathfrak{g})$ and $M$ are $\mathcal{O}(L)$-comodules, so is $U(r\mathfrak{g}) \uset{R} M$, see \cite[Section 1.8]{Mont}  for details.

\begin{definition}
\label{d3nonequivariantLUgmodule}
A weakly $L$-equivariant $U(r\mathfrak{g})$ module is a triple $(M,\alpha,\rho)$, where $M$ is an $R$-module, $\alpha:U(r\mathfrak{g}) \uset{R} M \to M$ is a left $U(r\mathfrak{g})$-action, $\rho:M \to \mathcal{O}(L) \uset{R} M$ is a $\mathcal{O}(L)$ co-action such that $\alpha$ is a morphism of $\mathcal{O}(L)$-comodules.

Furthermore, if the action of $r\mathfrak{l} \subset \mathfrak{l}=\Lie(L)$ induced by $\rho$ by the derivation of the $L$-action coincides with the restriction of the $r\mathfrak{g}$ action to $r\mathfrak{l}$, we say that $(M,\alpha,\rho)$ is $L$-equivariant. As for equivariant $\mathcal{D}$-modules, we will omit the equivariance structure when it is understood from the context.

A morphism of (weakly) $L$-equivariant $U(\mathfrak{g})$-modules $(M,\alpha,\rho_1)$ and $(N,\beta,\rho_2)$ is a map $f:M \to N$ of Abelian groups that is $U(\mathfrak{g})$-linear with respect to actions $\alpha,\beta$ and $\mathcal{O}(L)$-colinear with respect to $\rho_1$ and $\rho_2$. We call such a morphism $L$-equivariant.

Denote $\Mod(U(r\mathfrak{g}),L)$ the category of consisting of $L$-equivariant $U(r\mathfrak{g})$-modules together with $L$-equivariant morphisms.

\end{definition}

We can reformulate the weakly equivariant condition in the following way: by the module-comodule correspondence $M$ can be viewed as a representation of the algebraic group $L$. Since $U(r\mathfrak{g})$ is also an $L$-representation we may rewrite condition that the map $\alpha:U(r\mathfrak{g}) \uset{R} M \to M$ is a morphism of $\mathcal{O}(L)$-comodules as:

$$ l.( \psi.m)=(l.\psi).(l.m),$$

for all $R$-algebras $A$, $l \in L(A)$, $\psi \in U(r\mathfrak{g})_A$ and $m \in M_A$.  By abuse of language we define an equivalent notion of a weakly $L$-equivariant $U(r\mathfrak{g})$-module by:

\begin{center}
$M$ is a representation of $L$,
\end{center}
\begin{equation}
\label{equivliealgmodeq}
l.(\psi.m)=(l.\psi).(l.m) \text{ for all } l \in L, \psi \in U(r\mathfrak{g}), m \in M.
\end{equation}

\begin{proposition}\cite[Proposition 5.1,Proposition 5.4]{Sta2}
\label{d3equivariancepreservedunderlocandglo}
Let $L$ be a closed subgroup of $G$ and let $(\mathcal{D},i_{\mathfrak{g}})$ be an $r$-deformed $G$-htdo on $X$.

\begin{enumerate}[label=\roman*)]
\item{Let $M$ be an $L$-equivariant $U(r \mathfrak{g})$-module. Then $\mathcal{D} \uset{U(r \mathfrak{g})} M$ is an $L$-equivariant $\mathcal{D}$-module.}
\item{Let $\mathcal{M}$ be an $L$-equivariant $\mathcal{D}$-module. Then $\Gamma(X,\mathcal{M})$ is an $L$-equivariant $U(r \mathfrak{g})$-module.}

\end{enumerate}
\end{proposition}

We are interested to apply the localisation mechanism when $X$ is the flag scheme of a split semisimple group $G$. Let $\mathfrak{h}$ be the Cartan subalgebra of $\mathfrak{g}$. Recall from \cite[Section 5.2]{Sta2} that for each $R$-linear map $\lambda:r\mathfrak{h} \to R$ we constructed a sheaf of \emph{$r$-deformed $\lambda$-twisted} differential operators $\mathcal{D}_{\lambda,r}$.

\begin{proposition}\cite[Corollary 5.15]{Sta2}
\label{d3DlambdarisrdeformedGhtdo}
There exists a map $\alpha:r \mathfrak{g} \to \mathcal{D}_{\lambda,r}$ such that $(\mathcal{D}_{\lambda,r}, \alpha)$ is an $r$-deformed $G$-htdo on $X$. 
\end{proposition}

\section{Affinoid enveloping algebras and Verma modules}
\label{d3sectionaffinoidenveloping}

From now on, till the end of the chapter, we will assume that $R$ is a complete mixed characteristic $(0,p)$ discrete valuation ring with field of fractions $K$, uniformiser $\pi$ and residue field $k$.

For a deformable $R$-algebra $A$ and $n \in \mathbb{N}^*$, we denote $A_n:=A_{\pi^n}$ the $\pi^n$-th deformation of $A$.

\subsection{Background on affinoid enveloping algebras}

In this subsection, we recall the main construction and results concerning affinoid enveloping algebras.

Let $G$ be a connected, simply connected,
split semisimple, smooth affine algebraic group scheme over $\Spec R$. Denote $\mathfrak{g}$ the Lie algebra of $G$. The Lie algebra $\mathfrak{g}$ is a linear $G$ representation via the Adjoint action; see \cite[II.1.12]{Jan1} for details. In particular the functor of points $G(R)$ acts on $\mathfrak{g}$. Using the functoriality one may extend this action to the enveloping algebra $U(\mathfrak{g})$. For example, if we consider a monomial $x_1x_2 \ldots x_n \in U(\mathfrak{g})$, with $x_i \in \mathfrak{g}$, we get that for each $g \in G(R)$ we have

          $$g \cdot x_1x_2 \ldots x_n = (g \cdot x_1) (g \cdot x_2) \ldots (g \cdot x_n).$$

It follows that the action of $G(R)$ preserves the standard PBW filtration on $U(\mathfrak{g})$. Consider the corresponding comodule structure on $\mathcal{O}(G)$ induced by the action of $G$ and let $\rho:U(\mathfrak{g}) \to \mathcal{O}(G) \otimes U(\mathfrak{g})$ be the defining map. It follows from the definition of the $G(R)$ action that the comodule map satisfies $\rho(ab)=\rho(a)\rho(b)$ for any $a,b \in U(\mathfrak{g})$.

Let $H$ be a fixed maximal torus for $G$ and $\bm{\Phi}$ the corresponding root system, and $x_{\alpha}: G_a \to G$ and $e_{\alpha}:=(dx_{\alpha})(1) \in \mathfrak{g}$ be the root homomorphism and root vector corresponding to a root $\alpha \in \bm{\Phi}$.

\begin{lemma}[\cite{Munster}, Lemma 4.1]
Let $r \in R$ and  $a \in \bm{\Phi}$. Then the following hold:

\begin{enumerate}
\item{For every $G$ module $M$, the action of $\frac{e_{\alpha}^m}{m!}$ preserves $M$.}
\item{For all $b \in U(\mathfrak{g})$, there exists $i \geq 1$ such that $\frac{\ad(re_\alpha)^i}{i!} \cdot b=0$.}
\item{$x_\alpha(r) \cdot a = \sum_{m=0}^{\infty} \frac{\ad(re_\alpha)^m}{m!}(a)$ for all $a \in U(\mathfrak{g})$. }
\end{enumerate}
\end{lemma}

\begin{definition}
Let $A$ be an $R$-algebra. The $\pi$-adic completion of the $R$-algebra $A$ is defined to be $\widehat{A}=\invlim A/\pi^iA$.
\end{definition}

Let $u_1,u_2, \ldots u_d$ be a free $R$-basis of $\mathfrak{g}$. Then as a vector space we have

\begin{equation}
\label{d3affinoidenvelopingalgebradescription}
\hugnK= \{\sum_{\alpha \in \mathbb{N}^d} \lambda_\alpha u^{\alpha}: \quad \lambda_\alpha \in K,   \pi^{-n|\alpha|} \lambda_\alpha \to 0 \text{ as } |\alpha| \to \infty \}.
\end{equation}

Here for a $d$-tuple $\alpha=(\alpha_1, \alpha_2, \ldots \alpha_d)$, we define $|\alpha|= \sum_{i=1}^d \alpha_i$ and \\$u^{\alpha}=u_1^{\alpha_1} u_2 ^{\alpha_2} \ldots
u_d^{\alpha_d}$.

By functoriality, the Adjoint action of $G$ on $U(\mathfrak{g})$ extends to a $G$-action on $\hugnK$. The following proposition extends the classical results for enveloping algebras defined over a field of characteristic 0.

\begin{lemma}\cite[Corollary 4.3]{Munster}
\label{d3twosidedidealsarepreservedbyadjointaction}

\begin{enumerate}[label=\roman*)]
\item{Every two sided ideal in $\hugnK$ is preserved by $G(R)$.}
\item{For any $z \in Z(\hugnK)$ and for any $g \in G(R)$, we have $g \cdot z=z$.}
\end{enumerate} 
\end{lemma}

One may wonder if the converse of Lemma \ref{d3twosidedidealsarepreservedbyadjointaction} $ii)$ also holds. Classically, we have $Z(U(\mathfrak{g}_K)) \cong U(\mathfrak{g}_K)^{G}$. The following theorem states that the result carries in the affinoid setting:

\begin{theorem}\cite[Theorem 4.4]{Munster}
\label{d3centreofaffinoidenvelopingalgebras}
We have $ Z(\hugnK) \cong \widehat{U(\mathfrak{g})^G_{n,K}}.$
\end{theorem}

Recall that $H \subset B^{-}$ is a split maximal torus in $G$ contained in $B^{-
}$. The unipotent radical $N^{-}$ of $B^{-}$ will be considered as generated by negative roots corresponding to the adjoint action of $H$ on $G$. Furthermore, let $N^+$ be the unipotent radical of the opposite Borel group $B^+$ containing $H$. Let $\mathfrak{h,b^{-},n^{-},n^+,b^+}$ be the Lie algebras corresponding to the algebraic groups so that we have a decomposition
                           $$ \mathfrak{g}=\mathfrak{n}^- \oplus  \mathfrak{h} \oplus \mathfrak{n}^+.$$

Let $\lambda:\pi^n\mathfrak{h} \to R$ be an $R$-linear character; extend this to an $R$-linear map $\pi^n\mathfrak{b}^+ \to R$ by pulling along the projection map $\pi^n\mathfrak{b}^+ \to \pi^n \mathfrak{h}$. Similar to the classical case, we denote $K_\lambda$ the corresponding one dimensional module over $\widehat{U(\mathfrak{b}^+)_{n,K}}$; the Lie algebra $\mathfrak{b}^+$ acts on $K_\lambda$ via the corresponding map $\pi^n \mathfrak{b}^+ \to R$ and we may extend this action to the whole algebra  $\widehat{U(\mathfrak{b}^+)_{n,K}}$.

\begin{definition}
\label{d3affinoidvermadef}
The affinoid Verma module with highest weight $\lambda$ is defined to be

                      $$\widehat{M(\lambda)}:=\widehat{U(\mathfrak{g})_{n,K}} \uset{\widehat{U(\mathfrak{b}^+)_{n,K}}} K_\lambda.$$

\end{definition}

Notice that affinoid Verma modules are non-trivial for $\mathfrak{h}_K$-weights induced by  weights of $\pi^n \mathfrak{h}$; for a general $\mathfrak{h}_K$-weight, an unit in $\hugnK$ may annihilate the affinoid Verma module. It is clear by construction that, similarly to the classical case, the affinoid Verma modules are cyclic: $ \widehat{M(\lambda)}$ is generated by $v_\lambda= 1 \uset{\widehat{U(\mathfrak{b}^+)_{n,K}}} 1$.

The centre $Z(\mathfrak{g}_K)$ of $U(\mathfrak{g}_K)$ acts on the classical Verma module defined by  $M(\lambda):=U(\mathfrak{g}_K)\uset{U(\mathfrak{b}^+_K)} K_\lambda$ by a character $\chi_\lambda:Z(\mathfrak{g}_K) \to K_\lambda$. As  $\widehat{M(\lambda)}$ contains $M(\lambda)$ as a dense subset the action of $Z(\mathfrak{g}_K)$ on $\widehat{M(\lambda)}$ also factors through $\chi_\lambda$. In \cite{Munster}, the authors compute the annihilator of the affinoid Verma module $\widehat{M(\lambda)}$.

\begin{theorem}\cite[Theorem 4.6]{Munster}
\label{d3annihilatoraffinoidverma}
If $p$ is a very good prime for $G$ then the annihilator of the affinoid Verma module $\widehat{M(\lambda)}$ inside $\widehat{U(\mathfrak{g})_{n,K}}$ is $$\widehat{I_\lambda}:=\ker \chi_\lambda \widehat{U(\mathfrak{g})_{n,K}}.$$
\end{theorem}

For the rest of section we fix a $R$-linear map $\lambda:\pi^n \mathfrak{h} \to R$ and let $\widehat{M(\lambda)}$ and $M(\lambda)$ be the affinoid respectively classical Verma module of weight $\lambda$. In the next subsections, we prove there is an explicit one-to-one correspondence between submodules of $\widehat{M(\lambda)}$ and submodules of $M(\lambda)$.

\subsection{The height function}

For the semisimple Lie algebra $\mathfrak{g}$, let $\Delta$ denote the set of simple positive roots and $\mathbf{\Phi}^+$ the set of positive roots. For any root $\alpha$, we will denote $\alpha^{\vee}$ the corresponding coroot. In the Killing form identification of $\mathfrak{h}$ and $\mathfrak{h}^*$ the coroot $\alpha^{\vee}$ corresponds to $h_\alpha \in \mathfrak{h}$.

\begin{definition}
Let $\beta \in \mathbf{\Phi}^+$ be a positive root. Then $\beta=\sum_{\alpha \in \Delta} c_{\alpha} \alpha $, with $c_{\alpha} \in \mathbb{Z}^+$ determined uniquely, see \cite[section 0.2]{Hu1} for details. We define the \emph{height} of $\beta$ to be 
                 $$\height(\beta):=\sum_{\alpha \in \Delta} c_{\alpha}.$$
\end{definition}

We now extend this definition to monomials in the universal enveloping algebra using the correspondence between roots and root vectors. Fix an order between the positive roots and let $e_1,e_2 \ldots, e_m$ the corresponding order between root vectors. For a root vector $e_i$, we define the height, $\height(e_i),$ to be the height of the root corresponding to $e_i$.
\begin{definition}
For $a_1,a_2 \ldots a_m \in \mathbb{N}$, let $e^A:=e_1^{a_1}e_2^{a_2} \ldots e_m^{a_m} \in U(\mathfrak{n}^+)$ be such that $e_i \in \mathfrak{n}^+$. Then we define the \emph{height} of $e^A$ to be

       $$\height(e^A):=\sum_{i=1}^{m} a_i \height (e_i).$$

Let $f^B=f_1^{b_1} f_2^{b_2} \ldots  f_m^{b_m} \in U(\mathfrak{n}^-) $ such that $f_i \in \mathfrak{n}^-$. Then we define the \emph{height} of $f^B$ to be
     $$\height(f^B):=\sum_{i=1}^{m} b_i   \height (e_i) ,$$       

where $e_i$ is the positive root vector corresponding to $f_i$.

\end{definition}

Let $\rho$ be the half sum of positive roots and $\delta=\rho^{\vee} \in \mathfrak{h}$ the corresponding coroot. Let $\alpha$ be a positive root; then by the roots-coroots duality we have $\alpha(\delta)=\rho(\alpha^{\vee})$; furthermore  by \cite[section 0.6]{Hu1}, we have $\rho(\alpha^{\vee})=\height(\alpha)$, therefore we obtain $\alpha(\delta)= \text{ht}(\alpha)$.

\begin{lemma}\cite[Section 4.7]{Munster}
\label{d3cartanaffinoidaction}
Let $f^{B}=f_1^{b_1}f_2^{b_2} \cdots f_m^{b_m} \in U(\mathfrak{n}^{-})$ for $\beta \in \mathbb{N}^{m}$. Then  for any  $h \in \mathfrak{h}$, we have:

          $$h \cdot f^{B} v_\lambda=(\lambda-\sum_{j=1}^m b_j \alpha_j)(h)f^{B}v_\lambda.$$

\end{lemma}

For the ease of notation, denote $\Lambda:= \lambda(\delta)$. Setting $h=\delta$ in the equation above we get:

\begin{equation}
\label{d3deltaaction}
 \delta f^Bv_{\lambda}=(\lambda - \sum_{i=1}^m b_i \alpha_i)(\delta)f^Bv_\lambda=(\Lambda-\sum_{i=1}^m b_i  \height(\alpha_i))f^Bv_\lambda= (\Lambda- \height(f^B))f^Bv_\lambda.
\end{equation}

As an easy corollary we get: 

\begin{corollary}
\label{d3shifteddeltaaction}
Let $a \in \mathbb{N},f^B \in  U(\mathfrak{n}^-)$. Then:

$$ (\delta-\Lambda+a)(f^B)v_\lambda=(a-\height(f^B))f^Bv_\lambda.$$

\end{corollary}

\begin{definition}
Let $M$ be a $U(\mathfrak{h})$-module and $\mu \in \mathfrak{h}_K^*$. We say that $m \in M$ has weight $\mu$ if $hm= \mu(h)m$ for all $h \in \mathfrak{h}_K$. The set of vectors of weight $\mu$ is denoted $M_\mu$.
\end{definition}

 The following lemma  follows easily from the construction of affinoid Verma modules:

\begin{lemma}
\label{d3weightspacefinitedimensional}
Let $N$ a submodule of $\widehat{M (\lambda)}$. Then $N_\mu$ is a finite dimensional vector space for any $\mu \in \mathfrak{h}_K^*.$

\end{lemma}

\subsection{Submodules of affinoid Verma modules}

Throughout this subsection, we will make free use of the following well known facts:

\begin{itemize}
\item{$U(\mathfrak{g}_K) \cong U(\mathfrak{g})_K=U(\mathfrak{g}) \uset{R} K$.}
\item{$\widehat{U(\mathfrak{g})_n}$ is flat over $U(\mathfrak{g})_n$ and $\hugnK$ is flat over $U(\mathfrak{g})_K$.}
\item{Given $N$ a submodule of $M(\lambda)$, we may view $N$ as a subset of the topological module $\widehat{M(\lambda)}$. The closure of $N$ inside $\widehat{M(\lambda)}$ is given by $$\hat{N}:=\overline{N}=\hugnK \uset{U(\mathfrak{g})_K} N.$$}
\item{ The affinoid Verma module $\widehat{M(\lambda)}$ has a $K$-topological basis given by $f^B v_\lambda$, where $f^B \in U(\mathfrak{n}^-)$; recall that $v_{\lambda}=1 \otimes 1$.}

\end{itemize}

 We begin by extending our definition of the height function to homogeneous polynomials, homogeneity being given by height. We say that a  polynomial in $U(\mathfrak{n}_K^-)$ has height $n$ if all the monomials appearing in its expansion have height $n$. We also let $M=$max($\height(e_i))$, so that we have the inequality 
\begin{equation} 
\label{d3heightfunctionequality}
M|B| \geq \height(f^B) \geq |B|.
\end{equation}

By construction, we know that as a vector space

$$\widehat{U(\mathfrak{n}^-)_{n,K}}=\{\sum_{B \in \mathbb{N}^m} a_B f^B, \quad p^{-n |B|} ||a_B|| \to 0 \text{ as } |B| \to \infty \}.$$

We can reformulate this in terms of height function using equation \eqref{d3heightfunctionequality}:

$$\widehat{U(\mathfrak{n}^-)_{n,K}}=\{\sum_{B \in \mathbb{N}^m} a_B f^B, \quad p^{-n |B|} ||a_B|| \to 0 \text{ as } \height(f^B) \to \infty \}.$$

Let $N$ a closed submodule of $\widehat{M (\lambda)}$ and let $\nu \in \mathfrak{h}^*$. We know by equation \ref{d3deltaaction} that $\delta f^B v_\lambda = (\Lambda-\height(f^B))f^Bv_\lambda$ so any element of $N_\nu$ must be of the form $P v_\lambda$, where $P \in U(\mathfrak{n}_K^-)$ is a homogeneous polynomial of height $\Lambda-\nu(\delta)$.

\begin{proposition}
\label{d3submodulesclos}
Let $N$ be a closed submodule of $\widehat{M (\lambda)}$. Then $$N=\overline{\bigoplus_{\mu \in \mathfrak{h}_K^*} N_\mu}=\overline{N \cap M(\lambda)},$$ where $M(\lambda)$ is the Verma module of weight $\lambda$. 
\end{proposition}

Fix the closed submodule $N$ and an element $u \in N$, which we write as
              $$u=\sum_{B \in \mathbb{N}^m}a_{B}f^Bv_{\lambda},$$ 

with $p^{-n |B|} ||a_{B}|| \rightarrow 0$ as $\height(f^B) \rightarrow \infty.$
Furthermore, fix $Pv_\lambda \in N_\mu$, with $P \neq 0$ of height $L$ appearing in the expansion of $u$. To prove Proposition \ref{d3submodulesclos}, it is enough to prove that $Pv_\lambda \in N$. To do this, we begin by eliminating all the other terms of height $L$ appearing in the expansion of $u$. We write $u$ as
    
      $$u=\sum_{B \in \mathbb{N}^m,\height(f^B) \neq L}a_{B}f^Bv_{\lambda}+ Pv_{\lambda}+ \sum_{s \in S} Q_sv_\lambda,$$
  where $S$ is a set such that $Q_s$ has  height $L$ and $Q_sv_\lambda \in N_{\nu_s}$ for some weight $\nu_s \neq \mu$.  By Lemma \ref{d3weightspacefinitedimensional}, the set $S$ is finite.

Let $s \in S$. As $\nu_s \neq \mu$, there exists $h_s \in \mathfrak{h}_K$ such that $\mu(h_s) \neq \nu_s(h_s)$. So we can define an operator $H_S:= \prod_{s \in S}(h_s-\nu_s(h_s)) \in U(\mathfrak{h}_K)$. Since $Q_s \in N_{\nu_s}$, we get $H_S \cdot Q_sv_\lambda=0$. Therefore applying the operator $H_S$ to $u$ we obtain
a new element $u' \in N$, which can be written as
$$u'=\sum_{B \in \mathbb{N}^m,\height(f^B) \neq L}b_{B}f^Bv_{\lambda}+\prod_{s \in S}(\mu(h_s)-\nu_s(h_s))Pv_\lambda,$$       

with $p^{-n |B|} ||b_{B}|| \rightarrow 0$ as $\height(f^B) \rightarrow \infty$. By our construction  
$\prod_{s \in S}(\mu(h_s)-\nu_s(h_s)) \neq 0$, so for the ease of notation we set    $P=\prod_{s \in S}(\mu(h_s)-\nu_s(h_s))P \neq 0$, so that

    $$N \ni u'=\sum_{B \in \mathbb{N}^m,\height(f^B) \neq L}b_{B}f^Bv_{\lambda}+Pv_{\lambda}.$$

To complete the proof we need to define a new set of operators in $U(\mathfrak{h}_K)$. We use the convention that for $x \in U(\mathfrak{h}_K)$ and $i \in \mathbb{N}$ the symbol $\binom{x}{i}$ will denote 

$$\prod_{l=0}^{i-1} \frac{1}{i!} (x-l)  \in U(\mathfrak{h}_K).$$  

\begin{definition}
For  $i,j \in \mathbb{N}$  define $\epsilon_{i,j}:=\binom{\delta-\Lambda+i+j}{i} \in U(\mathfrak{h}_K) $.
\end{definition}

\begin{lemma}
\label{d3epsilonaction}
For  $B \in \mathbb{N}^m$, we have that $\epsilon_{i,j}f^Bv_{\lambda}=\binom{i+j-\height(f^B)}{i} f^B v_{\lambda}$.
\end{lemma} 

\begin{proof}
We have 
\begin{equation}
\begin{split}
 \epsilon_{i,j}f^Bv_\lambda&= \frac{1}{i!}\prod_{l=0}^{i-1}(\delta-\Lambda+i+j-l)f^Bv_\lambda \\
              &=\frac{1}{i!}\prod_{l=0}^{i-1}(i+j-l-\height(f^B))f^Bv_\lambda \text{ (by Corollary \ref{d3shifteddeltaaction}) }\\
                &=\binom{i+j-\height(f^B)}{i}f^Bv_\lambda. \qedhere
\end{split}
\end{equation}
\end{proof}

We are now ready to prove Proposition \ref{d3submodulesclos}; recall that it is enough to prove that for $u'=\sum_{B \in \mathbb{N}^m,\height(f^B) \neq L}b_{B}f^Bv_{\lambda}+Pv_{\lambda}$, we have $Pv_\lambda \in N$.

\begin{proof}[Proof of Proposition \ref{d3submodulesclos}]

Write $u'$ as $$u'=\sum_{B \in \mathbb{N}^m,\height(f^B) < L}b_{B}f^Bv_{\lambda}+Pv_\lambda+ \sum_{B \in \mathbb{N}^m,\height(f^B) > L}b_{B}f^Bv_{\lambda}.$$ Consider the operator $\epsilon_{L-1,0}$ acting on $u'$. We get

\begin{equation}
\begin{split}
N \ni u'':&=\epsilon_{L-1,0} \cdot u' \\
    &=\sum_{B \in \mathbb{N}^m,\height(f^B) < L} \binom{L-1-\height(f^B)}{L-1}b_{B}f^Bv_{\lambda}+ \binom{L-1-L}{L-1}Pv_\lambda+ \\
    &+\sum_{B \in \mathbb{N}^m,\height(f^B) > L} \binom{L-1-\height(f^B)}{L-1}b_{B}f^Bv_{\lambda} \\
    &=b_0v_{\lambda}+(-1)^{L-1}Pv_\lambda+\sum_{B \in \mathbb{N}^m,\height(f^B) > L} \binom{L-1-\height(f^B)}{L-1}b_{B}f^Bv_{\lambda}.
\end{split}
\end{equation}

Next, we apply the operator $\frac{(-1)^L(\delta-\Lambda)}{L}$. We have

\begin{equation}
\begin{split}
N \ni u^{(3)}:&=\frac{(-1)^n(\delta-\Lambda)}{L} \cdot u'' \\
           &=\frac{(-1)^L}{L}[(\Lambda-\Lambda)b_0v_\lambda+(\Lambda-L-\Lambda)(-1)^{L-1}Pv_\lambda + \\
           & +^{\text{ (by \ref{d3shifteddeltaaction})}}   \sum_{B \in \mathbb{N}^m,\height(f^B) > L} (\Lambda-\height(f^B) -\Lambda)\binom{(L-1-\height(f^B))}{L-1}b_{B}f^Bv_{\lambda}]  \\
           &=Pv_\lambda+\sum_{B \in \mathbb{N}^m,\height(f^B) > L} c_Bf^Bv_\lambda,  
\end{split}
\end{equation}

for some $c_B \in K$ with $p^{-n |B|} ||c_B|| \rightarrow 0$ as $\height(f^B) \rightarrow \infty.$ Finally, we consider the family of operators $\epsilon_{i,L}$ where $i \in \mathbb{N}$ varies. By Lemma \ref{d3epsilonaction}, we have that 
$$\epsilon_{i,L}f^Bv_{\lambda}=\binom{i+L-\height(f^B)}{i}f^Bv_\lambda .$$

In particular, $\epsilon_{i,L}Pv_\lambda=Pv_\lambda$ and $\epsilon_{i,L}f^Bv_{\lambda}=0,$ for $L<\height(f^B) \leq i+L$. Therefore, for any $i \in \mathbb{N}$ one has

$$N \ni u_i=\epsilon_{i,L} \cdot u^{(3)}=Pv_\lambda+\sum_{B \in \mathbb{N}^m,\height(f^B) > i+L} \binom{i+L-\height(f^B)}{i} c_Bf^B v_\lambda.$$

Since $p^{-n|B|}||c_B|| \to 0$ as $\height(f^B) \to \infty$ and $||\binom{i+L-\height(f^B)}{i}|| \leq 1$, we get $\lim_{i \to \infty} u_i =Pv_\lambda$. Since we assumed that $N$ is closed submodule of $\widehat{M(\lambda)}$, we  may conclude that $Pv_\lambda \in N$, which is the desired result.
\end{proof}

Recall that we aim to prove that there is a one-to-one correspondence between submodules of affinoid Verma module $\widehat{M(\lambda)}$ and the corresponding classical Verma module $M(\lambda)$ where $\lambda:\pi^n\mathfrak{h} \to R$ is a $R$-linear map. Define a function $\mathscr{F}$ from submodules of $\widehat{M (\lambda)}$ to submodules of $M(\lambda)$ sending the submodule $N$ to $N \cap M(\lambda)$.

\begin{lemma}
\label{d3submodulecorrespondenceinjective}
The function $\mathscr{F}$ is injective.

\end{lemma}

\begin{proof}
Let $N_1,N_2$ be submodules of $\widehat{M(\lambda)}$ such that  $$\mathscr{F}(N_1)=N_1 \cap M(\lambda)= N_2 \cap M(\lambda)=\mathscr{F}(N_2).$$ 

As the induced metric topology on $\widehat{M(\lambda)}$ is complete, any submodule of $\widehat{M(\lambda)}$ is closed by \cite[I.5.5]{LVO}. Thus,  applying Proposition \ref{d3submodulesclos}, we obtain $N_1=\overline{ N_1 \cap M(\lambda)}= \overline{N_2 \cap M(\lambda)}=N_2$, so $N_1=N_2$.  Therefore, the function $\mathscr{F}$ is injective.
\end{proof}

We aim to prove that $\mathscr{F}$ is also surjective.
For any ring $A$, $S$ a subset of $A$ and $M$ an $A$-module, we say that $M$ is not $S$-torsion if there exists $m \in M$ such that for all $s \in S$, $sm \neq 0$.

\begin{proposition}
\label{d3tensorproductU(g)Kisnon-zero}
Let $\lambda:\pi^n \mathfrak{h} \to R$ be an $R$-linear map and let $M$ be a $U(\mathfrak{g}_K)$ subquotient of $M(\lambda)$  that is not $1+\pi U(\mathfrak{g})$-torsion. Then

         $$\hugnK \uset{U(\mathfrak{g}_K)} M \neq 0.$$

\end{proposition}

To prove this proposition we will need a few additional results. We say that a right ideal $I$ in a right Noetherian ring $A$ has the Artin-Rees property if for any right ideal $J$, there exists $n \in \mathbb{N}^*$ such that $J \cap I^n \subset JI.$

\begin{proposition}\cite[Proposition 4.2.9]{ConRob}
\label{d3ArtinReesConRob}

Let $A$ be a right Noetherian ring and $I$ an ideal with the right Artin-Rees property. Then:

\begin{enumerate}
\item{$1-I$ is a right Ore set, so a right denominator set.}
\item{Writing $S$ for $1-I$, we have $I_S \subset J(A_S)$, where $J(\bullet)$ denotes the Jacobson radical of a ring and $I_S$ and $A_S$ denote the sets $S^{-1}I$ and $S^{-1}A$, respectively.}
\end{enumerate}

\end{proposition}

\begin{corollary}
\label{d3LocalizationofU(g)at1+Icorrolary}
Consider  the ideal $I=\pi U(\mathfrak{g})_n$ of $U(\mathfrak{g})_n$. Then  $U(\mathfrak{g})_{n_{1+I}}$ exists and is non-zero; furthermore $\pi U(\mathfrak{g})_{n_{1+I}} \subset J(U(\mathfrak{g})_{n_{1+I}})$. 

\end{corollary}

\begin{proof}
By \cite[Proposition 4.2.6]{ConRob}, any ideal generated by normal elements in a right Noetherian ring has the Artin-Rees property, so in particular we get that the ideal $\pi U(\mathfrak{g})_n$ in $U(\mathfrak{g})_n$ has the Artin-Rees property.  The claim follows from Proposition \ref{d3ArtinReesConRob}.
\end{proof}

\begin{remark}
Notice that as the ring $U(\mathfrak{g})_n$ is both left and right Noetherian we get the same results for the left localization. Thus, $_{1+\pi U(\mathfrak{g})_n}U(\mathfrak{g})_n=U(\mathfrak{g})_{n_{1+ \pi U(\mathfrak{g})}}$ by \cite[Corollary 2.1.4]{ConRob}.

\end{remark}

From now on until the end of the section,  $S$ will denote the set $1+\pi U(\mathfrak{g})_n$. The ring $U(\mathfrak{g})_{n_S}$ has a $\pi$-adically negative filtration $F_{\bullet}$ given by
       
              $$F_{i} U(\mathfrak{g})_{n_S} = \pi^{-i} U(\mathfrak{g})_{n_S}, \text{ for } i \leq 0.$$

Denote $\widehat{U(\mathfrak{g})_{n_S}}$ the $\pi$-adically completion of $U(\mathfrak{g})_{n_S}$, i.e. the completion induced by the filtration $F_{\bullet}$.

\begin{proposition}
\label{d3completionoflocalisationfaitfullyflatproposition}
$\widehat{U(\mathfrak{g})_{n_S}}$ is a faithfully flat right $U(\mathfrak{g})_{n_S}$-module.

\end{proposition}

\begin{proof}

Since $U(\mathfrak{g})_n$ is left and right Noetherian, and $S$ is an Ore set, we get that the ring $U(\mathfrak{g})_{n_S}$ is also left and right Noetherian. Furthermore, we have that the ideal inducing the $\pi$-adic filtration on $U(\mathfrak{g})_S$ is generated by a central element, so by \cite[Proposition 3.12]{Nas}, the Rees ring $\widetilde{U(\mathfrak{g})_{n_S}}$ is also left and right Noetherian.  By Corollary \ref{d3LocalizationofU(g)at1+Icorrolary}, we have $F_{-1}U(\mathfrak{g})_{n_S}\subset J(F_0 U(\mathfrak{g})_{n_S})$, therefore by combining the two statements, we get that $U(\mathfrak{g})_S$ is a left \emph{Zariski} ring. The claim follows from \cite[Theorem II.2.2]{LVO}.
\end{proof}

\begin{corollary}
\label{d3tensorproductU(g)isnonzerocorollary}
Let $M$ be a non-zero  $U(\mathfrak{g})_n$-module that is not $S$-torsion. Then 
          
                $$\widehat{U(\mathfrak{g})_n} \uset{U(\mathfrak{g})_n} M \neq 0.$$

\end{corollary}

\begin{proof}
  Since $M$ is not $S$-torsion, we have by localising
          $$0 \neq S^{-1}M= S^{-1}U(\mathfrak{g})_n \uset{U(\mathfrak{g})_n} M=U(\mathfrak{g})_{n_S} \uset{U(\mathfrak{g})_n} M.$$

By Proposition \ref{d3completionoflocalisationfaitfullyflatproposition}, $\widehat{U(\mathfrak{g})_{n_S}}$ is faithfully flat over $U(\mathfrak{g})_{n_S}$, so
            $$\widehat{U(\mathfrak{g})_{n_S}} \uset{U(\mathfrak{g})_{n_S}} S^{-1}M \neq 0.$$          

Recall that all the elements of $S$ are of the form $1+\pi x$, with $x \in U(\mathfrak{g})_n$, but when we $\pi$-adically complete $U(\mathfrak{g})_{n_S}$ everything in $S$ becomes a unit, so $\widehat{U(\mathfrak{g})_{n_S}} \cong \widehat{U(\mathfrak{g})_n}$. We then get:

\begin{equation}
\begin{split}
0 \neq \widehat{U(\mathfrak{g})_{n_S}} \uset{U(\mathfrak{g})_{n_S}} S^{-1}M &\cong \widehat{U(\mathfrak{g})_n} \uset{U(\mathfrak{g})_{n_S}} S^{-1}M \\
          &\cong \widehat{U(\mathfrak{g})_n} \uset{U(\mathfrak{g})_{n_S}} U(\mathfrak{g})_{n_S} \uset{U(\mathfrak{g})_n}M \\
          &\cong \widehat{U(\mathfrak{g})_n} \uset{U(\mathfrak{g})_n}M. \qedhere
\end{split}
\end{equation}
\end{proof}

We now apply the results for objects in category $\mathcal{O}$ for the enveloping algebra $U(\mathfrak{g}_K).$ Let $\lambda:\pi^n\mathfrak{h} \to R $ be an $R$-linear map and  consider the simple $U(\mathfrak{g}_K)$ module $L(\lambda)$-the unique simple quotient of $M(\lambda)$- and view it as a $U(\mathfrak{g})$-module.

\begin{lemma}
\label{d3Llambdaistorsionfreelemma}

Let the notations be as above. The module $L(\lambda)$ is not $1+\pi U(\mathfrak{g})$-torsion.

\end{lemma}

\begin{proof}
The module $L(\lambda)$ is cyclic being generated by $v_\lambda+N(\lambda)$, where $N(\lambda)$ is the unique maximal submodule  of $M(\lambda)$. It is enough to prove that $v_\lambda+N(\lambda)$ is $1+ \pi U(\mathfrak{g})$ torsion-free.

Consider the Cartan Lie subalgebra $\mathfrak{h}$. We extend the character $\lambda:\mathfrak{h} \to R$  to an $R$ algebra homomorphism  $\lambda:U(\mathfrak{h}) \to R$. We use the decomposition of $U(\mathfrak{g})$  given by $U(\mathfrak{g})=(\mathfrak{n}^-U(\mathfrak{g})+ U(\mathfrak{g})\mathfrak{n}^+) \oplus U(\mathfrak{h})$. Notice that if $x \in U(\mathfrak{g}) \mathfrak{n}^+$, then $x v_\lambda=0$. Furthermore, if $x \in \mathfrak{n}^{-} U(\mathfrak{g})$, then

\begin{equation}
\label{d3Llamdaistorsionfreeeq1}        
         x v_\lambda \in \mathfrak{n}^{-}_KU(\mathfrak{n}^-_{K})v_\lambda.  
\end{equation}

In fact, one may prove that $x v_\lambda$ is in $\mathfrak{n^-}U(\mathfrak{n^-})v_\lambda$, but that requires a messy computation and we do not need this in our argument. Next, for $y \in U(\mathfrak{h})$, we have

\begin{equation}
\label{d3Llambdaistorsionfreeeq2}
 y v_\lambda= \lambda(y) v_\lambda, \text{ where } \lambda(y) \in R.
\end{equation}

Now, let $U(\mathfrak{g}) \ni z=x+y$ be such that $x \in  \mathfrak{n}^-U(\mathfrak{g})+ U(\mathfrak{g})\mathfrak{n}^+$ and $y \in U(\mathfrak{h})$. By equation \eqref{d3Llamdaistorsionfreeeq1} there exists $s \in \mathfrak{n}_K^-U(\mathfrak{n}_K^-)$ such that  $xv_\lambda=s v_\lambda$ and by equation \eqref{d3Llambdaistorsionfreeeq2}, $y v_{\lambda}=\lambda(y) v_{\lambda}$. Therefore, we get

$$(1+\pi z) \cdot (v_\lambda+N(\lambda))= (1+\pi s + \pi \lambda(y) )v_\lambda +N(\lambda).$$

Proving that $(1+\pi z) v_{\lambda}+N(\lambda) \neq 0$ is equivalent to $(1+\pi z) v_\lambda \notin N(\lambda)$. Assume for a contradiction that $(1+\pi z) v_\lambda \in N(\lambda)$. Then $(1+ \pi \lambda(y))v_\lambda+ \pi s v_\lambda \in N(\lambda)$. View $ (1+ \pi \lambda(y))v_\lambda+ \pi s v_\lambda$ as an element in $M(\lambda)$. Consider the decomposition of $M(\lambda)$ given by

            $$M(\lambda)=M(\lambda)_{\lambda} \oplus M(\lambda)_{<\lambda}=K v_\lambda \oplus M(\lambda)_{<\lambda},$$

where $M(\lambda)_{<\lambda}$ denotes the $K$-span of all $v_\mu \in M(\lambda)_\mu$ with $\mu < \lambda$. Notice that since $s \in \mathfrak{n}_K^-U(\mathfrak{n}_K^-)$, we have $\pi s v_\lambda \in M(\lambda)_{<\lambda}$; furthermore, $1+\pi \lambda(y) v_\lambda \in M(\lambda)_{\lambda}$. Now since the module $N(\lambda)$ is itself $\mathfrak{h}_K$-semisimple, we have

        $$(1+\pi \lambda(y)) v_\lambda \in N(\lambda).$$          
          
By construction, $\lambda(y) \in R$, so $||\pi \lambda(y)||<1$, thus $1+ \pi \lambda(y)$ is a unit in $R$. Therefore, multiplying by its inverse we conclude that $v_\lambda \in N(\lambda)$, so $N(\lambda)=M(\lambda)$ which is the desired contradiction. We conclude that $L(\lambda)$ is indeed  not $1+\pi U(\mathfrak{g})$-torsion.       
\end{proof}

As an easy corollary of the Lemma above, using the fact $1+ \pi U(\mathfrak{g})_n$ is a subset of $1+\pi U(\mathfrak{g})$, we obtain:

\begin{corollary}
\label{d3LlambdaistorsionfreeCorollary}
Let the notations as in the previous lemma. View $L(\lambda)$ as $U(\mathfrak{g})_n$-module. Then $L(\lambda)$ is not $1+ \pi U(\mathfrak{g})_n$-torsion.

\end{corollary}

We may now prove Proposition \ref{d3tensorproductU(g)Kisnon-zero}:

\begin{proof}
Let $M$ be a subquotient of $M(\lambda)$ and view it as a $U(\mathfrak{g})_n$-module.  Any subquotient of the Verma module $M(\lambda)$ has finite length and can be viewed as extension of modules of the form $L(\mu)$. Each $L(\mu)$ is not $1+ \pi U(\mathfrak{g})_n$-torsion by Corollary  \ref{d3LlambdaistorsionfreeCorollary}. As finite extension of modules that are not $1+ \pi U(\mathfrak{g})_n$-torsion  is not $1+ \pi U(\mathfrak{g})_n$-torsion, $M$ is not $1+\pi U(\mathfrak{g})_n$-torsion. Therefore, by Corollary \ref{d3tensorproductU(g)isnonzerocorollary}, we have

        $$\widehat{U(\mathfrak{g})_n} \uset{U(\mathfrak{g})_n} M \neq 0.$$

As this space has no $\pi$-torsion we get
\begin{equation}
\begin{split}
   0 \neq  (\widehat{U(\mathfrak{g})_n} \uset{U(\mathfrak{g})_n} M) \uset{R} K &=(\widehat{U(\mathfrak{g})_n} \uset{R} K) \uset{(U(\mathfrak{g})_n \uset{R} K)} (M \uset{R} K) \\
&=\hugnK \uset{ U(\mathfrak{g}_K)} M.   \qedhere 
\end{split}
\end{equation}
\end{proof}

We are now able to prove the correspondence between the lattice of submodules of $M(\lambda)$ and $\widehat{M(\lambda)}$.

\begin{theorem}
\label{d3affinoidclassicvermacorrespondenceprop}
Let $\lambda$ be a weight in $\pi^n \mathfrak{h}^*$ and extend this to a weight $\lambda \in \mathfrak{h}_K^*$.There is a one to one correspondence between submodules of $\widehat{M(\lambda)}$ and submodules of $M(\lambda)$.

\end{theorem}

\begin{proof}
 Recall the function $\mathscr{F}$ going from submodules of $\widehat{M (\lambda)}$ to submodules of $M(\lambda)$ sending a submodule $N$ to $N \cap M(\lambda)$. We have already proven in Lemma \ref{d3submodulecorrespondenceinjective} that $\mathscr{F}$ is injective, so we only need prove that $\mathscr{F}$ is surjective.
 
 Let $N$ be a submodule of $M(\lambda)$ and let $\overline{N} =\widehat{N}=\hugnK \uset{U(\mathfrak{g}_K)}N$. Furthermore, let $N'=\widehat{N} \cap M(\lambda)$. We aim to prove that $N=N'= \mathscr{F}(\widehat{N})$. By construction we have that $N \subset N'$ and by Proposition \ref{d3submodulesclos}, $\widehat{N}=\hugnK \uset{U(\mathfrak{g}_K)}N'$. Assume for a contradiction that $N$ is strictly included in $N'$. Consider the short exact sequence

           $$0 \to N \to N' \to N'/N \to 0.$$

As $\hugnK$ is flat over $U(\mathfrak{g}_K)$ we get a short exact sequence

          $$0 \to \hugnK \uset{U(\mathfrak{g}_K)} N \to \hugnK \uset{U(\mathfrak{g}_K)} N' \to \hugnK \uset{U(\mathfrak{g}_K)} N/N' \to 0, \text{ so }$$

           $$ 0 \to \widehat{N} \to \widehat{N} \to \hugnK \uset{U(\mathfrak{g}_K)} N/N' \to 0$$

is  a short exact sequence, which implies that $\hugnK \uset{U(\mathfrak{g}_K)} N/N'=0$. Finally, by Proposition \ref{d3tensorproductU(g)Kisnon-zero}, we have $\hugnK  \uset{U(\mathfrak{g}_K)} N/N' \neq 0$, which is the desired contradiction.           
\end{proof}

One might also try to prove the theorem above using \cite[Korollar 1.3.12]{FeLa}; we were not aware of the existence of this paper at the time of the proof.
 
Using the theorem above, we obtain immediately:

\begin{proposition}
\label{d3AffinoidVermafinitelengththeorem}
Let $\lambda: \pi^n \mathfrak{h} \to R$ be an $R$-linear map. The affinoid Verma module $\widehat{M(\lambda)}$ has finite length equal to the length of classical Verma module $M(\lambda)$.
\end{proposition}

\begin{proof}
By Theorem \ref{d3affinoidclassicvermacorrespondenceprop}, there is a one to one correspondence between submodules of $\widehat{M(\lambda)}$ and submodules of $M(\lambda)$. As the module $M(\lambda)$ has finite length by \cite[Theorem 1.11]{Hu1}, it follows that $\widehat{M(\lambda)}$ also has finite length. Furthermore, the correspondence is 1-1, so the lengths must be the same.
\end{proof}

For $\lambda$ as in Theorem \ref{d3affinoidclassicvermacorrespondenceprop} we get the following corollaries:

\begin{corollary}
An affinoid Verma module $\widehat{M(\lambda)}$ is simple if and only if the corresponding classical Verma module $M(\lambda)$ is simple.
\end{corollary}

\begin{corollary}
\label{d3affinoidsimplequotientcorollary}
Any affinoid Verma module has a unique maximal submodule and a unique simple subquotient. 
The unique simple quotient $\widehat{L(\lambda)}$ of $\widehat{M(\lambda)}$ is given by 

                $$\widehat{L(\lambda)}:=\hugnK \uset{U(\mathfrak{g}_K)} L(\lambda),$$
where $L(\lambda)$ denotes the unique simple quotient of $M(\lambda)$.                
\end{corollary}

\begin{proof}

Let $N(\lambda)$ denote the unique maximal submodule of $M(\lambda)$. Consider the short exact sequence

$$ 0 \to N(\lambda) \to M(\lambda) \to L(\lambda) \to 0.$$

Since $\hugnK$ is flat over $U(\mathfrak{g}_K)$ we obtain a short exact sequence

\begin{equation}
\label{d3basechangeVermases}
 0 \to \hugnK \uset{U(\mathfrak{g}_K)} N(\lambda) \to  \hugnK \uset{U(\mathfrak{g}_K)} M(\lambda) \to \hugnK \uset{U(\mathfrak{g}_K)} L(\lambda) \to 0.
\end{equation}

By construction, $\widehat{M(\lambda)} \cong \hugnK \uset{U(\mathfrak{g}_K)} M(\lambda)$ and by Theorem \ref{d3affinoidclassicvermacorrespondenceprop}, one obtains  $\widehat{N(\lambda)}:=\hugnK \uset{U(\mathfrak{g}_K)} N(\lambda)$ is the unique maximal submodule of $\widehat{M(\lambda)}$. The claim now follows from equation \eqref{d3basechangeVermases}. 
\end{proof}

\begin{proposition}
\label{d3compositionseriesaffinoidsubquotientVerma}

Let $\hat{M}$ be a subquotient of $\widehat{M(\lambda)}$. Then $\hat{M}$ has a finite composition series and all the simple quotients are of the form $\widehat{L(\mu)}$ for some $\mu \in \pi^n \mathfrak{h}^*$.

\end{proposition}

\begin{proof}

The first statement follows directly from Proposition \ref{d3AffinoidVermafinitelengththeorem}. It is enough to prove the second statement in the case $\hat{M}=\widehat{M(\lambda)}$. Let

$$ 0=\hat{M_0} \subset \hat{M_1} \subset \hat{M_2} \subset \ldots \hat{M_n}=\widehat{M(\lambda)},$$

be a composition series for $\widehat{M(\lambda)}$. By Theorem \ref{d3affinoidclassicvermacorrespondenceprop}, there exists a composition series of $M(\lambda)$

$$ 0=M_0 \subset M_1 \subset M_2 \subset \ldots M_n= M(\lambda),$$

such that $\hat{M_i}= \hugnK \uset{U(\mathfrak{g}_K)} M_i$ for $0 \leq i \leq n.$  

Fix $1 \leq j \leq n$; it is enough to prove that $\hat{M_j}/\hat{M}_{j-1} \cong \widehat{L(\mu)} $ for some $\mu \in \mathfrak{h}^*$. Consider the short exact sequence:
  $$ 0 \to M_{j-1} \to M_j \to M_j/M_{j-1}.$$  
  
Since $\hugnK$ is flat over $U(\mathfrak{g}_K)$ we obtain by tensoring on the left  a short exact sequence

$$ 0 \to \hat{M}_{j-1} \to \hat{M_j} \to \hugnK \uset{U(\mathfrak{g}_K)} M_j/M_{j-1}, $$

so $\hat{M_j} / \hat{M}_{j-1} \cong \hugnK \uset{U(\mathfrak{g}_K)} M_j/M_{j-1} $. Since $M_j/M_{j-1}$ is a simple subquotient of  $M(\lambda)$, we have $M_j/M_{j-1} \cong L(\mu)$ for some $\mu \in \mathfrak{h}_K^*$ by \cite[Section 1.11]{Hu1}. This is induced by some $R$-linear map $\mu:\pi^n \mathfrak{h} \to R$. The conclusion follows from Corollary \ref{d3affinoidsimplequotientcorollary}.  
\end{proof}

\section{An affinoid equivalence of categories a la Borho-Brylinski}
\label{d3sectionBoBr}

Recall that $G$ is a connected, simply connected smooth affine algebraic group scheme defined over $\Spec R$ with Lie algebra $\mathfrak{g}$. We also let $B$ be a closed subgroup of $G$. Throughout this section we retain Assumption \ref{d3loctrivialassumption}, that is, we assume that the quotient scheme $X=G/B$ is an $R$-variety and the quotient map $d_B:G \to X$ given by $d_B(g)=gB$ is a locally trivial $B$-torsor with respect to the action $\diamond$ given by $b \diamond g=gb^{-1}$.

\subsection{Introduction to \texorpdfstring{$\widehat{\mathcal{D}}$}{hat D}-modules}

We use the following convention, for a sheaf of $R$-modules $\mathcal{M}$, we define its $\pi$-adic completion $\widehat{\mathcal{M}}:= \invlim \mathcal{M}/\pi^i \mathcal{M}$.

Let $Y$ be an $R$-variety and $\mathcal{D}$ be a sheaf of Noetherian rings on $Y$. Since $\pi$-adic completion preserves Noetherianity we obtain that $\widehat{\mathcal{D}}$ is a sheaf of Noetherian rings. Thus, a module $\mathcal{M}$ over $\widehat{\mathcal{D}}$ is coherent if and only if it is locally finitely generated.  Furthermore, we will use without further comments that if $\mathcal{M}$ is a coherent $\widehat{\mathcal{D}}$-module, then $\mathcal{M} \cong \invlim \mathcal{M}/\pi^i \mathcal{M}$; this follows from \cite[Lemma 5.4]{Annals}. We also use that for any $i \in \mathbb{N}^*$, we have $\widehat{\mathcal{D}}/ \pi^i \widehat{\mathcal{D}} \cong \mathcal{D} /\pi^i \mathcal{D}.$ For more background on $\widehat{\mathcal{D}}$-modules, the reader is advised to consult \cite[Section 5]{Annals} and \cite[Section 3]{Berthelot}.

In general, it is hard to determine whether a $\widehat{\mathcal{D}}$-module $\mathcal{M}$ is coherent. This is true for example if $\mathcal{M}=\widehat{\mathcal{N}}$ for some coherent $\mathcal{D}$-module $\mathcal{N}$. In the following, we give a more general set of sufficient conditions.

\begin{proposition}\cite[Lemme 3.2.2]{Berthelot}
\label{d3BerthelotI}
Let $D$ be a ring and $I$ an ideal generated by finitely many central elements, and let $D_i=D/I^iD$, $i \in \mathbb{N}^*$. Furthermore, suppose there exists $(M_i)$  an inverse system of $D_i$-modules such that for $j \geq 2$ the canonical morphisms $M_j/ \pi^{j-1} M_j \to M_{j-1}$ are isomorphisms. We let $M= \invlim M_i$. Then:

\begin{enumerate}
\item{For $i \geq 1$ the canonical morphisms
                $$M/I^iM \to M_i$$
are isomorphisms.                }
\item{If $M_1$ is finitely generated over $D_1$, then $M$ is finitely generated over $\widehat{D}:=\invlim D_i$.} Furthermore, a generating set for $M$ can be obtained by lifting a generating set for $M_1$.
\end{enumerate}
\end{proposition}

\begin{corollary}
\label{d3inversesystemlemma}
Let $Y$ be an $R$-variety and $\mathcal{D}$ a sheaf of Noetherian rings on $Y$. Further, let $(\mathcal{M}_i)$ be an inverse system of coherent modules over $\mathcal{D}/\pi^i\mathcal{D}$ and suppose that the connecting maps induce isomorphisms $\mathcal{M}_i/\pi^{i-1}\mathcal{M}_i \cong \mathcal{M}_{i-1}$ for all $i \geq 2$. Define $$\mathcal{M}:= \invlim \mathcal{M}_i.$$ 

Then $\mathcal{M}$ is a coherent $\widehat{\mathcal{D}}$-module and $\mathcal{M}_i \cong \mathcal{M}/\pi^i \mathcal{M}$ for all $i \geq 1$. 
\end{corollary}

\begin{proof}

The question is local; as $\widehat{\mathcal{D}}$ is a sheaf of Noetherian rings, a module is coherent if and only if it is locally finitely generated. Let $U \subset Y$ be open affine and let $M_{iU}:=\mathcal{M}_i(U)$ and $D_{iU}:=\mathcal{D}(U)/\pi^i \mathcal{D}(U)$, so that $\mathcal{M}(U)= \invlim M_{iU}.$

Since $\mathcal{M}_i$ is coherent as a $\mathcal{D}/\pi^i \mathcal{D}$-module, we get that $M_{iU}$ is a finitely generated $D_{iU}$-module.  By definition, we have $M_{iU}/\pi^{i-1}M \cong M_{i-1U}$, so by the second part of Proposition \ref{d3BerthelotI}, we get that $\mathcal{M}(U)$ is finitely generated as a $\widehat{\mathcal{D}}(U)$-module, so $\mathcal{M}$ is indeed a coherent $\widehat{\mathcal{D}}$-module. For the second part of the statement we have by the first part of Proposition \ref{d3BerthelotI} that $$M_{iU} \cong \mathcal{M}(U) / \pi^i \mathcal{M}(U).$$

As this is true for any open affine and there is a map $\mathcal{M}/\pi^i \mathcal{M} \to \mathcal{M}_i$, we get the desired conclusion.
\end{proof}

\subsection{Pullback of \texorpdfstring{$\widehat{\mathcal{D}}$}{hat D}-modules}

\begin{lemma}
\label{d3preservationofinversesystemlemma}
Let $f: Z \to Y$ be a  map of smooth $R$-varieties and let $\mathcal{M}$ be a quasi-coherent $\mathcal{O}_Y$-module. Then as $\mathcal{O}_Z$-modules we have
              $$f^*(\mathcal{M}/\pi^j \mathcal{M}) \cong f^*(\mathcal{M})/\pi^j f^*(\mathcal{M}), \text{ for any } j \geq 1.$$

\end{lemma}

\begin{proof}
Consider the short exact sequence:

$$   \mathcal{M} \xrightarrow{\cdot \pi^j} \mathcal{M} \rightarrow \mathcal{M}/\pi^j \mathcal{M}      \rightarrow 0   .   $$

The functor $f^*$ is right exact, so applying this to the short exact sequence above we get:

$$  f^*\mathcal{M} \xrightarrow{f^*(\cdot \pi^j)} f^*\mathcal{M} \rightarrow f^*(\mathcal{M}/\pi^j \mathcal{M}) \rightarrow 0.$$
            
Finally, notice that $f^*(\cdot \pi^j)=\cdot \pi^j$, so we get that indeed

\begin{equation*} 
f^*\mathcal{M}/ \pi^j f^*\mathcal{M} \cong f^*(\mathcal{M}/\pi^j \mathcal{M}). \qedhere
\end{equation*}
\end{proof}

For the rest of this section, we fix $n$ a deformation parameter. Let $\mathcal{D}$ be a $\pi^n$-deformed tdo on an $R$-variety $Y$.

\begin{definition}
\label{d3affinoidinverseimagedef}
Let  $f:Z\to Y$ be a  map of smooth $R$-varieties and  let $\mathcal{M}$ be a coherent $\widehat{\mathcal{D}}$-module on $Y$. Then we define the $\pi$-adic pullback of $\mathcal{M}$ to be
                      $$\hat{f}^{\#}(\mathcal{M}):= \invlim f^{\#}(\mathcal{M}/\pi^i \mathcal{M}).$$

\end{definition}

\begin{remark}  The inverse limit is considered in the category of presheaves over $Z$. By construction, we have that $\mathcal{M}_i:=\mathcal{M}/\pi^i \mathcal{M}$ is in particular a $\mathcal{D}$-module, so $f^{\#}( \mathcal{M}_i)$ is a $f^{\#} \mathcal{D}$-module. Since $\pi^i \mathcal{M}_i=0$, we obtain $\pi^i f^{\#}( \mathcal{M}_i)=0$, thus $f^{\#}( \mathcal{M}_i)$ is a $f^{\#} \mathcal{D}/\pi^i f^{\#} \mathcal{D}$-module. Therefore, we obtain that $\hat{f}^{\#}(\mathcal{M})$ has the structure of a $\widehat{f^{\#} \mathcal{D}}$-module.
\end{remark}

Let $L$ be a smooth affine algebraic group locally of finite type defined over $\Spec R$ acting on $Y$ and let $\mathcal{D}$ be a $\pi^n$-deformed $L$-htdo on $Y$. We define the notion of $\hat{L}$-equivariant $\widehat{\mathcal{D}}$-modules.

\begin{definition}
\label{d3piadicequivariantDmod}
A $\hat{L}$-equivariant  coherent $\widehat{{\mathcal{D}}}$-module is a triple $(\mathcal{M},(\mathcal{M}_i),(\alpha_i))$ such that:

\begin{enumerate}
\item{$(\mathcal{M}_i)$ is an inverse system of $\mathcal{D}$-modules  and $\pi^i \mathcal{M}_i=0$.} 
\item{For $i \in \mathbb{N}^*$, $(\mathcal{M}_i,\alpha_i) \in \Coh(\mathcal{D},L)$.}
\item{ For  $i \geq 2$, the connecting map in the inverse system induces an isomorphism $\mathcal{M}_i/\pi^{i-1} \mathcal{M}_i \cong \mathcal{M}_{i-1}$ of $L$-equivariant $\mathcal{D}$-modules.}
\item{$\mathcal{M} \cong \invlim \mathcal{M}_i$ as $\widehat{\mathcal{D}}$-modules.}

\end{enumerate}

A $\hat{L}$-equivariant morphism between $\hat{L}$-equivariant $\widehat{\mathcal{D}}$-modules $(\mathcal{M},(\mathcal{M}_i),(\alpha_{i}))$ and \\
$  (\mathcal{N}, (\mathcal{N}_i), (\beta_{i}))$ is a $\widehat{\mathcal{D}}$-linear morphism $\phi: \mathcal{M} \to \mathcal{N}$ such that there exist compatible maps $\phi_i \in \Hom_{\Coh(\mathcal{D},L)}(\mathcal{M}_i,\mathcal{N}_i)$ with $\phi=\invlim \phi_i$.

We define the category of $\hat{L}$-equivariant $\widehat{\mathcal{D}}$-modules to consist of $\hat{L}$-equivariant objects and $\hat{L}$-equivariant morphisms. As before, we will omit the equivariance structure when it is understood from the context. We denote $\Coh(\widehat{\mathcal{D}},L)$ the category of $\hat{L}$-equivariant coherent $\widehat{\mathcal{D}}$-modules.

\end{definition}

\begin{proposition}
\label{d3categoryCohhatDLisAbelian}
Let the notation be as above. The category $\Coh(\widehat{\mathcal{D}},L)$ is Abelian.

\end{proposition}

To prove this proposition, we will need the following lemma:

\begin{lemma}
\label{d3kernelsinversesytemsexist}

Let $A$ be a $\pi$-adically complete Noetherian $R$-algebra. Let $(M_i)_{i \in \mathbb{N}^*}$ and $(N_i)_{i \in \mathbb{N}^*}$ be inverse systems of $A$-modules such that $\pi^i M_i=\pi^i N_i=0$ for all $i \in \mathbb{N}^*$ and assume that transition maps induce isomorphisms $M_i/\pi^{i-1}M_i \cong M_{i-1}$ and $N_i/\pi^{i-1}N_i \cong N_{i-1}$. Let $(f_i):(M_i) \to (N_i)$ be a map of inverse systems and $(K_i)=\ker(f_i)$. Then
$K_i/\pi^{i-1}K_i \cong K_{i-1}$.

\end{lemma}

\begin{proof}

We follow the idea in \cite[\href{https://stacks.math.columbia.edu/tag/087X}{087X}]{StackProject}. Let $M:=\invlim M_i$, $N:=\invlim N_i$ and $f:M \to N$ the induced map;  further let $K=\ker(f)$. We have by Proposition \ref{d3BerthelotI} that for any $j \in \mathbb{N}^*$, $M_j \cong M/\pi^j M$ and $N_j \cong N/ \pi^j N$, so we may assume that the map $f_j:M/\pi^j M \to N/\pi^j N$ is given by $f_j(m+\pi^j M)=f(m)+\pi^j N$ for all $m \in M$.

Next, we know by \cite[3.2.3i)]{Berthelot} that there exists $c \in \mathbb{N}$ such that for $n \geq c$, we have $\pi^n N\cap f(M) \subset \pi^{n-c} f(M)$. In particular, we obtain: 

\begin{equation}
\label{d3eqinvimagepiton}
f^{-1}(\pi^n N) \subset K+\pi^{n-c} M.
\end{equation}

For $s,t \in \mathbb{N}, s \geq t$, we let $K'_{s,t}:=\im(\ker(f_s) \to M_t)$. We claim that for a fixed $t$, $K'_{s,t}$ is eventually constant and we denote $K'_{t}$ this value. We have that for $s \geq t+c$

\begin{equation}
\begin{split}
K'_{s,t}&=f^{-1} (\pi^s N) + \pi^t M/\pi^t M  \\
        &=K+\pi^t M/\pi^t M \text{ (by equation \eqref{d3eqinvimagepiton}) }\\
        &\cong K/K \cap \pi^t M.
\end{split}
\end{equation}

Therefore $K'_t=K/K \cap \pi^t M$ is the constant value we seek. We claim that for any $n \in \mathbb{N}$ the system $(K'_t/\pi^n K'_t)_{t \geq n}$ is eventually constant with value $K/\pi^n K$. Again, we have by \cite[3.2.3i)]{Berthelot} that there exists $d \in \mathbb{N}$ such that

\begin{equation}
\label{d3equationKerneppiU}
K \cap \pi^u M \subset \pi^{u-d} \text{ for any $u \geq d$.}    
\end{equation}

Therefore we obtain that for $t \geq n+d$

\begin{equation}
\begin{split}
 K'_t/\pi^n K'_t &\cong K/K \cap \pi^t M /(\pi^n K/K \cap \pi^t M)\\
            &\cong K/(K \cap \pi^t M +\pi^n K) \\
            &\cong K/\pi^n K \text{ (by equation \eqref{d3equationKerneppiU})}.
\end{split} 
\end{equation}

Finally, to prove that $K/\pi^n K \cong K_n$ for all $n\in \mathbb{N}$, we repeat the argument in \cite[\href{https://stacks.math.columbia.edu/tag/087X}{087X}]{StackProject} to prove that the inverse system $(K/\pi^i K)$ is indeed the the kernel of $(f_i)$.
\end{proof}

\begin{proof}[Proof of Proposition \ref{d3categoryCohhatDLisAbelian}]
We have by \cite[Section 9]{Sta1} that the category $\Coh(\mathcal{D},L)$ is Abelian. We view $\Coh(\widehat{\mathcal{D}},L)$ as a full subcategory of the Abelian category of towers consisting of objects in $\Coh(\mathcal{D},L)$. It is easy to see that $0 \in  \Coh(\widehat{\mathcal{D}},L)$ and the category is closed under direct sums. Therefore, we only need to prove that $\Coh(\widehat{\mathcal{D}},L)$ is closed under kernels and cokernels.

Let $\phi:(\mathcal{M},(\mathcal{M}_i),(\alpha_i)) \to (\mathcal{N},(\mathcal{N}_i),(\beta_i))$ be a map of objects in $\Coh(\widehat{\mathcal{D}},L)$. For $i \in \mathbb{N}^*$, let $\phi_i:\mathcal{M}_i \to \mathcal{N}_i$ be the corresponding map and $\mathcal{K}_i=\ker(\phi_i)$. Since $\Coh(\mathcal{D},L)$ is Abelian, we have $\mathcal{K}_i \in \Coh(\mathcal{D},L)$; further by construction we have $\pi^i \mathcal{K}_i=0$ and that $(\mathcal{K}_i)$ forms an inverse system of $\mathcal{D}$-modules. Finally, by working locally and using Lemma \ref{d3kernelsinversesytemsexist}, we obtain that for any $i \in \mathbb{N}^*,$ $\mathcal{K}_i/ \pi^{i-1} \mathcal{K}_i \cong \mathcal{K}_{i-1}$, so $\mathcal{K}=\ker \phi=\invlim \mathcal{K}_i \in \Coh(\widehat{\mathcal{D}},L)$; the coherence of $\mathcal{K}$ follows form Corollary \ref{d3inversesystemlemma}.

A similar argument proves that $\Coh(\widehat{\mathcal{D}},L)$ is closed under cokernels.
\end{proof}

Recall that $i_l:X \to X \times X$ denotes the inclusion of $X$ into the left copy of $X \times X$. Further, recall from Theorem \ref{d3alaBorho-Brylinski}     that for a $\pi^n$-deformed $G$-equivariant htdo on $X \times X$, the functor $i_l^{\#}$ induces an equivalence of categories between $\Coh(\mathcal{D},G)$ and $\Coh(i_l^{\#} \mathcal{D},B)$. We denote $\mathscr{H}_l$ the quasi-inverse of $i_l^{\#}$.

\begin{proposition}
\label{d3piadicBorhoBrylinski}
Let $\mathcal{D}$ be a $\pi^n$-deformed $G$-equivariant htdo on $X \times X$. The functor $\hat{i_l}^{\#}$ induces an equivalence of categories between $\Coh(\widehat{\mathcal{D}},G)$ and $\Coh(\widehat{i_l^{\#} \mathcal{D}},B)$. A quasi-inverse is given $\widehat{\mathscr{H}_l}$ defined by $\widehat{\mathscr{H}_l}(\mathcal{N}):= \invlim \mathscr{H}(\mathcal{N}/\pi^i \mathcal{N})$ for $\mathcal{N} \in \Coh(\widehat{i_l^{\#} \mathcal{D}},B)$.

\end{proposition}

\begin{proof}

Let $\mathcal{M} \in \Coh(\widehat{\mathcal{D}},G)$ and $\mathcal{M}_i:= \mathcal{M}/\pi^i \mathcal{M}$ for $i \geq 1$. By construction $\mathcal{M}_i \in \Coh(\mathcal{D},G)$, so applying Theorem \ref{d3alaBorho-Brylinski}, we obtain $\mathcal{N}_i:=i_l^{\#} \mathcal{M}_i \in \Coh(i_l^{\#} \mathcal{D},B)$. Further, we have $\pi^i \mathcal{N}_i=0$ since $\pi^i \mathcal{M}_i=0$. By Lemma \ref{d3preservationofinversesystemlemma}, $\mathcal{N}_i/\pi^{i-1} \mathcal{N}_i \cong \mathcal{N}_{i-1}$. Therefore, we obtain by Corollary \ref{d3inversesystemlemma} that $$\widehat{\mathcal{N}}:=\hat{i_l}^{\#} \mathcal{M} = \invlim \mathcal{N}_i \in \Coh(\widehat{i_l^{\#} \mathcal{D}},B). $$

We have by Corollary \ref{d3inversesystemlemma} that $\mathcal{N}_i \cong \mathcal{N}/\pi^i \mathcal{N}$, so we get:

\begin{equation}
\begin{split}
\widehat{\mathscr{H}_l} \circ \hat{i_l}^{\#} (\mathcal{M})&= \widehat{\mathscr{H}_l} (\mathcal{N}) \\
  &\cong \invlim \mathscr{H}_l (\mathcal{N}_i) \\
  &\cong \invlim \mathcal{M}_i \text{ (by Theorem \ref{d3alaBorho-Brylinski}) } \\
  &\cong \mathcal{M}.
\end{split}
\end{equation} 

Thus $\widehat{\mathscr{H}_l}$ is a left quasi-inverse for $\hat{i_l}^{\#}.$ A similar argument shows that $\widehat{\mathscr{H}_l}$ is also a right quasi-inverse.
\end{proof}

\subsection{Some category theory lemmas}

To prove an affinoid version of the Borho-Brylinski theorem, we need some lemmas for $R$-linear Abelian categories.

Throughout this subsection we fix $\mathcal{A}$  an $R$-linear small Abelian category and let $\mathcal{B}$ be the full $\mathcal{A}$-subcategory of \emph{$\pi$-torsion elements}, i.e. $\ob(\mathcal{B})= \{ A \in \mathcal{A} | \quad \pi^n \id_{A}=0 \text{ for some } n \in \mathbb{N} \}$ (here $\id_{A}$ denotes the identity morphism going from $A$ to $A$). We also call a morphism $f \in \Hom(A,B)$ \emph{$\pi$-torsion} if there exists $n \in \mathbb{N}$ such that $\pi^n f=0$.

 Throughout this subsection we use that in an $R$-linear category, we have for $f \in \Hom(A,B) $, $g \in \Hom(B,C)$ and $r \in R$

\begin{equation*}
         r (g \circ f)= (rg) \circ f= g \circ (rf).
\end{equation*}         

Define a new category $\mathcal{A}_K$, where $\ob(\mathcal{A}_K)=\ob(\mathcal{A})$ and Hom sets given by $\Hom_{\mathcal{A}_K}(M,N):=\Hom_\mathcal{A}( M,N) \uset{R} K$, for all $M,N \in  \ob(\mathcal{A})$. Furthermore, denote $\mathcal{F}$ the natural functor $\mathcal{A} \to \mathcal{A}_K$. 

The aim of this subsection is to establish the following theorem:

\begin{theorem}
\label{d3R-linearequivalencetheorem}
There exists an equivalence of categories between the quotient category $\mathcal{A}/\mathcal{B}$ and the category $\mathcal{A}_K$.

\end{theorem}

One should notice that apriori it is not clear why the quotient category $\mathcal{A}/\mathcal{B}$ is well-defined, so we should begin by proving that $\mathcal{B}$ is a Serre subcategory of $\mathcal{A}$. We start by proving a very useful lemma:

\begin{lemma}
\label{d3pitorsionmorphism}

Let $B \in \mathcal{B}$, $C \in \mathcal{A}$ and consider morphisms $f \in \Hom(B,C)$ and $g \in \Hom(C,B)$. Then $f$ and $g$ are $\pi$-torsion.

\end{lemma}

\begin{proof}

Since $B \in \mathcal{B}$, there exists $n \in \mathbb{N}$ such that $\pi^n \id_B=0$. We have 
 
  $$\pi^nf=\pi^n (\id_B \circ f)= (\pi^n \id_B) \circ f=0,$$

so $f$ is indeed $\pi$-torsion. A similar argument shows that $g$ is also $\pi$-torsion.  
\end{proof}

\begin{proposition}
The category $\mathcal{B}$ is a Serre subcategory of $\mathcal{A}$.

\end{proposition}

\begin{proof}

Consider a short exact sequence: 

 $$ 0 \to  A \xrightarrow{f} B \xrightarrow{g} C \to 0.$$

One needs to prove that $B \in \mathcal{B}$ if and only if $A,C \in \mathcal{B}
$. 

First assume that $B \in \mathcal{B}$. By Lemma \ref{d3pitorsionmorphism}, $f$ is $\pi$-torsion so there exists $n \in \mathbb{N}$ such that $\pi^n f=0$, so that
        
               $$0= \pi^n f =\pi^n ( f \circ \id_A)= f \circ \pi^n \id_A.$$        

Since $f$ is a monomorphism, we can left cancel to get $\pi^n \id_A=0$, so $A \in \mathcal{B}$. 

By Lemma \ref{d3pitorsionmorphism}, $g$ is $\pi$-torsion, so there exists $n \in \mathbb{N}$ such that $\pi^n g=0$, so that
            
            $$0= \pi^n g = \pi^n (\id_C \circ g) = \pi^n \id_C \circ g.$$
            
As $g$ is an epimorphism, we can right cancel to obtain  $\pi^n \id_C=0$, so $C \in \mathcal{B}$.

Now assume that $A,C \in \mathcal{B}$. By Lemma \ref{d3pitorsionmorphism}, $f,g$ are $\pi$-torsion so there exist $n_1,n_2 \in \mathbb{N}$ such that $\pi^{n_1} f=\pi^{n_2}g =0$. Let $n=\text{max}(n_1,n_2)$ and $h:= \pi^n \id_B$. We have

\begin{equation}
\begin{aligned}
0&=\pi^n f= \pi^n (\id_B \circ f)= (\pi^n \id_B) \circ f= h \circ f. \\
0&=\pi^n g= \pi^n (g \circ \id_B)= g \circ (\pi^n  \id_B) = g \circ h.
\end{aligned}
\end{equation}

Since $h \in \Hom(B,B)$, we have by Lemma \ref{d3h2iszero} below that $h^2=0$, so $\pi^{2n} \id_B=0$. Thus $B \in \mathcal{B}$.
\end{proof}

\begin{lemma}
\label{d3h2iszero}
Let $\mathcal{C}$ be a small Abelian category and let 
$$ 0 \to  A \xrightarrow{f} B \xrightarrow{g} C \to 0$$

be a short exact sequence. Let $h \in \Hom(B,B)$ such that $ h \circ f = g \circ h=0$. Then $h^2=0$.
\end{lemma}

\begin{proof}

By the Freyd-Mitchell embedding we may assume that $\mathcal{C}=S$-mod for some ring $S$. In particular, we may assume that $A,B$ and $C$ are Abelian groups. Let $b \in B$; then $g(h(b))=0$, so $h(b) \in \ker(g)=\im(f)$. Thus, there exists $a \in A$ with $f(a)=h(b)$. Then 

$$ 0=h(f(a))=h(h(b)),$$

proving that  $h^2=0$.
\end{proof}

Let $S$ be the collection of $\mathcal{B}$-isomorphisms, i.e. morphisms $f$ in $\mathcal{A}$ such that $\ker(f)$ and coker$(f)$ are in $\mathcal{B}$. Then $S$ is a multiplicative system in the sense defined in \cite[Appendix II]{Wei}. Furthermore, by \cite[Example A.1.2]{Wei}, the quotient category $\mathcal{A}/\mathcal{B}$ is equivalent to the localised category $\mathcal{A}_S$. Denote $\loc: \mathcal{A} \to \mathcal{A}_S$ the localisation functor.

\begin{proof}[Proof of Theorem \ref{d3R-linearequivalencetheorem}]

By the discussion above, it is enough to prove that there exists an equivalence of categories between $\mathcal{A}_S $ and $ \mathcal{A}_K$. By construction, we have that for any $s \in S$, $\mathcal{F}(s)$ is an isomorphism, so by the universal property of localisation there exists a unique functor $\mathcal{G}: \mathcal{A}_S \to \mathcal{A}_K$ defined by $\mathcal{G}(s^{-1}f)= \mathcal{F}(s)^{-1} F(f)$ for any $s^{-1}f$ in $\Hom_{\mathcal{A}_S}(X,Y)$. 

We claim that $\mathcal{G}$ is an equivalence of categories. It is clear that $\mathcal{G}$ is essentially surjective, so we need to prove that it is fully faithful.

Let $\phi \in \Hom_{\mathcal{A}_K}(A,B)= \Hom_{\mathcal{A}}(A,B)\uset{R} K$. Then there exists $n \in \mathbb{N}$ such that $ \phi = f \otimes \pi^{-n}$ for some $f \in \Hom_{\mathcal{A}}(A,B)$. By construction, we have that $\pi^{n} \id_B \in S$, so we get that

\begin{equation}
\begin{aligned}
\mathcal{G}((\pi^n \id_B)^{-1}  f)&=\mathcal{F}(\pi^n \id_B)^{-1} \circ  \mathcal{F}(f) \\
&=(\id_B \otimes \pi^{-n}) \circ ( f \otimes 1) \\
&=f \otimes \pi^{-n} \\
&= \phi.
\end{aligned}
\end{equation}                                                                                                                                              Thus, $\mathcal{G}$ is indeed full. Lastly, we need to prove that $\mathcal{G}$ is faithful. As all the categories involved are Abelian it is enough to prove that for $s^{-1}f \in \Hom_{\mathcal{A}_{\mathcal{S}}}(X,Y)$, if $\mathcal{G}(s^{-1}f)=0$, then $s^{-1}f=0$. Here we assume $s \in \Hom_\mathcal{A}(X',X)$, $s \in S$ and $f \in \Hom_\mathcal{A}(X',Y)$. We have $0=\mathcal{G}(s^{-1}f)$= $\mathcal{F}(s)^{-1} \circ \mathcal{F}(f)$, so $F(f)=0$. Therefore, we get that $f$ is $\pi$-torsion, so there exists $n \in \mathbb{N}$ such that $\pi^n f=0$. Then:

          $$f \circ \pi^n \id_Y = \pi^n f \circ \id_Y= 0,$$

and since $\pi^n \id_Y \in S$, we obtain by \cite[Lemma 2.1.5]{Mil}  that $s^{-1}f=0$. Thus, $\mathcal{G}$ is indeed faithful.
\end{proof}

We finish the subsection by proving a categorical proposition that we will need in the next subsection.

\begin{proposition}
\label{d3equivalenceofquotientcategories}
Let $\mathcal{F}: \mathcal{A} \to \mathcal{B}$ be an equivalence of Abelian categories. Let $\mathcal{C}$ and $\mathcal{D}$ be Serre subcategories of $\mathcal{A}$ and $\mathcal{B}$, respectively such that $\mathcal{F}$ restricts to an equivalence $\mathcal{F}: \mathcal{C} \to \mathcal{D}$. Then $\mathcal{F}$ induce an equivalence between the quotient categories $\mathcal{A}/\mathcal{C}$ and $\mathcal{B}/\mathcal{D}$.

\end{proposition}

\begin{proof}
Let $q_{\mathcal{A}}:\mathcal{A} \to \mathcal{A}/\mathcal{C}$ and $q_{\mathcal{B}}:\mathcal{B} \to \mathcal{B}/\mathcal{D}$ denote the localisation functors and let $\mathcal{H}:=q_{\mathcal{B}} \circ \mathcal{F}$. By assumptions, we have $\ker \mathcal{H}=\mathcal{C}$, so by \cite[Exercise 5, Section 4.4]{Popescu}, there exists a faithful and exact functor $\overline{\mathcal{H}}:\mathcal{A}/\mathcal{C}$ such that $\overline{\mathcal{H}} \circ q_{\mathcal{A}}=\mathcal{H}=q_{\mathcal{B}} \circ \mathcal{F}$. Since $q_{\mathcal{B}} \circ \mathcal{F}$ is essentially surjective, we obtain that $\overline{\mathcal{H}}$ is also essentially surjective. Finally, for any morphism $f$ in $\mathcal{B}/\mathcal{D}$, there is a morphism $g$ in $\mathcal{A}/\mathcal{C}$, such that $\overline{\mathcal{H}}(g)=f$, so $\overline{\mathcal{H}}$ is also full.
\end{proof}

\subsection{Affinoid equivariant equivalence a la Borho-Brylinski}

Let $Y$ be an $R$-variety and $L$ a smooth affine algebraic group locally of finite type defined over $R$ and $\mathcal{D}$ a sheaf of $\pi^n$-deformed $L$-equivariant htdo on $Y$. Recall that by Proposition \ref{d3categoryCohhatDLisAbelian} the category of $\hat{L}$-equivariant coherent $\widehat{\mathcal{D}}$-modules, $\Coh(\widehat{\mathcal{D}},L)$, is Abelian.

\begin{definition}

Let $Y$ be a quasi-compact $R$-variety and $L$ an algebraic group acting on $Y$. Let $\Coh( \widehat{\mathcal{D}}, L)^\pi$ be the full subcategory of $\Coh( \widehat{\mathcal{D}_Y}, L) $ consisting of $\pi$-torsion objects. As $Y$ is quasi-compact this is equivalent to the full subcategory of $\Coh( \widehat{\mathcal{D}_Y}, L)$ such that all the sections are $\pi$-torsion.
\end{definition}

\begin{proposition}
\label{d3quotientcategoryequivalence}

There is an equivalence of categories between the quotient category\\ $\Coh( \widehat{\mathcal{D}}, L)  / \Coh( \widehat{\mathcal{D}}, L)^{\pi}  $ and the category $\Coh( \widehat{\mathcal{D}}, L)_K $.

\end{proposition}

\begin{proof}

This follows directly from Theorem \ref{d3R-linearequivalencetheorem}.
\end{proof}

\begin{definition}

Let $Y$ be a $R$-variety; recall that $\widehat{\mathcal{D}_{K}}= \widehat{\mathcal{D}} \uset{R} K$. A coherent $\hat{L}_K$-equivaraint $\widehat{\mathcal{D}_{K}}$-module is quadruple ($\mathcal{M},\mathcal{M}_0,(\mathcal{M}_i),(\alpha_i))$ such that $\mathcal{M}_0$ is a lattice of $\mathcal{M}$ and $(\mathcal{M}_0,(\mathcal{M}_i),(\alpha_i)) \in \Coh(\widehat{\mathcal{D}},L)$. 

Let ($\mathcal{M},\mathcal{M}_0,(\mathcal{M}_i),(\alpha_i))$ and ($\mathcal{N},\mathcal{N}_0,(\mathcal{N}_i),(\beta_i))$ be $\hat{L}_K$-equivariant $\widehat{\mathcal{D}_{K}}$-modules and let $\phi: \mathcal{M} \to \mathcal{N}$ be a $\widehat{\mathcal{D}_{K}}$-linear morphism. We have by the proof of \cite[Proposition 3.4.5]{Berthelot} that

$$ \Hom_{\widehat{\mathcal{D}}} (\mathcal{M}_0, \mathcal{N}_0) \uset{R} K  \cong  \Hom_{\widehat{\mathcal{D}_{K}}} (\mathcal{M}, \mathcal{N}),$$

so there exists a pair ($\phi_0,x)$, $\phi_0 \Hom_{\widehat{\mathcal{D}}} (\mathcal{M}_0, \mathcal{N}_0)$ and $x \in K$ such that $\phi= \phi_0 \otimes x$. We say that $\phi$ is $\hat{L}_K$-equivariant if $\phi_0$ is $\hat{L}$-equivariant.

We denote $\Coh(\widehat{\mathcal{D}_{K}}_,L)$ the category of coherent $\widehat{\mathcal{D}_{K}}$-modules consisting of $\hat{L}_K$-equivariant objects together with $\hat{L}_K$-equivariant morphisms. 
\end{definition}

We will ignore the equivariance structure when it is well understood from the context and just call $\mathcal{M}$ an $\hat{L}_K$-equivariant $\widehat{\mathcal{D}_K}$-module.

\begin{lemma}
\label{d3tensorKequiv}

Assume that $Y$ is quasi-compact. Then there exists an explicit equivalence of categories between $\Coh(\widehat{\mathcal{D}},L)_K$ and $\Coh(\widehat{\mathcal{D}_{K}},L)$.

\end{lemma}

\begin{proof}
Define $\mathcal{F}: \Coh(\widehat{\mathcal{D}},L)_K \to\Coh(\widehat{\mathcal{D}_{K}},L)$ by $F(\mathcal{M})= \mathcal{M} \uset{R} K$ for any object $\mathcal{M} \in \Coh(\widehat{\mathcal{D}},L)_K$ and 

 $$\mathcal{F} ( f \otimes x)= f \otimes x \text{, for all } f  \otimes x \in \Hom(\mathcal{M},\mathcal{N}) \uset{R} K.$$
 
By construction, it is clear that $F$ is essentially surjective and since the tensors in $\Hom(M,N) \uset{R} K$ are all pure, $ \mathcal{F}$ is also faithful. Furthermore, it follows by definition of the morphisms in $\Coh(\widehat{\mathcal{D}_{K}},H)$ that $\mathcal{F}$ is also full.
\end{proof}

Until the end of the section, we assume that $\mathcal{D}$ is a $\pi^n$-deformed $G$-equivariant htdo on $X \times X$.

\begin{lemma}
\label{d3restrictionBoBr}

The functor $\hat{i_l}^{\#}$ in Proposition \ref{d3piadicBorhoBrylinski} restricts to an equivalence between $\Coh(\widehat{\mathcal{D}},G)^\pi$ and $\Coh(\widehat{i_l^{\#}\mathcal{D}},B)^\pi$. A quasi-inverse is given $\widehat{\mathscr{H}_l}.$

\end{lemma}

\begin{proof}
Let $\mathcal{M} \in \Coh(\widehat{\mathcal{D}},G)^\pi$ and define $\mathcal{M}_i:= \mathcal{M}/\pi^i \mathcal{M}$. By definition, there exists $m \in \mathbb{N}^*$ such that for $j \geq m$, $\mathcal{M}_j=\mathcal{M}$. Let $\mathcal{N}_i=i_l^{\#} \mathcal{M}_i$;
we have $\hat{i_l}^{\#} \mathcal{M} = \invlim \mathcal{N}_i$ and by Corollary \ref{d3inversesystemlemma}, $\mathcal{N}_i=\hat{i_l}^{\#} \mathcal{M}/ \pi^i \hat{i_l}^{\#} \mathcal{M}$. Further by construction, we have that for $j \geq m$, $\mathcal{N}_j=i_l^{\#} \mathcal{M}$, therefore $\hat{i_l}^{\#} \mathcal{M} \in \Coh(\widehat{i_l^{\#}\mathcal{D}},B)^\pi$.

An analogous argument proves that for $\mathcal{N} \in \Coh(\widehat{i_l^{\#}\mathcal{D}},B)^\pi$, we have $\widehat{\mathscr{H}_l}(\mathcal{N}) \in \Coh(\widehat{\mathcal{D}},G)^\pi$. The conclusion follows from Proposition \ref{d3piadicBorhoBrylinski}.
\end{proof}

\begin{theorem}
\label{d3affinoidBorhoBrylinskitheorem}

There is an equivalence of categories between $\Coh(\widehat{\mathcal{D}_K},G)$ and $\Coh(\widehat{i_l^{\#}\mathcal{D}_{K}},B)$. 

\end{theorem}

\begin{proof}

To simplify the proof, we use $\cong$ to denote an equivalence of categories. Since $G$ is affine and the quotient map $G \to G/B$ is surjective, we obtain that  $X$ is quasi-compact, thus so is  $X \times X$. We have by Lemma \ref{d3tensorKequiv}: 

\begin{equation}
\label{d3affinoidBoBr1}
 \Coh(\widehat{\mathcal{D}_{K}},G) \cong \Coh(\widehat{\mathcal{D}},G)_K   \text{ and } \Coh(\widehat{i_l^{\#}\mathcal{D}_{K}},B) \cong \Coh(\widehat{i_l^{\#}\mathcal{D}},B)_K.
 \end{equation}

Furthermore, we have by Proposition \ref{d3quotientcategoryequivalence} that 

\begin{equation} 
\label{d3affinoidBobr2}
\begin{split}
\Coh(\widehat{\mathcal{D}},G)_K &\cong \Coh(\widehat{\mathcal{D}},G)/ \Coh(\widehat{\mathcal{D}},G)^ \pi,  \\
\Coh(\widehat{i_l^{\#}\mathcal{D}},B)_K &\cong \Coh(\widehat{i_l^{\#}\mathcal{D}},B)/ \Coh(\widehat{i_l^{\#} \mathcal{D}},B)^\pi.
\end{split} 
\end{equation}

Next, we have by Proposition \ref{d3piadicBorhoBrylinski} that there is an equivalence of categories $\hat{i_l}^{\#}: \Coh(\widehat{\mathcal{D}},G) \cong \Coh(\widehat{i_l^{\#} \mathcal{D}},B)$ and by Lemma \ref{d3restrictionBoBr} this restricts to an equivalence $ \Coh(\widehat{\mathcal{D}},G)^\pi \cong \Coh(\widehat{i_l^{\#}\mathcal{D}},B)^\pi$, so applying Proposition \ref{d3equivalenceofquotientcategories}, we obtain an equivalence between the quotient categories:  

\begin{equation} 
\label{d3affinoidBoBr3}
  \Coh(\widehat{\mathcal{D}},G)/ \Coh(\widehat{\mathcal{D}},G)^ \pi  \cong   \Coh \widehat{(i_l^{\#} \mathcal{D}},B)/ \Coh( \widehat{i_l^{\#} \mathcal{D}},B)^\pi.            \end{equation}

Therefore, by combining equations \eqref{d3affinoidBoBr1}, \eqref{d3affinoidBobr2} and \eqref{d3affinoidBoBr3}, we get 

\begin{equation*}
 \Coh(\widehat{\mathcal{D}_{K}},G) \cong \Coh(\widehat{i_l^{\#} \mathcal{D}_{K}},B). \qedhere
\end{equation*}
\end{proof}

\begin{remark}
\label{d3definitionofaffinoidpullback}
Denote $\hat{i}^{\#}_{l,K}$ the equivalence functor from the category $\Coh(\widehat{\mathcal{D}_{K}},G)$ to the category $\Coh(i_l^{\#} \widehat{\mathcal{D}_{K}},B)$. Let $\mathcal{M} \in \Coh(\widehat{\mathcal{D}_{K}},G)$ and let $\mathcal{M}_0$ be the corresponding lattice of $\mathcal{M}$. Then under the equivalence of categories above we have that 
  
   $$\hat{i}^{\#}_{l,K} \mathcal{M} = (\hat{i_l}^{\#} \mathcal{M}_0) \uset{R} K.$$
\end{remark}

Let us finish the section by proving an affinoid version of Corollary \ref{d3classicalzeroglosection}.

\begin{corollary}
\label{d3affinoidzeroglosection}
Let $\mathcal{M} \in \Coh(\widehat{\mathcal{D}_{K}},G) $ and assume that $\Gamma(X,\hat{i}^{\#}_{l,K} \mathcal{M})=0$. Then $\Gamma(X, \mathcal{M})=0$.
\end{corollary}

\begin{proof}
Let $\mathcal{M}_0$ be the corresponding lattice of $\mathcal{M}$ and define $\mathcal{M}_i:=\mathcal{M}_0/ \pi^i \mathcal{M}_0$ and $\mathcal{N}_i:=i_l^{\#} \mathcal{M}_i$. By construction, we have $\mathcal{N}:=\hat{i_l}^{\#} \mathcal{M}= \invlim \mathcal{N}_i$ and by Corollary \ref{d3inversesystemlemma}, $\mathcal{N}_i= \mathcal{N}/\pi^i \mathcal{N}$. 

By assumption, we know that $\Gamma(X,\hat{i}^{\#}_{l,K} \mathcal{M})=\Gamma(X, \mathcal{N}) \uset{R} K=0$. Since $\mathcal{N} \in \Coh(\widehat{\mathcal{D}},L)$, the sections of $\mathcal{N}$ are finitely generated over $\widehat{\mathcal{D}}$; in particular, there exists $m \in \mathbb{N}$ such that $\pi^m \Gamma(X, \mathcal{N})=0$, so $\Gamma(X,\pi^ m \mathcal{N})=0.$. Since $\Gamma(X, \mathcal{N})= \invlim \Gamma(X,\mathcal{N}_i)$, we obtain that for $j \geq m$, $\Gamma(X,\mathcal{N}_j)=\Gamma(X,\mathcal{N})$, so $\Gamma(X,\pi^m \mathcal{N}_j)=0$. Therefore, by applying Corollary \ref{d3classicalzeroglosection}, we obtain $\Gamma(X \times X, \pi^m \mathcal{M}_j)=0$ for $j \geq m$, so $\pi^m \Gamma(X \times X, \mathcal{M}_j)=0$. Since $\Gamma(X \times X,\mathcal{M}_0)=\invlim \Gamma(X \times X, \mathcal{M}_j)$, we conclude that $\pi^m \Gamma(X \times X,\mathcal{M}_0)=0$, so $\Gamma(X \times X, \mathcal{M})=\Gamma(X \times X,\mathcal{M}_0) \uset{R} K=0$.
\end{proof}


\section{Affinoid equivariant Beilinson-Bernstein localisation}
\label{d3sectionaffinoidequivariantbb}

Throughout this section we let $G$ be a connected, simply connected, smooth affine algebraic group scheme locally of finite type defined over $\Spec R$ and we let $\mathfrak{g}=\Lie(G)$ be its Lie algebra. We also let $X$ be a quasi-compact $R$-variety on which $G$ acts.

\subsection{Affinoid localisation mechanism}

We fix $L$ a closed subgroup of $G$ and $n$ a deformation parameter. We denote $U(\mathfrak{g})_n$ the $\pi^n$-th deformation of $U(\mathfrak{g})$ and we let $(\mathcal{D},i_{\mathfrak{g}})$ be a $\pi^n$-deformed $L$-htdo on $X$. Throughout this section we also make the following assumption:

\begin{assumption}
\label{d3assumptionglobalsectioncoherent}
Let $\mathcal{M}$ be a coherent $\mathcal{D}$-module. Then $\Gamma(X, \mathcal{M})$ is a finitely generated $U(\mathfrak{g})_n$-module.
\end{assumption}

Recall that $\widehat{U(\mathfrak{g})_n}$ denotes the $\pi$-adic completion of $U(\mathfrak{g})_n$; further we denoted $\widehat{U(\mathfrak{g})_{n,K}}:=\widehat{U(\mathfrak{g})_n} \uset{R} K$. Similar to Definition \ref{d3piadicequivariantDmod}, we define the notion $\hat{L}$-equivariant $\widehat{U(\mathfrak{g})_n}$-modules by extending Definition \ref{d3nonequivariantLUgmodule}.

\begin{definition}
A $\hat{L}$-equivariant $\widehat{U(\mathfrak{g})_n}$-module is quadruple $(M,(M_i),(\alpha_i),(\rho_i))$ such that $M$ is a finitely generated $\widehat{U(\mathfrak{g})_n}$-module, $(M_i)$ is an inverse system of $U(\mathfrak{g})_n$-modules and

\begin{itemize}
\item{$(M_i,\alpha_i,\rho_i)$ is a finitely generated $L$-equivariant $U(\mathfrak{g})_n$-module and $\pi^i M_i=0.$ }
\item{The transition maps induce isomorphisms $M_i/\pi^{i-1} M_i \cong M_{i-1}$ of \\
$L$-equivariant $U(\mathfrak{g})_n$-modules}.
\item{$M \cong \invlim M_i$ as $\widehat{U(\mathfrak{g})_n}$-modules.}
\end{itemize}

A morphism between two $\hat{L}$-equivariant $\widehat{U(\mathfrak{g})_n}$-modules $(M,(M_i),(\alpha_i), (\rho_{Mi}))$ and\\ $(N,(N_i),(\beta_i),(\rho_{Ni}))$ is a map of $f:M \to N$ of $\widehat{U(\mathfrak{g})_n}$-modules such that there is a family of compatible $L$-equivariant morphisms $f_i:M_i \to N_i$ such that $f=\invlim f_i$. We call such a morphism $\hat{L}$-equivariant and denote $\Mod_{\fg}(\widehat{U(\mathfrak{g})_n},L$) the subcategory of $\Mod_{\fg}(\widehat{U(\mathfrak{g})_n})$ consisting of $\hat{L}$-equivariant modules and morphisms. As for equivariant $\widehat{\mathcal{D}}$-modules, we will omit the equivariance structure when it is clear in the context and just call $M$ a $\hat{L}$-equivariant $\widehat{U(\mathfrak{g})_n}$-module.
\end{definition}

We also define the notion of equivariant modules for the ring $\hugnK$.

\begin{definition}

A $\hat{L}_K$-equivariant $\hugnK$-module  is a quintuple \\$(M,M_0,(M_i),(\alpha_i), (\rho_{Mi}))$ such that  $M_0$ is a lattice for $M$ and \\$(M_0,(M_i),(\alpha_i), (\rho_{Mi})) \in \Mod_{\fg}(\widehat{U(\mathfrak{g})_n},L)$.

Next, let $(M,M_0,(M_i),(\alpha_i), (\rho_{Mi}))$ and $(N,N_0,(N_i),(\beta_i), (\rho_{Ni}))$ be $\hat{L}_K$-equivariant $\hugnK$-modules, and let $f:M \to N$ be a $\hugnK$ linear morphism.  As $M$ and $N$ are finitely generated, we have

$$\Hom_{\widehat{U(\mathfrak{g})_n}}(M_0,N_0) \uset{R} K \cong \Hom_{\hugnK}(M,N),$$

so there exists $f_0: M_0 \to N_0$ and $x \in K$ such that $f=f_0 \otimes x$. We say that $f$ is $\hat{L}_K$-equivariant if $f_0$ is $\hat{L}$-equivariant. Denote $\Mod_{\fg}(\hugnK,L$) the subcategory of finitely generated $\hugnK$ modules consisting of $\hat{L}_K$-equivariant objects along with $\hat{L}_K$-equivariant morphisms. We will ignore the equivariance structure when it is well understood from the context.
\end{definition}

Before stating the affinoid localisation mechanism, we need one more lemma:

\begin{lemma}
\label{d3inversesystemtensor}

Let $B$ be a Noetherian $R$-algebra and $A$ a finitely generated $B$-module. Let $A_i=A/\pi^i A$, $B_i=B/\pi^iB$, $\hat{A}=\invlim A_i$, $\hat{B}=\invlim B_i$. Further, let $C_i$ be a inverse system of $B_i$-modules and $C=\invlim C_i$. Assume that $C_i=C/\pi^i C$. Then: 

$$\invlim(A_i \uset{B_i} C_i) \cong \hat{A} \uset{\hat{B}} C.$$

\end{lemma}

\begin{proof}
Since $B$ is a Noetherian $R$-algebra, we have $\hat{A} \cong A \uset{B} \hat{B}$, so $$\hat{A} \uset{\hat{B}} C \cong  A \uset{B} \hat{B} \uset{\hat{B}} C \cong A \uset{B} C.$$

Consider the exact sequence 
                  $$ C \xrightarrow{\cdot \pi } C \to C/\pi^{i} C \to 0.$$

Since tensor product is right exact, one obtains:

              $$ A \uset{B} C \xrightarrow{\cdot \pi } A \uset{B} C \to A \uset{B} C_i \to 0 .$$

Therefore, we get $$[A \uset{B} C/ \pi^i (A \uset{B} C)] \cong A \uset{B} C_i \cong  A_i \uset{B_i}C_i.$$

The claim follows since $A \uset{B} C$ is $\pi$-adically complete.
\end{proof}

Recall by Definition \ref{d3htdodef} that there exists a map $i_{\mathfrak{g}}:U(\mathfrak{g})_n \to \mathcal{D}$.  By functoriality, the map $i_{\mathfrak{g}}:U(\mathfrak{g})_n \to \mathcal{D}$ induces a map $\widehat{i_{\mathfrak{g}}}: \hugn \to \widehat{\mathcal{D}}$ and thus a map $\widehat{i_{\mathfrak{g}}}: \hugnK \to \widehat{\mathcal{D}_K}$. 

\begin{definition}
We define two functors:

\begin{equation}
\begin{split}
&\Loc: \Mod(\hugnK) \to \Mod(\widehat{\mathcal{D}_K}), \qquad \Loc(M)= \widehat{\mathcal{D}_K} \uset{\hugnK} M, \\
&\Gamma: \Mod(\widehat{\mathcal{D}_K}) \to \Mod(\hugnK), \qquad \Gamma(\mathcal{M})=\Gamma(X,\mathcal{M}). 
\end{split}
\end{equation}

\end{definition}

\begin{proposition}

\label{d3equivariantaffinoidlocalisationmechanism}

\begin{enumerate}[label=\roman*)]
\item[]
\item{Let $M \in \Mod_{\fg}(\hugnK,L)$. Then $\Loc(M) \in \Coh(\widehat{\mathcal{D}_K},L)$.}
\item{Let $\mathcal{M} \in \Coh(\widehat{\mathcal{D}_K},L)$. Then $\Gamma(X, \mathcal{M}) \in \Mod_{\fg}(\hugnK,L)$.}

\end{enumerate}

\end{proposition}

\begin{proof}
Let $M_0$ be the lattice of $M$ such that $M_0 \in \Mod(\widehat{U(\mathfrak{g})_n},L)$. Then

$$(\widehat{\mathcal{D}} \uset{\widehat{U(\mathfrak{g})_n}} M_0) \uset{R} K \cong \widehat{\mathcal{D}_{K}} \uset{\widehat{U(\mathfrak{g})_{n,K}}} M \cong \Loc(M),$$

so $\mathcal{M}:=\widehat{\mathcal{D}} \uset{\widehat{U(\mathfrak{g})_n}} M_0$ is a lattice for $\Loc(M)$, so we need to prove $\mathcal{M}$ is $L$-equivariant.

Let $M_i= M_0/\pi^i M_0$ and $\mathcal{M}_i:= \mathcal{D} \uset{U(\mathfrak{g})_n} M_i$. Then, we  have by applying Lemma \ref{d3inversesystemtensor}  that $\mathcal{M} \cong \invlim \mathcal{M}_i$. Fix $i \in \mathbb{N}^*$; by construction we have  $\pi^i \mathcal{M}_i=0$; next, by definition we have that $M_i$ is a $L$-equivariant finitely generated $U(\mathfrak{g})_n$-module, so $\mathcal{M}_i$ is a quasi-coherent $L$-equivariant $\mathcal{D}$-module by Proposition \ref{d3equivariancepreservedunderlocandglo}. Since $M_i$ is finitely generated as a $U(\mathfrak{g})_n$-module, by picking a presentation of $M_i$ we obtain that $\mathcal{M}_i$ is also coherent.

Finally, consider the short exact sequence: 

$$ M_i \xrightarrow{ \cdot \pi} M_i \to M_{i-1} \to 0.$$

Since tensor product is right exact, we get a short exact sequence:

$$ \mathcal{D} \uset{U(\mathfrak{g})_n} M_i \xrightarrow { \cdot \pi} \mathcal{D} \uset{U(\mathfrak{g})_n} M_i \to \mathcal{D} \uset{U(\mathfrak{g})_n} M_{i-1} \to 0,$$

so $\mathcal{M}_i /\pi^{i-1} \mathcal{M}_i \cong \mathcal{M}_{i-1}$. Thus, we proved that $\mathcal{M}$ is indeed $L$-equivariant, so $\Loc(M)$ is also $L$-equivariant. This proves the first statement.

On the other hand, consider $\mathcal{M} \in \Coh(\widehat{\mathcal{D}_{K}},L)$ and let $M=\Gamma(X,\mathcal{M})$. Further, let $\mathcal{M}_0 \in \Coh(\widehat{\mathcal{D}},L)$ be the corresponding lattice of $\mathcal{M}$ and $M_0= \Gamma(X,\mathcal{M}_0$). By construction $M_0 \uset{R} K \cong M$, so it is enough to prove that $M_0$ is $L$-equivariant. Let $\mathcal{M}_i:= \mathcal{M}/ \pi^i \mathcal{M}$, $M_i:= \Gamma(X,\mathcal{M}_i)$. Since $\mathcal{M}$ is coherent, we have $\mathcal{M} \cong \invlim \mathcal{M}_i$, so $M_0= \invlim M_i$. Further, $M_i/\pi^{i-1} M_i \cong M_{i-1}$ and $\pi^i M_i=0$ for all $i \in \mathbb{N}$, so we are left to prove that $M_i$'s are $L$-equivariant finitely generated $U(\mathfrak{g})_n$-modules.

Since $\mathcal{M}_0 \in \Coh(\widehat{\mathcal{D}},L)$, we obtain by construction and Corollary \ref{d3inversesystemlemma} that for all $i \in \mathbb{N}^*$, $\mathcal{M}_i$ is a $L$-equivariant coherent $\mathcal{D}$-module. Then by Proposition \ref{d3equivariancepreservedunderlocandglo} and Assumption \ref{d3assumptionglobalsectioncoherent} we obtain that for all $i \in \mathbb{N}^*$, $M_i$ is a $L$-equivariant finitely generated $U(\mathfrak{g})_n$-module. This concludes the proof.
\end{proof}

\subsection{Applications of the localisation mechanism}

Throughout this subsection, we assume that $G$ is a connected, simply connected,
split semisimple, smooth affine algebraic group scheme over $\Spec R$. We also let $X=G/B$ denote the flag scheme which is a quasi-compact $R$-variety. Fix $n$ a deformation parameter and $\mathfrak{g}= \mathfrak{n}^- \oplus \mathfrak{h} \oplus \mathfrak{n}^+$ a Cartan decomposition of $\mathfrak{g}=\Lie(G)$. Further, we fix $\lambda:\pi^n \mathfrak{h} \to R$ an $R$-linear map and denote $R_{\lambda}$ the corresponding $U(\mathfrak{h})_n$-module. By \cite[Section 6.10]{Annals} we have an induced map $(U(\mathfrak{g})^G)_n \to U(\mathfrak{h})_n$ and we view $R_{\lambda}$ as a $(U(\mathfrak{g})^G)_n$-module via this map. We also let $K_{\lambda}:=R_{\lambda} \uset{R} K$ the corresponding $\widehat{U(\mathfrak{g})_{n,K}^G}$-module. We make the following definitions:

\begin{itemize}
\item{$U(\mathfrak{g})_n^{\lambda}:=U(\mathfrak{g})_n \uset{(U(\mathfrak{g})^G)_n}R_{\lambda}$,}
\item{$\widehat{U(\mathfrak{g})_n^{\lambda}}:= \invlim U(\mathfrak{g})_n^{\lambda}/\pi^i U(\mathfrak{g})_n^{\lambda}$  and }
\item{$\widehat{U(\mathfrak{g})_{n,K}^{\lambda}}:= \widehat{U(\mathfrak{g})_n^{\lambda}} \uset{R} K$.}
\end{itemize}

We should remark that by \cite[Theorem 6.10a)]{Annals}, $\widehat{U(\mathfrak{g})_{n,K}^{\lambda}} \cong \hugnK \uset{\widehat{U(\mathfrak{g})_{n,K}^G}} K_{\lambda}$, so in particular $\widehat{U(\mathfrak{g})_{n,K}^{\lambda}}$ is a quotient of $\hugnK$.

We also let $\mathcal{D}_n^{\lambda}$ be as in \cite[Section 6.4]{Annals}. This coincides with $\mathcal{D}_{\lambda,\pi^n}$ as defined in \cite[Definition 5.14]{Sta1}. By \cite[Theorem 6.10b)]{Annals} one has $\Gamma(X,\widehat{\mathcal{D}_{n,K}^{\lambda}}) \cong \widehat{U(\mathfrak{g})_{n,K}^{\lambda}}$. We define a localisation functor 

$$\Loc^{\lambda}: \Mod_{\fg}(\widehat{U(\mathfrak{g})_{n,K}^{\lambda}}) \to \Coh(\widehat{\mathcal{D}_{n,K}^{\lambda}}) \qquad \Loc^{\lambda}(M):=\widehat{\mathcal{D}_{n,K}^{\lambda}} \uset{\widehat{U(\mathfrak{g})_{n,K}^{\lambda}}}M.$$

We say that $\lambda \in \mathfrak{h}_K^*$ is \emph{dominant} if $(\lambda+\rho)(h)  \geq 0$ for any positive coroot $h \in \mathfrak{h}$. Given $\lambda:\pi^n \mathfrak{h} \to R$, we say that $\lambda$ is \emph{dominant} if the corresponding root $\lambda \in \mathfrak{h}_K^*$ is dominant. We say that $\lambda: \pi^n \mathfrak{h} \to R$ is \emph{regular} if the corresponding $\lambda \in \mathfrak{h}_K^*$ is regular, i.e. the stabiliser of the Weyl group action on $\lambda$ is trivial.

In \cite{Annals}, the authors prove an affinoid version of Beilinson-Bernstein localisation:

\begin{theorem}\cite[Theorem C]{Annals} \cite[Theorem 5.3.13]{Eqdcap}
\label{d3affinoidnonequivariantBB}
Let $\lambda:\pi^n \mathfrak{h} \to R$ be a dominant weight. The functor $\Gamma$ is exact and the functors $\Loc^{\lambda}$ and $\Gamma$ induce quasi-inverse equivalences of categories between $\Mod_{\fg}(\widehat{U(\mathfrak{g})_{n,K}^{\lambda}})$ and the quotient category $\Coh(\widehat{\mathcal{D}_{n,K}^{\lambda}})/\ker \Gamma$. In case $\lambda$ is also regular, then $\ker \Gamma=0$ whenever $n>0$ or $p$ is a very good prime for $G$.
\end{theorem}

We should remark that the part of the proof where $\lambda$ is dominant does not require $p$ to be a very good prime for $G$. The restriction on $p$ has been removed in \cite[Theorem 5.3.13]{Eqdcap} provided that $n>0$.  

We may prove an equivariant version of the affinoid localisation theorem.

\begin{theorem}
\label{d3affinoidequivariantBB}
Let $L$ be a closed subgroup of $G$ and let $\lambda:\pi^n \mathfrak{h} \to R$ be a dominant weight. The functors $\Loc^{\lambda}$ and $\Gamma$ induce quasi-inverse equivalences of categories between $\Mod_{\fg}(\widehat{U(\mathfrak{g})_{n,K}^{\lambda}},L)$ and the quotient category $\Coh(\widehat{\mathcal{D}_{n,K}^{\lambda}},L)/\ker \Gamma$. In case $\lambda$ is also regular, then $\ker \Gamma=0$ whenever $n>0$ or $p$ is a very good prime for $G$.

\end{theorem}

\begin{proof}
By Theorem \ref{d3affinoidnonequivariantBB}, it is enough to prove that $\Loc^{\lambda}$ and $\Gamma$ preserve the $L$-equivariance. We have by 
Proposition \ref{d3DlambdarisrdeformedGhtdo} that $\mathcal{D}_n^{\lambda}$ is a $\pi^n$-deformed $G$-htdo, so in particular it is $\pi^n$-deformed $L$-htdo. Further, we have by \cite[Proposition 5.15]{Annals} that $\mathcal{D}_n^{\lambda}$ satisfies Assumption \ref{d3assumptionglobalsectioncoherent}. The claim follows from Proposition \ref{d3equivariantaffinoidlocalisationmechanism} since $\widehat{U(\mathfrak{g})_{n,K}^{\lambda}}$ is a quotient of $\hugnK$.
\end{proof}

As a corollary we obtain:

\begin{corollary}
\label{d3corlocgloiso}
Let $M \in \Mod_{\fg}(\widehat{U(\mathfrak{g})_{n,K}^{\lambda}})$. Then

$$\Gamma(X,\Loc^{\lambda}(M)) \cong M.$$

\end{corollary}

In the next section, we will apply Theorem \ref{d3affinoidequivariantBB} in two cases: $B$ is a Borel subgroup of a $G$ and $G \cong G_{d}=\{(g,g) | g \in G\}$ is the diagonal subgroup of $G \times G$.

\subsection{Equivariance of two-sided ideals}

We keep the notation from the previous section. The Lie algebra of the algebraic group $G \times G$ is given by $\Lie(G \times G)=\Lie(G) \times \Lie(G)= \mathfrak{g} \times \mathfrak{g}$. We aim to prove that any two-sided ideal in $\hugnK$ is $G$-equivariant when viewed as $\huggnK$-module. Here we view $G$ as the diagonal subgroup of $G \times G$. To avoid confusion, we denote this group $G_d$.

Recall that the enveloping algebra $U(\mathfrak{g})$ is a $G$-representation. In particular, the group $G(R)$ acts on $U(\mathfrak{g})$ via the Adjoint action inducing a comodule map

$$\rho: U(\mathfrak{g}) \to \mathcal{O}(G) \otimes U(\mathfrak{g}).$$

Further, since the $G$-action preserves $U(\mathfrak{g})_n$, the map $\rho$ restricts to a comodule map $\rho:U(\mathfrak{g})_n \to \mathcal{O}(G) \otimes U(\mathfrak{g})_n$. Let $\widehat{\mathcal{O}(G)}:=\invlim \mathcal{O}(G)/\pi^i \mathcal{O}(G) $ denote the $\pi$-adic completion of the Hopf algebra $\mathcal{O}(G)$ corresponding to the group $G$. Using the fact that the $\pi$-adic completion is a functor, we obtain a map

               $$\hat{\rho}:\hugn \to \widehat{\mathcal{O}(G)} \hat{\otimes} \hugn,$$ 

where $\hat{\otimes}$ denotes the completed tensor product.
\begin{definition}
We say that a two-sided ideal $I$ in $\hugn$ is $\pi$-closed if the quotient $\hugn/I$ is $\pi$-torsion-free.
\end{definition}

For the rest of this subsection, we let $I$ a $\pi$-closed two-sided ideal in $\hugn$.

 By construction, we have that $g \cdot x= \hat{\rho}(x) (g) $, for all $ g \in G(R), x \in \hugn$, so applying \cite[Corollary 4.3]{Munster}, we obtain

\begin{equation}
\label{d3groupactionpreservesidealequation}
\hat{\rho}(x) (g) =g \cdot x \in I , \text{ for all } g \in G(R), x \in I. 
\end{equation}

For each $g \in G(R)$ consider the map $\epsilon_g:\widehat{\mathcal{O}(G)} \to R$, $\epsilon_g(f):=f(g)$ and  let $q:\hugn \to \hugn /I$ denote the natural projection. Consider the following commutative diagram:

\begin{equation}
\label{d3comodulestructurediagram}
\begin{tikzcd}
\widehat{\mathcal{O}(G)} \hat{\otimes} \hugn \arrow[d,"\epsilon_g \hat{\otimes} \id"] \arrow[r,"\id \hat{\otimes} q"] & \widehat{\mathcal{O}(G)} \hat{\otimes} \hugn/I \arrow[d,"\epsilon_g \hat{\otimes} \id"]           \\
R \hat{\otimes} \hugn \arrow[r,"\id \hat{\otimes} q"] & R \hat{\otimes} \hugn/I.
\end{tikzcd}
\end{equation}

By equation \eqref{d3groupactionpreservesidealequation}, we have that $(\id \hat{\otimes} q) \circ (\epsilon_g \hat{\otimes} \id) \circ \hat{\rho}(i)=0$, for all $i \in I$ and $g \in G(R)$, therefore we obtain

\begin{equation}
\label{d3comodulediagramequation}
(\epsilon_g \hat{\otimes} \id) \circ (\id \hat{\otimes} q) \circ \hat{\rho}(i)=0 \text{ for all } g \in G(R),i \in I.  
\end{equation}

Let $\widehat{K(G)}:=\widehat{\mathcal{O}(G)} \uset{R} K$. We wish to prove that  the Jacobson radical of $\widehat{K(G)}$ is $0$, and if $f \in \widehat{\mathcal{O}(G)}$, viewed as an element of $\widehat{K(G)}$, is such that $\epsilon_g(f)=0$ fo all $g \in G(\mathcal{O}_L)$ and all $L/K$  finite extensions, then $f$ is in the intersection of all maximals ideals of $\widehat{K(G)}$. By combining the results, we obtain $f=0$.

\begin{proposition}
\label{d3Jacobsonradicalofcompletedaffinoid0}
The Jacobson radical of $\widehat{K(G)}$, $J(\widehat{K(G)})$, is 0.

\end{proposition}

\begin{proof}

Any free Tate algebra over a non-archimidean field $K$ is a Jacobson ring by \cite[Proposition 3.1.3]{Bos}; in particular as $\widehat{K(G)}=K \langle x_1,x_2, \ldots x_n \rangle /J$ is a quotient of a free Tate algebra by some closed ideal $J$, we have $J(\widehat{K(G)})=$nilradical($\widehat{K(G)})$, so it suffices to prove that nilradical($\widehat{K(G)}$)=0.

 As $G$ is a reductive connected group scheme, we have by \cite[II.1.9 (4)]{Jan1} that $\mathcal{O}(G)$ is an integral domain, therefore $k(G):=\mathcal{O}(G) \uset{R} k$ is also an integral domain. Consider the $\pi$-adic  filtrations on $\mathcal{O}(G)$ and $\widehat{\mathcal{O}(G)}$; we have by the properties of $\pi$-adic completions
                   $$\text{gr}(\widehat{\mathcal{O}(G))}=\text{gr}(\mathcal{O}(G))\cong (\text{gr }R)(G).$$
                   
As $(\text{gr R})$ is a polynomial ring over $k$ and $k(G)$ is an integral domain, we obtain that  $\text{gr}(\widehat{\mathcal{O}(G)})$ is an integral domain, so $\widehat{\mathcal{O}(G)}$ is an integral domain. Therefore, $\widehat{K(G)}=\widehat{\mathcal{O}
(G)}\uset{R} K$ is an integral domain, so in particular $\widehat{K(G)}$ has trivial nilradical.                    
\end{proof}

\begin{proposition}
\label{d3functionis0proposition}
Let $f \in \widehat{\mathcal{O}(G)}$ such that $\epsilon_g(f)=0$ for all $g \in G(\mathcal{O}_L)$ and all $L$ finite extensions of $K$. Then $f=0$. 

\end{proposition}

\begin{proof}

View $f$ as an element of $\widehat{K(G)}$. Further, let $K \langle x_1, x_2 \ldots, x_n \rangle$ be a free Tate algebra projecting onto $\widehat{K(G)}$ via a map denoted $\phi$; let $J=\ker \phi$.  Finally, let $\mathfrak{m} \subset \widehat{K(G)}$ be a maximal ideal of $\widehat{K(G)}$; we aim to prove that $f \in \mathfrak{m}$.

As $\mathfrak{m} \subset \widehat{K(G)}$ is maximal, $\phi^{-1}( \mathfrak{m})$ is  a maximal ideal in $K \langle x_1, x_2 \ldots, x_n \rangle $, so we get an induced map $\zeta:\widehat{K(G)}/\mathfrak{m} \to K \langle x_1 x_2 \ldots x_n \rangle / \phi^{-1} (\mathfrak{m})$. Further, we have by \cite[Corollary 2.2.12]{Bos}, $K \langle x_1, x_2, \ldots, x_n \rangle / \phi^{-1} (\mathfrak{m}) \cong L$, where $L$ is a finite extension of $K$. The image of $f$ under the composition of the maps (call this composition $\eta:\widehat{\mathcal{O}(G)} \to L$) lies into the ring of integers of $L$, $\mathcal{O}_L$.

By \cite[Example 1.8 ii)]{Che}, there is a correspondence between maps from $\widehat{K(G)}$ to $L$ and the zero locus of a system of generators for the ideal defining $J$ (recall $\widehat{K(G)}=K \langle x_1,x_2 \ldots, x_n \rangle/J$) inside $\mathcal{O}_L^n$. Therefore, as $\epsilon_g(f)=0$ for all $g \in G(\mathcal{O}_L)$, we obtain $\eta(f)=0$. Consider the composition defining $\eta$:

$$\widehat{K(G)} \to \widehat{K(G)}/\mathfrak{m} \to K \langle x_1,x_2 \ldots x_n \rangle /\phi^{-1}(\mathfrak{m}) \cong L.$$

As $\mathfrak{m}$ is a maximal ideal, $\widehat{K(G)}/\mathfrak{m}$ is a field, so the map $\widehat{K(G)}/\mathfrak{m} \to L$ is an injection. Thus, as $\eta(f)=0$, one obtains that $f \in \mathfrak{m}$. In conclusion, $f$ lies in all the maximal ideals of $\widehat{K(G)}$, i.e.  $f \in J(\widehat{K(G)})$; applying Proposition \ref{d3Jacobsonradicalofcompletedaffinoid0}, we get $f=0$.
\end{proof}

\begin{theorem}
\label{d3affinoidcomodulestructure}
Let $I$ be a $\pi$-closed two-sided ideal in $\hugn$. Then $\hat{\rho}(I) \subset \widehat{\mathcal{O}(G)} \hat{\otimes} I.$
\end{theorem}

\begin{proof}

Consider the composition map $(\epsilon_g \hat{\otimes} \id)\circ (\id \hat{\otimes} q) \circ \hat{\rho}:I \to  R \hat{\otimes} \hugn/I $. By equation \eqref{d3comodulediagramequation}, we know that  for all $i \in I$, $(\epsilon_g \hat{\otimes} \id)\circ (\id \hat{\otimes} q) \circ \hat{\rho}(i)=0.$

Let $I_K=I \uset{R} K$ and notice that $I_K$ is a two-sided ideal in $\hugnK$. As $I$ is a $\pi$-closed ideal, the space $\hugn/I$ has no $\pi$-torsion, so we obtain

                $$\hugn /I \uset{R} K \cong \hugnK/I_K.$$

The space  $\hugnK/I_K$ is a $K$-Banach space that has a countably dimensional dense subspace consisting of elements of the form $x+I_K$, $x \in U(\mathfrak{g}_K)$. Therefore, applying \cite[Proposition 10.4]{Sch}, we get that $\hugnK/I_K$ has a countable  topological $K$-basis, so $\hugn/I$ has a countable topological $R$-basis; denote this basis $\{y_i| i \in \mathbb{N}\}$. Another way to see the existence of this basis is that the space $\hugnK/I_K$ is a separable $K$-Banach space, so it has a Schauder basis.

 Consider an element $a=\sum_{i=1}^\infty f_i \hat{\otimes} y_i \in \widehat{\mathcal{O}(G)} \hat{\otimes} \hugn$. Then we have for all $g \in G(R),$

       $$0=(\epsilon_g \hat{\otimes} \id)(\sum_{i=1}^{\infty} f_i \hat{\otimes}y_i)= \sum_{i=1}^{\infty} f_i(g)y_i.$$

As $y_i$'s form a topological basis of $\hugn$, we obtain

\begin{equation}
\label{d3evfunctionG(R)}
\epsilon_g(f_i)=0 \text{ for all }  g \in G(R),i \in \mathbb{N}. 
\end{equation}

Now, let $A$ be  $\pi$-adically complete commutative $R$-algebra finitely generated as an $R$-module. For $g \in G(A)$ let $\epsilon_g:\widehat{A(G)} \to A$ denote the evaluation map by abusing notation. 

Recall that $\hugn$ is a $G$-representation by extending the Adjoint action of $G$ on $\mathfrak{g}$. Consider the set $I \uset{R} A$ inside $\widehat{U(\mathfrak{g})_{n,A}}:=\hugn \uset{R} A$. Notice that since $A$ is finitely generated as an $R$-module, we only need to take the standard tensor product, not the completed one. Let $x_1 \otimes y_1$ be a simple tensor in $\widehat{U(\mathfrak{g})_{n,A}}$ and $x_2 \otimes y_2$ be a simple tensor in $I \uset{R} A$. Then
                $$(x_1 \otimes y_1)(x_2 \otimes y_2)=x_1x_2 \otimes y_1 y_2.$$

As $x_2 \in I$ and $I$ is a two-sided ideal $x_1x_2 \in I$, so $(x_1 \otimes y_1)(x_2 \otimes y_2) \in I \uset{R} A $. Extending this to non-simple tensors, taking in account all the possible ways to represent elements in $I \uset{R} A$ and $\widehat{U(\mathfrak{g})_{n,A}}$ as sums of simple tensors, we get that $I \uset{R} A$ is a left ideal in $\widehat{U(\mathfrak{g})_{n,A}}$. By symmetry it is also a right ideal, so $I \uset{R} A$ is indeed a two-sided ideal in $\widehat{U(\mathfrak{g})_{n,A}}$. As $\hugn$ is a $G$-representation and $G(R)$ preserves $I$, $G$ is a flat group scheme, we deduce that $I$ is a  $G$-subrepresentation of $\hugn$, so $G(A) \cdot (I \uset{R} A) \subset I \uset{R} A$. Therefore, by base changing equation \eqref{d3evfunctionG(R)} to $A$  we get

\begin{equation}
\label{d3evalgequation}
     \epsilon_g(f_i)=0 \text{ for all } g \in G(A), i \in \mathbb{N}.
\end{equation}

In particular we get that the result is true for any $\mathcal{O}_L$, where $L$ is a finite extension of $K$. Applying Proposition \ref{d3functionis0proposition}, we obtain $f_i =0 $ for all $i \in \mathbb{N}$. Thus, we have obtained that $a=0$, so $(\id \hat{\otimes} q) \circ \hat{\rho}(i)=0$, which implies

 $$\hat{\rho}(i) \in \ker(\id \hat{\otimes} q)=\widehat{\mathcal{O}(G)} \hat{\otimes} I. $$

Therefore, $\hat{\rho}(I) \subset  \widehat{\mathcal{O}(G)} \hat{\otimes} I$. 
\end{proof}

Let $\tau$ be the principal anti-automorphism of $U(\mathfrak{g})$ induced by $x \to -x$ for all $x \in \mathfrak{g}$. We use $x^{\tau}$ to denote $\tau(x)$.  For all $x_1,x_2 \ldots x_n \in \mathfrak{g}$, we have

 $$(x_1x_2 \ldots x_n)^{\tau}=(-1)^{n}x_nx_{n-1} \ldots x_2 x_1.$$
 
We define the action of the ring $\huggn \cong \hugn \hat{\otimes} \hugn$  on $\hugn$ via

$$ (a \otimes b) x= bxa^{\tau}, \text{ for all } a,b,x \in \hugn.$$

Let $m:(\hugn \hat{\otimes} \hugn) \hat{\otimes} \hugn \to\hugn$ denote the action map. The set of submodules of  $\hugn$ under this action coincide with the set of two-sided ideals. The group $G \times G$ acts on $\hugn \hat{\otimes} \hugn$ via the adjoint action:

$$ (g_1,g_2) \cdot (x \hat{\otimes} y) =(\Ad(g_1)x \hat{\otimes} \Ad(g_2)y).$$

In particular we get an action of the group $G_d \cong G$. Let $$\hat{\rho}_{\bimod}: \hugn \hat{\otimes} \hugn  \to  \widehat{\mathcal{O}(G)} \hat{\otimes} \hugn \hat{\otimes} \hugn $$ be the corresponding comodule map.

Finally, let $\mathfrak{g}_d=\Lie(G_d)$. It embeds into $ \hugn \hat{\otimes} \hugn$ via $x \mapsto
x \hat{\otimes} 1 + 1 \hat{\otimes} x$ for all $x \in \mathfrak{g}_d$.

\begin{proposition}
\label{d3piclosedidealequiv}

Let $I$ be a $\pi$-closed two-sided ideal in $\hugn$. Then, $I \in \Mod(\huggn, G_d)$.

\end{proposition}

\begin{proof}

By abuse of notation let $\hat{\rho}: I \to \widehat{\mathcal{O}(G)} \hat{\otimes} I$ be the restriction of $\hat{\rho}$ to $I$ induced by the $\Ad$ action; by Theorem \ref{d3affinoidcomodulestructure} this map is well defined. Furthermore, since the ring $\hugn \hat{\otimes} \hugn$ is Noetherian, $I$ is also finitely generated. Let $$\hat{\rho}_{\tensor}: \hugn \hat{\otimes} \hugn \hat{\otimes} I \to \widehat{\mathcal{O}(G)} \hat{\otimes}  \hugn \hat{\otimes} \hugn \hat{\otimes} I $$

be the comodule map induced by $\hat{\rho}$ and $\hat{\rho}_{\bimod}$. To prove that the multiplication $m$ is a morphism of comodules it is enough to prove that for all $g \in G, x,y \in \hugn, u \in I$.

$$ \Ad(g) \cdot ((x \hat{\otimes} y) \cdot u)= (\Ad(g) x \hat{\otimes} \Ad(g)y) \cdot (\Ad(g) u)$$

We have:

\begin{equation}
\begin{split}
 \Ad(g) \cdot ((x \hat{\otimes} y) \cdot u) &= \Ad(g)( yux^{\tau}) \\
                                       &= \Ad(g)y\Ad(g)u\Ad(g)x^{\tau} \\
                                       &= \Ad(g)y \Ad(g)u (\Ad(g)x)^{\tau}\\
                                       &= (\Ad(g)x \hat{\otimes} \Ad(g) y) \cdot (\Ad(g)u).
\end{split}
\end{equation}

Next, the differentiation of the $\Ad$ action is the ad action which coincides with the action of the Lie Algebra $\mathfrak{g}_d$. (*) 

Now, consider $I_i= I/\pi^i I$. Then it is easy to see that $I_i$ is finitely generated as $U(\mathfrak{g} \times \mathfrak{g})_n \cong U(\mathfrak{g} )_n\otimes U(\mathfrak{g})_n$-module (here $U(\mathfrak{g})_n\otimes U(\mathfrak{g})_n$ acts on $I_i$ via $(x \otimes y) \cdot (u+\pi^i I)=xuy^{\tau}+\pi^i I$), $\pi^iI_i=0$ and $I=\invlim I_i$. The map $\hat{\rho}:I \to \widehat{\mathcal{O}(G)} \hat{\otimes} I$ descends to a map 
$\rho_i: I_i \to \mathcal{O}(G) \otimes I_i$ which is compatible with the action map since $\hat{\rho}_{\tensor}$ is a comodule homomorphism.
 Finally, by (*) the differentiation of the $\Ad$ action descend to $I_i$, so $I_i $ is indeed  a $G_d$-equivariant $U(\mathfrak{g})_n \otimes U(\mathfrak{g})_n$-module. Thus, we have proven all the conditions required to make $I$ a $\hat{G}_d$-equivariant $\huggn$-module.  
\end{proof} 


\begin{corollary}
\label{d3equivarianceoftwosidedideals}
Let $J$ be a two-sided ideal in $\hugnK$. Then $$J \in \Mod_{\fg}(\huggnK, G_d).$$

\end{corollary}

\begin{proof}
Clearly, $J$ is finitely generated since $\huggnK$ is a Noetherian ring. Let $I=J \cap \hugn$. It is easy to see that $I$ is a two sided ideal in $\hugn$; we claim it is $\pi$-closed. Suppose there exists $x \in \hugn$ and $n \in \mathbb{N}^*$ such that $\pi^n(x+I)=0+I$. Then we obtain $\pi^n x \in I \subset J$. Since $J$ is a two-sided ideal in $\hugnK$, we have $x \in J$. By the construction of $I$, we obtain $x \in I$, i.e. $x+I=0+I$, so $I$ is indeed $\pi$-closed. Therefore, by applying Proposition \ref{d3piclosedidealequiv}, we obtain $I \in \Mod(\huggn, G_d)$.

To finish the proof, we need to prove that $I$ is a lattice for $J$. Notice that $I \uset{R} K \subset J$. Let $x \in J$; there exists $n \in \mathbb{N}$ such that $\pi^n x \in \hugn$, so $\pi^n x \in I$. Thus $x= \pi^nx \uset{R} \pi^{-n} \in I \uset{R} K$, so $J = I \uset{R} K$. Finally, since $\cap_{i=1}^{\infty} \hugn/ \pi^i \hugn=0$ we obtain that $\cap_{i=1}^{\infty} I/\pi^i I=0$, so $I$ is indeed a lattice for $J$.
\end{proof}

\subsection{Ideals with a given central character}

We now specialise to ideals in $\hugnK$ with a given central character. The Cartan subalgebra of $\mathfrak{g} \times \mathfrak{g}$ is given by $\mathfrak{h} \times \mathfrak{h}$, so picking a weight $\nu: \pi^n( \mathfrak{h} \times \mathfrak{h}) \to R$ is the same as picking a pair of weights $(\lambda,\mu)$ where $\lambda,\mu: \pi^n \mathfrak{h} \to R$. 

\begin{proposition}
\label{d3huggnKllisotensorhuggnKl}
We have $ \huggnKlm \cong \hugnKl \hat{\uset{K}} \hugnKm$.
\end{proposition}

We need the following Lemma:

\begin{lemma}
\label{d3completedtensorproductquotient}

Let $A$ and $B$ be complete normed $K$-algebras and $I,J$ be closed two-sided ideals in $A$ and $B$, respectively. Then

$$ A \hat{\uset{K}} B/(I \hat{\uset{K}} B + A \hat{\uset{K}} J) \cong A/I \hat{\uset{K}} B/J.$$

\end{lemma}

\begin{proof}

 We call a continuous morphism   $\phi:M \to N$ between two semi-normed $K$-vector spaces \emph{strict} if the natural morphism $\coim \phi \to \im \phi$ is a homeomorphism. If $M$ and $N$ are Banach spaces, then by \cite[Lemma 2.6]{Bode}, $\phi$ is strict if and only if the image of $\phi$ is closed in $N$.

Let $\phi_1:A \to A/I$, $\phi_2:B \to B/J$ be the natural projections. Since $A$ and $B$ are Banach spaces and $I$ and $J$ are closed ideals, we get by the discussion above that $\phi_1$ and $\phi_2$ are strict morphisms. Further, $R$ is mixed characteristic $(0,p)$, so the valuation on $\mathbb{Q} \subset K=\Frac(R)$ is non-trivial. Thus, by \cite[Theorem 2.8]{Bode} the morphism $\phi_1 \uset{K} \phi_2: A \uset{K} B \to A/I \uset{K} B/J$
is also strict. Furthermore, as $I,J$ are closed in $A$ and $B$ respectively, we get that the natural inclusion $I \uset{K} B + A \uset{K} J \to A \uset{K} B$ is also strict. Therefore, we obtain a \emph{strict} short exact sequence

$$ 0 \to I \uset{K} B + A \uset{K} J \to A \uset{K} B \to A/I \uset{K} B/J \to 0.$$

Applying \cite[Corollary 1.1.9/6]{BGR}, we get a strict exact sequence

\begin{equation}
\label{d3completetensorses}
0 \to \overline{I \uset{K} B + A \uset{K} J} \to A \hat{\uset{K}} B \to A/I \hat{\uset{K}} B/J \to 0.
\end{equation}
Since $I$ and $J$ are closed ideals, we have $\overline{I \uset{K} B + A \uset{K} J}= I \hat{\uset{K}} B + A \hat{\uset{K}} J$. The lemma follows from equation \eqref{d3completetensorses}.
\end{proof}

We can now prove Proposition \ref{d3huggnKllisotensorhuggnKl}:

\begin{proof}

For $\lambda:\pi^n \mathfrak{h} \to R$  denote $\chi_{\lambda}: U(\mathfrak{g})^G_n \to R$ the map obtained by composing the map $\lambda$ with the map $U(\mathfrak{g})^G_n \to U(\mathfrak{h})_n$. Recall by Theorem \ref{d3centreofaffinoidenvelopingalgebras}, we have $Z(\hugnK) \cong \widehat{U(\mathfrak{g})^G_{n,K}}$, so $\chi_{\lambda}$ determines a central character of $\hugnK$ which we denote $\chi_{\lambda}$ by abuse of language. Let $m_{\lambda}= \ker \chi_{\lambda}$, so that $\hugnKl=\hugnK/\hugnK m_{\lambda}$.
Further, consider $\chi_{\lambda,\mu}:U(\mathfrak{g} \times \mathfrak{g})^{G \times G}_n \to R$ and let $m_{\lambda,\mu}=\ker \chi_{\lambda,\mu}$ so that $\huggnKlm=\huggnK/m_{\lambda,\mu} \huggnK$.

We have by definition that $m_{\lambda,\mu}=m_{\lambda} \otimes (U(\mathfrak{g})^G)_n + (U(\mathfrak{g})^G)_n \otimes m_{\mu}$, so 

\begin{equation}
m_{\lambda,\mu} \huggnK=m_{\lambda} \hat{\uset{K}} \hugnK+ \hugnK \hat{\otimes} m_{\mu}.
\end{equation}

The claim now follows by applying Lemma \ref{d3completedtensorproductquotient}.
\end{proof}

Recall that $\tau$ denotes the principal anti-automorphism of $U(\mathfrak{g})$. It can be extended to an anti-automorphism of $\hugnK$, which we will also call $\tau$. We have by \cite[Lemma 5.4-Equation 5.5]{BeGi} that $\tau$ induces an isomorphism $U(\mathfrak{g})^{\lambda^{\op}} \cong U(\mathfrak{g})^{-w_{o} \lambda}$; here $w_{o}$ denotes the longest element of $W$. The map $\tau$ extends to an isomorphism $\hugnKl^{\op} \cong \reallywidehat{\U{g}_{n,K}^{-w_{o} \lambda}}.$ From now on, until the end of the document we will use $\lambda^*$ to denote $-w_{o} \lambda$.

Recall that if $I$ is a two-sided ideal in $\hugnK$, we have shown in Corollary \ref{d3equivarianceoftwosidedideals} that $I \in \Mod_{\fg}(\huggnK, G_d)$. Furthermore, if $I$ has central character $\chi_{\lambda}$, i.e. $m_{\lambda} \subset I$, we view $I$ as a two-sided ideal  in $\hugnKl$. We have by Proposition \ref{d3huggnKllisotensorhuggnKl} that  $I$ is a module over the ring $\huggnKlstl \cong \hugnKlst \hat{\uset{K}} \hugnKl$, so we have:

\begin{corollary}
\label{d3equivariancetwosidedidealslambdacentralcharacter}
Let $I$ be a two-sided ideal in $\hugnKl$. Then $$I \in \Mod_{\fg}(\huggnKlstl, G).$$.

\end{corollary}

We should remark that we have dropped the index $d$ as this should not cause any confusion for the rest of the document. Recall that we view $G$ via its diagonal embedding into $G \times G$. The corollary above is an affinoid version of \cite[Corollary 4.5]{Sta2}.


\section{Global sections under affinoid pullback}
\label{d3sectionglobalsections}

Throughout this section, we aim to compute global sections under the affinoid pullback defined in Remark \ref{d3definitionofaffinoidpullback}. We proceed by developing some machinery.

\subsection{Preliminary lemmas}

Recall that for any sheaf $R$-modules $\mathcal{F}$ on a topological space $Y$, we denote $\hat{\mathcal{F}}:= \invlim \mathcal{F}/\pi^i \mathcal{F}$  its $\pi$-adic completion and $\mathcal{F}_K$ the sheaf defined by $\mathcal{F}_K(U):=F(U) \uset{R} K$ for any $U \subset Y$ open.

Throughout this section we will freely make use of the following easy result: Let $f:Y \to W$ be a map of $R$-schemes and $\mathcal{F}$ be a sheaf of $R$-modules on $Y$, then $(f_* \mathcal{F})_K \cong f_*(\mathcal{F}_K).$

\begin{lemma}
\label{d3completiondirectimage}

Let $j:Y \to W$ be a closed embedding of $R$-varieties and let $(\mathcal{F}_i)$ be an inverse system of sheaves of $R$-modules such that $\mathcal{F}_i/\pi^{i-1} \mathcal{F}_i \cong \mathcal{F}_{i-1}$ for $i \geq 1$. Let $\mathcal{F}:=\invlim \mathcal{F}_i$. Then:

$$ j_*  \mathcal{F} \cong (\invlim j_* \mathcal{F}_i).$$

\end{lemma}

\begin{proof}

Notice that apriori it is not clear that the right-hand side is well defined. However, since $j$ is a closed embedding, it is in particular right exact. Therefore, one can prove in a similar fashion to Lemma \ref{d3preservationofinversesystemlemma} that $$j_* \mathcal{F}_i / \pi^{i-1} j_* \mathcal{F}_i \cong j_* \mathcal{F}_{i-1}.$$ 

Let $\mathcal{G}:= \invlim j_* \mathcal{F}_i$. Since $j_*$ is a right adjoint functor to the inverse image functor, $j_*$ commutes with inverse limits, thus

\begin{equation*}
 j_*\mathcal{F}= j_*( \invlim \mathcal{F}_i) \cong \invlim (j_* \mathcal{F}_i)= \mathcal{G}. \qedhere
\end{equation*}
\end{proof}

\begin{lemma}
\label{d3inverseimagecompletion}

Let $Y$ be an $R$-variety and let $p_r: Y \times Y \to Y$ be the projection on the right factor, and $\mathcal{F}$ a sheaf of $R$-modules. Then: $$\widehat{p_r ^{-1} \mathcal{F}} \cong p_r^{-1} \hat{\mathcal{F}}.$$

\end{lemma}

\begin{proof}

First,  notice that there is a map from the right hand side to the left-hand side via the maps $\hat{\mathcal{F}} \to \mathcal{\hat{F}}/\pi^{i} \mathcal{\hat{F}} \cong \mathcal{F}/\pi^{i} \mathcal{F}$. Let $U,V \subset Y$ be affine open, so that $U \times V \subset Y \times Y$ is affine open. Then,  we have:

\begin{equation}
\label{d3projectioncompletionsheafequation}
\begin{split}
\widehat{p_r^{-1} \mathcal{F}}(U \times V) &= \invlim (p_r^{-1} \mathcal{F}/ \pi^i p_r^{-1} \mathcal{F}) ( U \times V) \\
           & \cong \invlim p_r^{-1} \mathcal{F}(U \times V)/ \pi^{i} p_r^{-1} \mathcal{F} (U \times V) \\
           & \cong \invlim \mathcal{F}(V)/ \pi^{i} \mathcal{F}(V) \\
           & \cong \hat{\mathcal{F}}(V) \\
           & \cong p_r^{-1} \hat{\mathcal{F}}(U \times V).
\end{split}
\end{equation}

Now let $(U_i)_{i \in I}$ be an affine open covering of $Y$. Then $\{U_j \times U_k| j,k \in I \}$ is an affine open covering of $Y \times Y$. By equation \eqref{d3projectioncompletionsheafequation}, we have that the sheaves $\widehat{p_r^{-1} \mathcal{F}}$ and $p_r^{-1} \hat{\mathcal{F}}$ agree on affine open cover and since there exists a map between the two, they are isomorphic.
\end{proof}

\subsection{Simplyfing the pullback functor}

We retain the notation from the previous section. Further, for the rest of the section we assume $\lambda,\mu: \pi^n  \mathfrak{h} \to R$ are $R$-linear \emph{dominant weights}. Since $\mathcal{D}_n^{\lambda,\mu}$ is a sheaf of $\pi^n$-deformed $G$-htdo by Corollary \ref{d3DlambdarisrdeformedGhtdo} and the double flag variety $X \times X$ is quasi-compact we have by Theorem \ref{d3alaBorho-Brylinski}, Proposition \ref{d3piadicBorhoBrylinski} and Theorem \ref{d3affinoidBorhoBrylinskitheorem} equivalences of categories:

\begin{equation}
\begin{split}
&i_l^{\#}: \Coh(\mathcal{D}^{\lambda,\mu}_n,G) \to \Coh(\mathcal{D}^{\lambda}_n,B), \\
&\hat{i_l}^{\#}: \Coh(\widehat{\mathcal{D}^{\lambda,\mu}_n},G)  \to \Coh(\widehat{\mathcal{D}^{\lambda}_n},B), \\
&\hat{i}_{l,K}^{\#}: \Coh(\widehat{\mathcal{D}^{\lambda,\mu}_{n,K}},G)  \to \Coh(\widehat{\mathcal{D}^{\lambda}_{n,K}},B).
\end{split}
\end{equation}

We let $i:eB \to X$ and $p_r:X \times X \to X$ denote the natural inclusion and projection onto the right factor, respectively. We also define $\mathcal{M}^{\mu}_n:=i_*R \uset{\mathcal{O}_X} \mathcal{D}^{\mu}_n.$





We can now give a description of the pullback of $G$-equivariant coherent $\widehat{\mathcal{D}^{\lambda,\mu}_{n}}$-modules and $\widehat{\mathcal{D}^{\lambda,\mu}_{n,K}}$-modules. We start by proving $\pi$-adic and affinoid versions of \cite[Corollary 6.4]{Sta2}.
\begin{proposition}
\label{d3piadicpullback}
Let $\mathcal{M} \in \Coh(\widehat{\mathcal{D}^{\lambda,\mu}_n},G) $. Then:

$$i_{l_*}(\hat{i_l}^{\#} \mathcal{M}) \cong p_r^{-1}\widehat{\mathcal{M} ^{\mu}_n}\uset{p_r^{-1} \widehat{\mathcal{D}^{\mu}_n}} \mathcal{M}.$$

\end{proposition}

\begin{proof}
We know by construction that $\mathcal{M}/\pi^i\mathcal{M} \in \Coh(\mathcal{D}^{\lambda,\mu}_n,G)$. We have

\begin{equation}
\begin{split}
i_{l_*}(\hat{i_l}^{\#} \mathcal{M})&=i_{l_*}(\invlim i_l^{\#} (\mathcal{M}/\pi^i \mathcal{M})) \\
&\cong \invlim (i_{l_*} (i_l^{\#} (\mathcal{M}/\pi^i \mathcal{M}))) \text{ (by Lemma \ref{d3completiondirectimage}) }   \\
&\cong \invlim  (p_r^{-1}\mathcal{M}^{\mu}_n \uset{p_r^{-1} \mathcal{D}^{\mu}_n} \mathcal{M}/\pi^i \mathcal{M}) \text{ (by \cite[Corollary 6.4]{Sta2})} \\
&\cong \invlim (p_r^{-1}\mathcal{M}^{\mu}_n/ \pi^i p_r^{-1}\mathcal{M}^{\mu}_n \uset{p_r^{-1} \mathcal{D}^{\mu}_n/\pi^i p_r^{-1} \mathcal{D}^{\mu}_n} \mathcal{M}/\pi^i \mathcal{M}) \\
&\cong \widehat{p_r^{-1} \mathcal{M}^{\mu}_n} \uset{\widehat{p_r^{-1} \mathcal{D}^{\mu}_n}} \mathcal{M} \text{ (by Lemma \ref{d3inversesystemtensor}) } \\
&\cong p_r^{-1}\widehat{\mathcal{M} ^{\mu}_n}\uset{p_r^{-1} \widehat{\mathcal{D}^{\mu}_n} } \mathcal{M} \text{ (by Lemma \ref{d3inverseimagecompletion})}. \qedhere
\end{split}
\end{equation}
\end{proof}

\begin{corollary}
\label{d3affinoidpullback}
Let $\mathcal{M} \in \Coh(\widehat{\mathcal{D}^{\lambda,\mu}_{n,K}},G) $. Then:

$$i_{l_*}(\hat{i}_{l,K}^{\#} \mathcal{M}) \cong (p_r^{-1}\widehat{\mathcal{M} ^{\mu}_n})_K \uset{p_r^{-1} \widehat{\mathcal{D}^{\mu}_{n,K}}} \mathcal{M}. $$

\end{corollary}

\begin{proof}
Let $\mathcal{M}_0$ be the correspoding lattice for $\mathcal{M}$ such that $\mathcal{M}_0 \in \Coh(\widehat{\mathcal{D}^{\lambda,\mu}_n},G)$. By definition $(\hat{i}_{l,K}^{\#} \mathcal{M})=(\hat{i_l}^{\#} \mathcal{M}_0)_K$. We have

\begin{equation}
\begin{split}
i_{l_*}(\hat{i}_{l,K}^{\#} \mathcal{M})&=i_{l_*}((\hat{i_l}^{\#} \mathcal{M}_0)_K) \\
& \cong (i_{l_*} i_l^{\#} \mathcal{M}_0)_K \\
& \cong  (p_r^{-1}\widehat{\mathcal{M} ^{\mu}_n}\uset{p_r^{-1} \widehat{\mathcal{D}^{\mu}_n}} \mathcal{M}_0)_K \text{ (by Proposition \ref{d3piadicpullback}) }\\
& \cong (p_r^{-1}\widehat{\mathcal{M} ^{\mu}_n})_K \uset{(p_r^{-1} \widehat{\mathcal{D}^{\mu}_n})_K} (\mathcal{M}_0)_K \\
& \cong (p_r^{-1}\widehat{\mathcal{M} ^{\mu}_n})_K \uset{(p_r^{-1} \widehat{\mathcal{D}^{\mu}_n})_K} \mathcal{M} \\
& \cong (p_r^{-1}\widehat{\mathcal{M} ^{\mu}_n})_K \uset{p_r^{-1} \widehat{\mathcal{D}^{\mu}_{n,K}}} \mathcal{M}. \qedhere
\end{split}
\end{equation}
\end{proof}

We should remark that the argument above proves that the functor $\hat{i}_{l,K}^{\#}$ is well defined, i.e. it does not depend on the lattice of $\mathcal{M}$.

\subsection{Computation of global sections}

We aim to compute $\Gamma(X, \hat{i}_{l,K}^{\#} \mathcal{M})$ for $\mathcal{M} \in \Coh(\widehat{\mathcal{D}^{\lambda,\mu}_{n,K}},G)$. Recall that $Z= X \times X$ denotes the double flag variety. We start by making some notations.

To simplify the notation, we denote $\mathcal{A}:=\Gamma(X,\widehat{\mathcal{D}^{\mu}_n})$, so that $\mathcal{A}_K=\mathcal{A} \uset{R} K \cong \Gamma(X,\widehat{\mathcal{D}^{\mu}_{n,K}})$. Further, we let $\mathcal{B}:=\Gamma(X,\widehat{\mathcal{D}^{\lambda,\mu}_n})$. We have by combining Proposition \ref{d3huggnKllisotensorhuggnKl} and \cite[Theorem 6.10b)]{Annals} that  $\huggnKlm \cong \mathcal{B}_K \cong \hugnKl \hat{\uset{K}} \mathcal{A}_K$. Given a $\mathcal{B}_K$-module $M$, we may view it as $\mathcal{A}_K$-module via $x.m=(1 \otimes x).m$ for $x \in \mathcal{A}_K$ and $m \in M$.

\begin{proposition}
\label{d3glosectionafpullback}
$\mathcal{M} \in \Coh(\widehat{\mathcal{D}^{\lambda,\mu}_{n,K}},G)$. Then as $\hugnKl$-modules we have an isomorphism:

$$\Gamma(X,\hat{i}_{l,K}^{\#} \mathcal{M}) \cong \Gamma(X, \widehat{\mathcal{M}^{\mu}_{n,K}}) \uset{\mathcal{A}_K} \Gamma(Z, \mathcal{M}).$$
\end{proposition}

\begin{proof}

We have $\Gamma(X,\hat{i}_{l,K}^{\#} \mathcal{M}) \cong \Gamma(Z,i_{l_*}(\hat{i}_{l,K}^{\#} \mathcal{M}))$, so using Corollary \ref{d3affinoidpullback}, it is enough to compute $\Gamma(Z,(p_r^{-1}\widehat{\mathcal{M} ^{\mu}_n})_K \uset{p_r^{-1} \widehat{\mathcal{D}^{\mu}_{n,K}}} \mathcal{M})$.

We have by the first part of \cite[Lemma 6.2]{Sta2} that there is an exact sequence $(\mathcal{D}^{\mu}_{n})^a \to (\mathcal{D}^{\mu}_{n})^b \to \mathcal{M}^{\mu}_{n} \to 0$ for some $a,b \in \mathbb{N}^{*}$. Since $\pi$-adic completion is exact on coherent modules, we obtain an exact sequence $(\widehat{\mathcal{D}^{\mu}_{n}})^a \to (\widehat{\mathcal{D}^{\mu}_{n}})^b \to \widehat{\mathcal{M} ^{\mu}_n} \to 0$, so an exact sequence

$$(\widehat{\mathcal{D}^{\mu}_{n,K}})^a \to (\widehat{\mathcal{D}^{\mu}_{n,K}})^b \to \widehat{\mathcal{M} ^{\mu}_{n,K}} \to 0.$$

Since $p_r^{-1} \widehat{\mathcal{M} ^{\mu}_{n,K}} \cong (p_r^{-1}\widehat{\mathcal{M} ^{\mu}_n})_K$ and tensor product is right exact we obtain an exact sequence:

$$  (p_r^{-1}\widehat{\mathcal{D} ^{\mu}_{n,K}})^a \uset{p_r^{-1} \widehat{\mathcal{D}^{\mu}_{n,K}}} \mathcal{M}   \to  (p_r^{-1}\widehat{\mathcal{D} ^{\mu}_{n,K}})^b \uset{p_r^{-1} \widehat{\mathcal{D}^{\mu}_{n,K}}} \mathcal{M} \to                        (p_r^{-1}\widehat{\mathcal{M} ^{\mu}_n})_K \uset{p_r^{-1} \widehat{\mathcal{D}^{\mu}_{n,K}}} \mathcal{M} \to 0.$$

To simplify the notation, we let $\mathcal{E}=p_r^{-1}\widehat{\mathcal{D} ^{\mu}_{n,K}}$ and $M=\Gamma(Z,\mathcal{M})$. The above short sequence fits into the following commutative diagram: \\
\begin{tikzcd}[column sep=small,scale=1.7em]
&\mathcal{E}^a \uset{\mathcal{E}} \mathcal{M}  \arrow[r] \arrow[d] &\mathcal{E}^b \uset{\mathcal{E}} \mathcal{M} \arrow[r] \arrow[d] &(p_r^{-1}\widehat{\mathcal{M} ^{\mu}_n})_K \uset{\mathcal{E}} \mathcal{M} \arrow[r] \arrow[d] &0 \arrow[d]\\
&\mathcal{M}^a \arrow[r] &\mathcal{M}^b \arrow[r] &(p_r^{-1}\widehat{\mathcal{M} ^{\mu}_n})_K \uset{p_r^{-1} \widehat{\mathcal{D}^{\mu}_{n,K}}} \mathcal{M} \arrow[r] &0.
\end{tikzcd}

Since $\Gamma(Z,-)$ is exact on coherent modules by Theorem \ref{d3affinoidnonequivariantBB}  we obtain a commutative diagram:

\begin{center}
\begin{tikzcd}[column sep=tiny]

&\Gamma(Z,\mathcal{E}^a \uset{\mathcal{E}} \mathcal{M})  \arrow[r] \arrow[d] &\Gamma(Z,\mathcal{E}^b \uset{\mathcal{E}} \mathcal{M}) \arrow[r] \arrow[d] &\Gamma(Z,(p_r^{-1}\widehat{\mathcal{M} ^{\mu}_n})_K \uset{\mathcal{E}} \mathcal{M}) \arrow[r] \arrow[d] &0 \arrow[d]\\

&\Gamma(Z,\mathcal{M}^a) \arrow[r] \arrow[d] &\Gamma(Z,\mathcal{M}^b) \arrow[r] \arrow[d] &\Gamma(Z,(p_r^{-1}\widehat{\mathcal{M} ^{\mu}_n})_K \uset{\mathcal{E}} \mathcal{M}) \arrow[r] \arrow[d] &0 \arrow[d] \\

&M^a \arrow[r] \arrow[d] &M^b \arrow[r] \arrow[d] &\Gamma(Z,(p_r^{-1}\widehat{\mathcal{M} ^{\mu}_n})_K \uset{\mathcal{E}} \mathcal{M}) \arrow[r] \arrow[d]  &0 \arrow[d] \\

&\Gamma(X,\widehat{\mathcal{D}^{\mu}_{n,K}}^a) \uset{\mathcal{A}_K} M \arrow[r] &\Gamma(X,\widehat{\mathcal{D}^{\mu}_{n,K}}^b) \uset{\mathcal{A}_K} M \arrow[r] &\Gamma(X, \widehat{\mathcal{M}^{\mu}_{n,K}}) \uset{\mathcal{A}_K} M \arrow[r] &0.
\end{tikzcd}
\end{center}

By construction, we have that the vertical arrows on the first, second and fourth columns are isomorphisms.  Considering the first and fourth row and applying the Five Lemma, we get the desired isomorphism.
\end{proof}

\subsection{Global sections of \texorpdfstring{ $\widehat{\mathcal{M}^{\lambda}_{n,K}}$} {M \textasciicircum  lambda \_ n,K}}

Recall that $\mathfrak{b}^-=\mathfrak{n}^- \oplus \mathfrak{h}$ denotes the negative Borel subalgebra of $\mathfrak{g}$. Let $T(\mu)_0=R_{\mu} \uset{U(\mathfrak{b}^{-})_n}U(\mathfrak{g})_n$ denote the right $U(\mathfrak{g})_n$-module such that $U(\mathfrak{b}^{-})_n$ acts on $R$ via $\mu$. Further, we let $T(\mu):=T(\mu)_0 \uset{R}K \cong K_{\mu} \uset{U(\mathfrak{b}^{-})_K} U(\mathfrak{g})_K$ and $\widehat{T(\mu)}:= \widehat{T(\mu)_0} \uset{R} K \cong T(\mu) \uset{U(\mathfrak{g})_K} \hugnK$. We should remark that $T(\mu)$ is the same as defined in \cite[Lemma 6.2]{Sta2}.

\begin{proposition}
\label{d3globalsectionMlambdan}
We have $\Gamma(X,\widehat{\mathcal{M}^{\mu}_{n,K}}) \cong \widehat{T(\mu)}$ as right $\hugnK$-modules.

\end{proposition}

\begin{proof}

We have by construction that $\Gamma(X,\mathcal{M}^{\mu}_n) \uset{R} K \cong \Gamma(X_K,\mathcal{M}^{\mu}_{n,K|X_K})$ and by \cite[Lemma 6.2]{Sta2} that $\Gamma(X_K,\mathcal{M}^{\mu}_{n,K|X_K})\cong T(\mu)$. Therefore

$$ \Gamma(X,\mathcal{M}^{\mu}_n) \uset{R} K \cong T(\mu)_0 \uset{R} K \cong T(\mu).$$

By \cite[Proposition 5.15b)]{Annals}, $ \Gamma(X,\mathcal{M}^{\mu}_n)$ is finitely generated as a $U(\mathfrak{g})_n$-module, so $\Gamma(X,\mathcal{M}^{\mu}_n)$ and $T(\mu)_0$ are both lattices for $T(\mu)$. Therefore, they agree modulo bounded $\pi$-torsion, i.e. there exists an exact sequence

$$T(\mu)_0 \to  \Gamma(X,\mathcal{M}^{\mu}_n) \to C \to 0,$$

such that $\pi^mC=0$ for some $m \in \mathbb{N}$ and $C$ is a finitely generated $U(\mathfrak{g})_n$-module. Since, $\pi$-adic completion is exact on finitely generated $U(\mathfrak{g})_n$-modules we get an exact sequence

$$\widehat{T(\mu)_0} \to \widehat{\Gamma(X,\mathcal{M}^{\mu}_n)} \to \widehat{C} \to 0,$$

with $\pi^m \widehat{C}=0$. So by tensoring with $K$ we obtain

\begin{equation*}
\widehat{T(\mu)}=\widehat{T(\mu)_0} \uset{R} K \cong  \widehat{\Gamma(X,\mathcal{M}^{\mu}_n)} \uset{R} K \cong \Gamma(X, \widehat{\mathcal{M}^{\mu}_{n,K}}). \qedhere
\end{equation*}
\end{proof}

We may now prove the main theorem of this section, an affinoid version of \cite[Theorem 6.5]{Sta1}.

\begin{theorem}
\label{d3globalsectionsGequivariantDll}
Let $\mathcal{M} \in \Coh(\widehat{\mathcal{D}^{\lambda,\mu}_{n,K}},G)$. Then as $\hugnKl$-modules we have:

$$\Gamma(X,\hat{i}_{l,K}^{\#} \mathcal{M}) \cong \widehat{T(\mu)} \uset{\hugnKm} \Gamma(Z, \mathcal{M}).$$

\end{theorem}

\begin{proof}
This follows by combining Corollary \ref{d3affinoidequivariantBB}, Proposition \ref{d3globalsectionMlambdan} and \cite[Theorem 6.10b)]{Annals}.
\end{proof}

\section{Affinoid Duflo's theorem}
\label{d3sectionduflothm}

Throughout this section $\lambda:\pi^n \mathfrak{h} \to R$ denotes an $R$-linear dominant weight. Recall that we use $\lambda^*$ to denote the weight $-w_{o}\lambda$, where $w_{o}$ is the longest element of the Weyl group. If $\lambda$ is a dominant weight, so is $\lambda^*$. We consider the functor 

\begin{equation*}
\begin{split}
&\mathscr{F}:\Mod_{\fg}(\huggnKlstl,G) \to \Mod_{\fg}(\hugnKlst,B), \\
&\mathscr{F}(M) := \Gamma(X,\hat{i}_{l,K}^{\#} \Loc^{\lambda} (M)).
\end{split}
\end{equation*}

\begin{proposition}
\label{d3computationofF}
The functor $\mathscr{F}$ is exact and $\mathscr{F}(M) \cong \widehat{T(\lambda)} \uset{\hugnKl} M$ as $\hugnKlst$-modules for  $M \in \Mod_{\fg}(\huggnKlstl,G)$.

\end{proposition}

\begin{proof}
Let $M \in \Mod_{\fg}(\huggnKlstl,G)$ and $\mathcal{M}:=\Loc^{\lambda^*,\lambda} (M)$. We have by Corollary \ref{d3corlocgloiso} that $\Gamma(Z,\mathcal{M}) \cong M$ and by Theorem \ref{d3affinoidequivariantBB} that $\mathcal{M} \in \Coh(\widehat{\mathcal{D}^{\lambda^*,\lambda}_{n,K}},G)$, so the second claim follows from Theorem \ref{d3globalsectionsGequivariantDll}. Consider a short exact sequence in $\Mod_{\fg}(\huggnKlstl,G)$:

$$0 \to N \to M \to P \to 0.$$

We let $\mathcal{N}, \mathcal{M}, \mathcal{P}$ denote the localisation of $N,M$ and $P$ respectively. Further, we denote $\mathcal{L}:=\ker( \mathcal{N} \to \mathcal{M})$. By construction, $\Loc^{\lambda^*,\lambda}$ is right exact, so we obtain an exact sequence:

$$0 \to \mathcal{L} \to \mathcal{N} \to \mathcal{M} \to \mathcal{P} \to 0.$$

Since $\hat{i}_{l,K}^{\#}$ is an equivalence of Abelian categories, it is exact. Furthermore, by Theorem \ref{d3affinoidequivariantBB} the global sections functor is also exact, so we obtain an exact sequence

$$0 \to \Gamma(X,\hat{i}_{l,K}^{\#} \mathcal{L}) \to \Gamma(X,\hat{i}_{l,K}^{\#} \mathcal{N}) \to \Gamma(X,\hat{i}_{l,K}^{\#} \mathcal{M}) \to \Gamma(X,\hat{i}_{l,K}^{\#} \mathcal{P}) \to 0.$$ 

Combining Theorem \ref{d3globalsectionsGequivariantDll} and Corollary \ref{d3corlocgloiso} we obtain an exact sequence

\begin{equation}
\begin{split}
0 &\to \widehat{T(\lambda)} \uset{\hugnKl} \Gamma(Z, \mathcal{L}) \to  \widehat{T(\lambda)} \uset{\hugnKl} N \\
  & \to  \widehat{T(\lambda)} \uset{\hugnKl} M \to  \widehat{T(\lambda)} \uset{\hugnKl} P \to 0.
\end{split}
\end{equation}

The claim follows since $\Gamma(Z, \mathcal{L})=0$ by definition of $\mathcal{L}$ and Corollary \ref{d3corlocgloiso}.
\end{proof}

\begin{lemma}
\label{d3F(M)iszeroMiszero}
Let $M \in \Mod_{\fg}(\huggnKlstl,G)$ and assume $\mathscr{F}(M)=0$. Then $M=0$.
\end{lemma}

\begin{proof}
Let $\mathcal{M}:=\Loc^{\lambda^*,\lambda}(M)$. Then, by assumption, we have that $\Gamma(X,\hat{i}_{l,K}^{\#} \mathcal{M})=0$. By applying Corollary \ref{d3corlocgloiso} and Corollary \ref{d3affinoidzeroglosection}, we obtain $M=0$.
\end{proof}

We now specialise to two sided ideals in $\hugnKl$; recall that a two-sided ideal $I$ can be viewed as a module over $\hugnKlst \otimes \hugnKl$ via $(x \otimes y).i=yi\tau(x)$ for $x \in \hugnKlst,y \in \hugnKl$ and $i \in I$. Further, by Corollary \ref{d3equivarianceoftwosidedideals}, we have $I \in \Mod_{\fg}(\huggnKlstl,G)$, so $\mathscr{F}(I)$ is well-defined. As a corollary, we obtain immediately:

\begin{corollary}
\label{d3Fpreservesstrictinclusionofideals}

Let $I,J$ be two-sided ideals  in $\hugnKl$ such that $I \subseteq J$. Assume that $\mathscr{F}(I) \cong \mathscr{F}(J)$. Then $I=J$.

\end{corollary}

\begin{proof}

Consider the short exact sequence:

$$0 \to I \to J \to J/I \to 0.$$

By Proposition \ref{d3computationofF} the functor $\mathscr{F}$ is exact, so we obtain an exact sequence

$$0 \to \mathscr{F}(I) \to \mathscr{F}(J) \to \mathscr{F}(J/I) \to 0.$$

Using the assumption, we obtain $\mathscr{F}(J/I)=0$. The claim follows by Lemma \ref{d3F(M)iszeroMiszero}.
\end{proof}

\begin{corollary}
\label{d3Fonideals}
Let $I$ be a two-sided ideal in $\hugnKl$. Then as left $\hugnKlst$-modules we have 

      $$\mathscr{F}(I) \cong \widehat{T(\lambda)}I.$$

\end{corollary}

We should remark that $\hugnKlst$ acts on $\widehat{T(\lambda)}I$ via $x.(ti)=t(x.i)=ti\tau(x)$ for $x \in \hugnKlst, t \in \widehat{T(\lambda)}$ and $ i \in I$. Further, the isomorphism is natural in $I$.

\begin{proof}
Consider the following exact sequence:
$$0 \to I \to \hugnKl \to \hugnKl/I \to 0.$$

Applying Proposition \ref{d3computationofF} we obtain a short exact sequence:

$$0 \to \widehat{T(\lambda)} \uset{\hugnKl} I \to \widehat{T(\lambda)}\to \widehat{T(\lambda)} \uset{\hugnKl} \hugnKl/I \to 0.$$

This short exact sequence fits in the following commutative diagram:

\begin{tikzcd}[column sep =small]
&0 \arrow[r] &\widehat{T(\lambda)} \uset{\hugnKl} I \arrow[d] \arrow[r] & \widehat{T(\lambda)}   \arrow[r] \arrow[d] &\widehat{T(\lambda)} \uset{\hugnKl}\hugnKl/I \arrow[r] \arrow[d] &0 \\
&0 \arrow[r] &\widehat{T(\lambda)}  I \arrow[r] & \widehat{T(\lambda)}  \arrow[r] &\widehat{T(\lambda)} \uset{\hugnKl}\hugnKl/I \arrow[r] &0.
\end{tikzcd}

It is easy to see that the first map is a surjection and the second and the third maps are isomorphisms. Furthermore, by the diagram above,  the first map is also injective, so indeed we get $\mathscr{F}(I) \cong \widehat{T(\lambda)} I$.
\end{proof}

Let $\sigma: \mathfrak{g} \to \mathfrak{g}$ denote the Chevalley involution that swaps $\mathfrak{n}^+$ and $\mathfrak{n}^{-}$ and fixes $\mathfrak{h}$ . Then $\sigma$ extends to an anti-automorphism of $U(\mathfrak{g})_K$ that fixes the center by \cite[Exercise 1.10]{Hu1}. Therefore, we obtain that an anti-automorphism $\hat{\sigma}:\hugnKl \to \hugnKl$. Recall that $T(\lambda)_0=R_{\lambda} \uset{U(\mathfrak{b}^{-})_n} U(\mathfrak{g})$, so that
$\widehat{T(\lambda)} \cong K_{\lambda} \uset{\widehat{U(\mathfrak{b}^{-})_{n,K}}} \hugnK$.

\begin{lemma}
\label{d3hatTlambdahatMlambda}
The map

$$\hat{\phi}:\widehat{T(\lambda)} \to \widehat{M(\lambda)}, \qquad \hat{\phi}(k \otimes x)=\hat{\sigma}(x) \otimes k, \qquad k \in K_{\lambda}, x \in \hugnK$$

is a $K$-linear isomorphism of vector spaces satisfying $\hat{\phi}(tu)=\hat{\sigma}(u)\hat{\phi}(t)$ for all $u \in \hugnK$ and $t \in \widehat{T(\lambda)}$. In particular, if $I$ is a two-sided ideal in $\hugnKl$, then $\hat{\phi}(\widehat{T(\lambda)}I)=\hat{\sigma}(I)\widehat{M(\lambda)}$.

\end{lemma}

\begin{proof}
We have by \cite[Lemma 7.6]{Sta2} that the map $\phi:T(\lambda) \to M(\lambda)$, $\phi(k \otimes x)=\sigma(x) \otimes K$ is a $K$-linear isomorphism of vector spaces satisfying $\phi(tu)=\sigma(u)\phi(t)$ for all $t \in T(\lambda)$ and $u \in U(\mathfrak{g})_K$. Therefore, the claims follow from the construction of $\hat{\phi}$ and $\hat{\sigma}$.
\end{proof}

To prove the main theorem, we will need the following corollary:

\begin{corollary}
\label{d3IannihilatesIMl}
Let $I$ be a two-sided ideal in $\hugnKl$. Then:
         $$I=\Ann(\widehat{M(\lambda)}/I\widehat{M(\lambda)}).$$
\end{corollary}

\begin{proof}

Let $J:=\Ann(\widehat{M(\lambda)}/I\widehat{M(\lambda)})$. Since $I(\widehat{M(\lambda)}/I\widehat{M(\lambda)})=0$, we obtain that $I \subseteq J$, so $\hat{\sigma}(I) \subseteq \hat{\sigma}(J)$. We remark that since $\hat{\sigma}$ is an anti-automorphism of $\hugnKl$, $\hat{\sigma}(I)$ and $\hat{\sigma}(J)$ are also two-sided ideals in $\hugnKl$. Since $\mathscr{F}$ is exact, in particular left exact, we obtain $\mathscr{F}(\hat{\sigma}(I)) \subseteq \mathscr{F}(\hat{\sigma}(J))$. Consider the following diagram:

\begin{center}
\begin{tikzcd}
&\mathscr{F}(\hat{\sigma}(I)) \arrow[r] \arrow[d] & \mathscr{F}(\hat{\sigma}(J)) \arrow[d]\\
&\widehat{T(\lambda)}\hat{\sigma}(I) \arrow[r] \arrow[d,"\hat{\phi}"] & \widehat{T(\lambda)} \hat{\sigma}(J) \arrow[d,"\hat{\phi}"]\\
&I\widehat{M(\lambda)} \arrow[r] &J\widehat{M(\lambda)}

\end{tikzcd}
\end{center}

By construction, the bottom diagram commutes and by the definition of $J$, we have $J\widehat{M(\lambda)}=I\widehat{M(\lambda)}$. Using Lemma \ref{d3hatTlambdahatMlambda}, we get $\widehat{T(\lambda)}\hat{\sigma}(I)=\widehat{T(\lambda)}\hat{\sigma}(J)$. Furthermore, since the isomorphism is Lemma 
\ref{d3Fonideals} is natural on ideals, the top diagram is also commutative. Therefore, $\mathscr{F}(\hat{\sigma}(I)) \cong \mathscr{F}(\hat{\sigma}(J))$. The claim follows from Corollary \ref{d3Fpreservesstrictinclusionofideals} since $\hat{\sigma}$ is an anti-automorphism.
\end{proof}

\begin{theorem}
\label{d3affinoidduflotheorem}
Let $I$ be a prime ideal in $\hugnKl$. Then $$I=\Ann(\widehat{L(\mu)}) \text{ for some } \mu:\pi^n \mathfrak{h} \to R.$$

\end{theorem}

\begin{proof}
Since $\widehat{M(\lambda)}/I \widehat{M(\lambda)}$ is a quotient of $\widehat{M(\lambda)}$, we have by Proposition \ref{d3compositionseriesaffinoidsubquotientVerma} that there exists a finite composition series:

$$0=M_0 \subset M_1 \subset \ldots M_l=\widehat{M(\lambda)}/I \widehat{M(\lambda)}.$$

Let $I_i=\Ann(M_i/M_{i-1})$ for $1 \leq i \leq l$. We have

$$I_1I_2 \ldots I_l M_l=I_1 I_2 \ldots I_{l-1}M_{l-1}= \ldots =0,$$

so $I_1I_2 \ldots I_l \subset  \Ann(\widehat{M(\lambda)}/I \widehat{M(\lambda)}) =I$ by Corollary \ref{d3IannihilatesIMl}. Since $I$ is prime, there exists $1 \leq j \leq l$ such that $I_j \subset I$. On the other hand, we have by construction $I \subset I_j$, so $I=I_j=\Ann(M_j/M_{j-1})$. Finally, we have by Proposition \ref{d3compositionseriesaffinoidsubquotientVerma} that $M_j/M_{j-1} \cong \widehat{L(\mu)}$ for some $\mu:\pi^n \mathfrak{h} \to R$.
\end{proof}

As an easy corollary, we obtain a positive answer to \cite[Question A]{Munster}.

\begin{corollary}
\label{d3corollarymunsterduflo}
Every primitive ideal of $\hugnK$ with $K$-rational
infinitesimal central character is the annihilator of a simple affinoid highest weight module.
\end{corollary}

\begin{proof}
Any primitive ideal in $\hugnK$ with $K$-rational infinitesimal central character intersects $Z(\hugnK)$ in a maximal ideal of the form $\ker \chi_{\lambda}$; here we view $\ker \chi_{\lambda}$ as a central character of $\hugnK$ via Theorem \ref{d3centreofaffinoidenvelopingalgebras}. Therefore, classifying these ideals reduces to classifying the ideals in $\hugnKl$ for all $\lambda \in \pi^n\mathfrak{h}^*$.
There is an action of the Weyl group $W$ on the set of weights such that for two weights $\lambda$ and $\mu$, $\hugnKl=\widehat{U(\mathfrak{g})^{\mu}_{n,K}}$ if and only if $\lambda$ and $\mu$ are $W$-conjugate. Further, every $W$-conjugacy class contains at least one dominant weight. The claim follows by Theorem \ref{d3affinoidduflotheorem} since every primitive ideal is prime.
\end{proof}

We should remark that in the case $p$ is a very good prime for $G$, we have by Theorem \ref{d3annihilatoraffinoidverma} that the ideals $\{\hat{I}_{\lambda}=\Ann(\widehat{M(\lambda)})| \lambda \in \pi^n \mathfrak{h}^* \}$ form the set of minimal primitive ideals with $K$-rational central character.


\section{ A controller theorem}
\label{d3sectioncontrollertheorem}

We keep the notations and assumptions from the previous section; we further assume that $\lambda:\pi^n \mathfrak{h} \to R$ is also \emph{regular}. We will also make use of the fact that two-sided ideals in $\hugnK$ that contain $\hugnK \ker {\chi_{\lambda}}$ correspond to two-sided ideals in $\hugnKl$.

\begin{lemma}
\label{d3two-sidedidealsgeneratedbyidealinUgK}

Let $I$ be a two-sided ideal in $U(\mathfrak{g}_K)$. Then $\hugnK I$ is a two-sided ideal in $\hugnK I$ and furthermore, $\hugnK I= \hugnK I \hugnK$.

\end{lemma}

\begin{proof}

Clearly, it is enough to prove that $I \hugnK \subset \hugnK I.$

Viewing $\hugnK$ as a left $\hugnK$-module via left multiplication, we have that $\hugnK I$ is a $\hugnK$-submodule. Since the topology $\hugnK$ is complete, we have by \cite[I.5.5]{LVO} that $\hugnK I$ is a closed subset.

Let $i \in I$ and $x \in \hugnK$. Recall that assuming that $u_1,u_2 \ldots u_d$ is a free $R$-basis for $\mathfrak{g}$, we may write as $x=\sum_{\alpha \in \mathbb{N}^d} c_{\alpha}u^{\alpha}$, with $||p^{-n} c_{\alpha}|| \to 0$ as $|\alpha| \to \infty$. For $k \in \mathbb{N}$, let $x_k=\sum_{\alpha \in \mathbb{N}^d,|\alpha| \leq k} c_{\alpha}u^{\alpha}$. Since $I$ is a two-sided ideal, we obtain that $ix_k \in I \subset \hugnK I$.

Finally, we have $ix=i \lim_{k \to \infty} x_k=\lim_{k \to \infty} ix_k$. Since $\hugnK I$ is closed, we obtain $ix \in \hugnK I$ finishing the proof.
\end{proof}

\begin{proposition}
\label{d3hardcontrollerproposition}
Let $I$ be a two-sided ideal in $\hugnK$ such that $\hugnK \ker {\chi_{\lambda}} \subset I$. Then there exists a two-sided ideal $J$ in $U(\mathfrak{g}_K)$ such that $I=\hugnK J$.

\end{proposition}

\begin{proof}

Let $\hat{M}=I\widehat{M(\lambda)}.$ We have by Theorem \ref{d3affinoidclassicvermacorrespondenceprop} that there exists $M$ a submodule of $M(\lambda)$ such that $\hat{M}=\hugnK.M$. Further, we have by \cite[Theorem 4.3]{BG} that there exists $J$ a two-sided ideal $U(\mathfrak{g}_K)$ such that $JM(\lambda)=M$. By applying Lemma \ref{d3two-sidedidealsgeneratedbyidealinUgK} we obtain:

\begin{equation}
\begin{split}
(\hugnK J).\widehat{M(\lambda)}&=(\hugnK J) \hugnK. M(\lambda)\\
&=(\hugnK J). M(\lambda)=\hugnK. M=\hat{M},
\end{split}
\end{equation}
so $(\hugnK J)\widehat{M(\lambda)}=I\widehat{M(\lambda)}$. Let $J'=\hugnK J + I$, so that $J' \widehat{M(\lambda)}=I \widehat{M(\lambda)}$ and $I \subset J'$. We have by Lemma \ref{d3two-sidedidealsgeneratedbyidealinUgK} that $J'$ is also a two-sided ideal. Further, by combining Corollary \ref{d3Fpreservesstrictinclusionofideals}, Corollary \ref{d3Fonideals} and Lemma \ref{d3hatTlambdahatMlambda} we get $J'=I$, so $\hugnK J \subset I$. Applying the same strategy again, we obtain $\hugnK J=I$.  
\end{proof}

To finish the proof, we need one more lemma. This is probably well-known among the experts, but we have not been able to locate a reference.

\begin{lemma}
\label{d3easycontrollerlemma}
Let $S \subset T$ two rings and let $I$ be a left ideal of $T$ generated by $X \subset S$. Then $I=T(I \cap S)$.
\end{lemma}

\begin{proof}

Let $J=I \cap S$. Obviously, we have $TJ \subset I$. On the other hand, we have $X \subset S$, $X \subset I$, so $X \subset J$. Therefore, $I=T.X \subset TJ$. The claim follows.
\end{proof}

\begin{theorem}
\label{d3controllertheorem}
Let $\lambda:\pi^n\mathfrak{h} \to R$ be a $R$-linear dominant regular weight. Let $I$ be a two-sided ideal in $\hugnK$ with $\chi_{\lambda}$-central character. Then:

$$I=\hugnK(I \cap U(\mathfrak{g}_K)).$$
\end{theorem}

\begin{proof}

This follows immediately from Proposition \ref{d3hardcontrollerproposition} and Lemma \ref{d3easycontrollerlemma}.
\end{proof}

As a corollary, we obtain immediately:

\begin{corollary}
\label{d3twosidedidealscorrespondecebijection}
Let $\lambda:\pi^n\mathfrak{h} \to R$ be a $R$-linear dominant regular weight. The maps:

\begin{equation}
\begin{split}
I &\mapsto I \cap U(\mathfrak{g}_K) \\
J & \mapsto \hugnK J
\end{split}
\end{equation}
induce inverse bijections between the set of two sided ideals in $\hugnK$ with $\chi_{\lambda}$ central character and the set of two sided ideals in $\hugnK$ with $\chi_{\lambda}$ central character.
\end{corollary}

We may also prove which ideal controls the annihilator of the simple affinoid module $\widehat{L(\mu)}$.

\begin{proposition}
\label{d3controllerfinalproposition}

Let $\mu:\pi^n \mathfrak{h} \to R$ and assume that $\mu$ is $W$-linked to $\lambda$. Then:

$$\Ann(\widehat{L(\mu)})=\hugnK \Ann L(\mu).$$

\end{proposition}

\begin{proof}
Let $\hat{P}:= \Ann(\widehat{L(\mu)})$, $P=\Ann(L(\mu))$ and $J=\hat{P} \cap U(\mathfrak{g}_K)$. Then, we have by construction $P \subset J$.

We claim that $\hugnK P$ contains the annihilator of $\widehat{L(\mu)}$. We have by the proof of Lemma y\ref{d3two-sidedidealsgeneratedbyidealinUgK} that for all $p \in P$ and $x \in \hugnK$, there exist $q \in P$ and $y \in \hugnK$ such that $px=yq$. Therefore, we have for all $z \in \hugnK$ and $x \otimes l \in \hugnK \uset{U(\mathfrak{g}_K)} L$ that

$$zp.(x \otimes l)=z(px \otimes l)=zy(q \otimes l)=zy(1 \otimes q.l)=0,$$

so the claim is proven. Therefore, we obtain using Corollary \ref{d3twosidedidealscorrespondecebijection} that $P=J$ and so, $\hat{P}=\hugnK P$.
\end{proof}

\subsection{Characterisation of all two-sided ideals}

We may now use an affinoid version of Quillen's Lemma to prove that a large class of two-sided ideals in $\hugnK$ is controlled by two-sided ideals in the classical enveloping algebra $U(\mathfrak{g})_K$.

We are now able to characterise all the primitive ideals in $\hugnK$ for any $n \in \mathbb{N}$.

\begin{theorem}
\label{d3alltheprimitiveideals}
Let $I$ be a primitive ideal in $\hugnK$. Then there exists a finite extension $L/K$ and a primitive ideal $J \in \hugnK \uset{K} L$ with $L$-rational central character such that:

$$I=J \cap \hugnK.$$

Further, this ideal $J$ is of the form $\Ann(\widehat{L(\lambda)})$ for some suitable $\lambda$.

\end{theorem}

\begin{proof}

We have by Theorem \ref{d3centreofaffinoidenvelopingalgebras} that $Z(\hugnK)$ is isomorphic with a Tate algebra. Further, we have by \cite[Theorem 6.4.6]{StaPhd} that $I$ has some central character. Therefore, there exists $L/K$ finite extension such that $I \uset{K} L \subset \hugnK \uset{R} K$ has $L$-rational central character. 

Let $e$ be the ramification index of $L/K$, $\mathcal{O}_L$ the ring of integers of $L$, $\pi'$ the uniformiser of $\mathcal{O}_L$ and $\mathfrak{g}'=\mathfrak{g} \uset{R} \mathcal{O}_L$. We have by \cite[Lemma 3.9 c)]{Annals} that $\hugnK \otimes L \cong \hugnL$. Finally, we have by \cite[Theorem 10.2.9]{ConRob} that there exists a prime ideal $J \in \hugnL$ such that $I=J \cap \hugnK$. Thus $J$ contains $I \uset{K} L$, so in particular it has an $L$-rational central character. The claim follows from Theorem \ref{d3affinoidduflotheorem}.
\end{proof}

\begin{theorem}
\label{d3finaltheorem}
Let $I$ be a two-sided ideal in $\hugnK$ and assume it has a central character $\chi:Z(\hugnK) \to \overline{K}$ generated by a dominant regular weight. Then $I$ is controlled by a two-sided ideal in $U(\mathfrak{g})_K$.
\end{theorem}

\begin{proof}

We have by the proof of Theorem \ref{d3alltheprimitiveideals} that there exists a finite extension $L/K$ such that $I \uset{K} L$ has a $L$-rational central character. By passing to the Galois closure, if necessary, we may assume that $L/K$ is a Galois extension; let $\mathcal{G}=\Gal(L/K)$ denote the corresponding Galois group. Again, let $e$ be the ramification index of $L/K$,  $\mathcal{O}_L$ the ring of integers of $L$, $\pi'$ the uniformiser of $\mathcal{O}_L$ and $\mathfrak{g}'=\mathfrak{g} \uset{R} \mathcal{O}_L$. As before, it follows from \cite[Lemma 3.9 c)]{Annals} that $\hugnK \uset{K} L \cong \hugnL$.

Since $I$ is a two-sided ideal in $\hugnK$ and $L/K$ is finite, we get that $I \uset{K} L$ is a two-sided ideal in $\hugnK \uset{K} L \cong \hugnL$. Therefore, we have by Proposition \ref{d3hardcontrollerproposition} that there exists $J$, a two-sided ideal in $U(\mathfrak{g'})_L$, such that $I \uset{K} L=\hugnL J$. 

By construction, $I \uset{K} L$ is closed under the natural action of $\mathcal{G}$, therefore $J$ is also closed under the corresponding $\mathcal{G}$-action on $U(\mathfrak{g'})_L$. We have 

\begin{equation}
U(\mathfrak{g'})_L \cong U(\mathfrak{g} \uset{R} \mathcal{O}_L) \uset{\mathcal{O}_L} L \cong U(\mathfrak{g}) \uset{R} \mathcal{O}_L \uset{ \mathcal{O}_L} L \cong U(\mathfrak{g}) \uset{R} L \cong U(\mathfrak{g})_K \uset{K} L.
\end{equation}

Since $L/K$ is a finite Galois extension, we have by Galois descent and the same arguments as in \cite[Theorem 4.2]{ConradGal} that $J=J_0 U(\mathfrak{g'})_L$ for some two-sided ideal $J_0$ in $U(\mathfrak{g})_K$. We have

\begin{equation}
I \hugnL= I \uset{K} L= J \hugnL = (J_0 U(\mathfrak{g'})_L) \hugnL=J_0 \hugnL=(J_0 \hugnK) \uset{K} L.    
\end{equation}

Therefore, by taking $\mathcal{G}$-invariants and taking into account that $I$ and $J_0 \hugnK$ are fixed by the $\mathcal{G}$-action, we obtain

\begin{equation}
I=(I \uset{K} L)^{\mathcal{G}}=((J_0 \hugnK) \uset{K} L)^{\mathcal{G}}=J_0 \hugnK.    
\end{equation}

In conclusion, $I$ is controlled by $J_0$ and the claim is proven.
\end{proof}

\bibliography{duflo3}

\begin{thebibliography}{10}

\bibitem{Eqdcap}
Konstantin Ardakov.
\newblock Equivariant {D}-modules on rigid analytic spaces.
\newblock \url{http://people.maths.ox.ac.uk/ardakov/EqDcap.pdf}, 2018.
\newblock submitted.

\bibitem{Annals}
Konstantin Ardakov and Simon Wadsley.
\newblock On irreducible representation of compact $p$-adic analytic groups.
\newblock {\em Annals of Mathematics}, 178:453--557, 2013.

\bibitem{Munster}
Konstantin Ardakov and Simon Wadsley.
\newblock Verma modules for {Iwasawa} algebras are faithful.
\newblock {\em Münster Journal of Mathematics}, 7(1):5--26, 2014.

\bibitem{BB}
Alexandre Beilinson and Joseph Bernstein.
\newblock Localisation de $\mathfrak{g}$-modules.
\newblock {\em C. R. Acad. Sci., Paris}, 292:15--18, 1981.

\bibitem{BeGi}
Alexandre Beilinson and Victor Ginzburg.
\newblock Wall-crossing functors and {$\mathcal{D}$}-modules.
\newblock {\em Representation Theory}, 3, 1999.

\bibitem{MvdB}
{Michel van der} Bergh.
\newblock Some generalities on {$G$}-equivariant quasi-coherent
  {$\mathcal{O}_X$} and {$\mathcal{D}_X$}-modules.
\newblock \url{http://hardy.uhasselt.be/personal/vdbergh/Publications/Geq.pdf}.

\bibitem{BGG}
Joseph Bernstein, Izrail Gel'fand, and Sergei Gel'fand.
\newblock Category of $\mathfrak{g}$-modules.
\newblock {\em Functional Analysis and its applications}, 10(2):1--8, 1976.
\newblock English translation.

\bibitem{BG}
Joseph Bernstein and Sergei Gel'fand.
\newblock Tensor products of finite and infinite dimensional representations of
  semisimple {Lie} algebras.
\newblock {\em Compositio Mathematica}, 41:245--285, 1980.

\bibitem{Berthelot}
Pierre Berthelot.
\newblock {$\mathcal{D}$-modules arithm\'etiques I. Op\'erateurs
  diff\'erentiels de niveau fini}.
\newblock {\em Annales scientifiques de l'École Normale Supérieure},
  29(2):185--272, 1996.

\bibitem{Bode}
Andreas Bode.
\newblock Completed tensor products and a global approach to p-adic analytic
  differential operators.
\newblock {\em Mathematical Proceedings of the Cambridge Philosophical
  Society}, 167(2):1--28, 2019.

\bibitem{BoBr}
Walter Borho and Jean-Luc Brylinski.
\newblock Differential operators on homogeneous spaces. {III}.
\newblock {\em Inventiones Mathematicae}, 80(1), 1985.

\bibitem{Bos}
Siegfried Bosch.
\newblock {\em Lectures on Formal and Rigid Geometry}.
\newblock Springer, 2015.

\bibitem{BGR}
Siegfried Bosch, Ulrich G{\"u}ntzer, and Reinhold Remmert.
\newblock {\em Non-archimidean Analysis}.
\newblock Springer-Verlag, 1984.

\bibitem{Che}
Gaetan Chenevier.
\newblock Affinoid algebras and {Tate}’s $p$-adic analytic spaces : a brief
  survey.
\newblock
  \url{http://gaetan.chenevier.perso.math.cnrs.fr/coursIHP/chenevier_lecture5.pdf}.

\bibitem{ConradGal}
Keith Conrad.
\newblock Galois descent.
\newblock
  \url{https://kconrad.math.uconn.edu/blurbs/galoistheory/galoisdescent.pdf}.

\bibitem{FeLa}
Christian Tobias~Féaux de~Lacroix.
\newblock {Einige Resultate über die topologischen Darstellungen $p$-adischer
  Liegruppen auf unendlich dimensionalen Vektorräumen über einem $p$-adischen
  Körper }.
\newblock {\em Schriftenreihe Math. Inst. Univ. Münster 3. Ser.}, 1999.

\bibitem{Dix}
Jacques Dixmier.
\newblock {\em Enveloping Algebras}.
\newblock North-Holland Publishing Company, 1977.

\bibitem{Du}
Michel Duflo.
\newblock Sur la classification des ideaux primitifs dans l’alg`ebre
  enveloppante d’une alg`ebre de {Lie} semisimple.
\newblock {\em Annals of Mathematics}, 105:107--120, 1977.

\bibitem{Gin}
Victor Ginzburg.
\newblock On primitive ideals.
\newblock {\em Selecta Mathematica}, 9(3):379--407, 2003.

\bibitem{HTT}
Ryoshi Hotta, Kiyoshi Takeuchi, and Toshiyuki Tanisaki.
\newblock {\em $D$-modules, Perverse Sheaves and Representation Theory}, volume
  236 of {\em Progress in Mathematics}.
\newblock Birkh{\"a}user Boston Inc, 2000.

\bibitem{LVO}
Li~Huishi and Freddy van Oystaeyen.
\newblock {\em Zariskian Filtrations}.
\newblock Kluwer Academic Publishers, 1996.

\bibitem{Hu1}
James Humphreys.
\newblock {\em Representations of Semisimple Lie Algebras in the BGG category
  $\mathcal{O}$}, volume~94 of {\em Graduate Studies in Mathematics}.
\newblock Americal Mathematical Society, 2008.

\bibitem{Jan1}
Jens~Carsten Jantzen.
\newblock {\em Representations of Algebraic Groups}, volume 107 of {\em
  Mathematical Surveys and Monographs}.
\newblock American Mathematical Society, 2 edition, 2003.

\bibitem{Jos}
Anthony Joseph.
\newblock Dixmier's problem for {Verma} and principal series submodules.
\newblock {\em Journal of the London Mathematical Society}, 20(2):193--204,
  1979.

\bibitem{ConRob}
John McConnell and Chris Robson.
\newblock {\em Non-commutative Noetherian Rings}, volume~30 of {\em Graduate
  studies in Mathematics}.
\newblock American Mathematics Society, 2000.

\bibitem{Mil}
Dragan Milicic.
\newblock Lectures on derived categories.
\newblock \url{https://www.math.utah.edu/~milicic/Eprints/dercat.pdf}.

\bibitem{Mont}
Susan Montgomery.
\newblock {\em Hopf Algebras and Their Actions on Rings}, volume~82 of {\em
  Conference Series in Mathematics}.
\newblock American Mathematical Society, 1993.

\bibitem{Nas}
Constantin Nastasecu and Freddy van Oysten.
\newblock {\em Graded and Filtered Rings and Modules}, volume 758 of {\em
  Lecture Notes in Mathematics}.
\newblock Springer-Verlag, 1979.

\bibitem{Popescu}
Nicolae Popescu.
\newblock {\em Abelian categories with applications to rings and modules}.
\newblock L.M.S. Monographs. Academic Press Inc, 1973.

\bibitem{Sch}
Peter Schneider.
\newblock {\em Nonarchimidean Functional Analysis}.
\newblock Springer Monographs in Mathematics. Springer-Verlag, 2001.

\bibitem{StackProject}
The {Stacks Project Authors}.
\newblock {Stacks Project}.
\newblock \url{https://stacks.math.columbia.edu/}, 2020.

\bibitem{Sta2}
Ioan Stanciu.
\newblock { A geometric proof of Duflo's theorem}.
\newblock submitted.

\bibitem{Sta1}
Ioan Stanciu.
\newblock {Towards affinoid Duflo's theorem I: Lie algebroids and twisted
  differential operators}.
\newblock To appear in Mathematical Proceedings of the Cambridge Philosophical
  Society.

\bibitem{StaPhd}
Ioan Stanciu.
\newblock Affinoid enveloping algebras and their representations, 2020.
\newblock D.Phil Thesis.

\bibitem{Wei}
Charles Weibel.
\newblock {\em The $K$-book: an Introduction to Algebraic $K$-theory}, volume
  145 of {\em Graduate Studies in Mathematics}.
\newblock American Mathematical Society, 2013.

\end{thebibliography}
\bibliographystyle{plain}

\end{document}